\documentclass[11pt]{article}
\usepackage{geometry}
\usepackage{fullpage}
\usepackage{amsmath}
\usepackage{amsthm}
\usepackage{amsfonts}
\usepackage{amssymb}
\usepackage{graphicx, epsfig, psfrag, color}
\usepackage{epstopdf}
\usepackage{algorithm}
\usepackage{multicol}
\usepackage{subfigure}
\newtheorem{Theorem}{Theorem}
\newtheorem{Lemma}{Lemma}

\newtheorem{Corollary}{Corollary}
\newtheorem{Assume}{Assumption}
\newtheorem{Proposition}{Proposition}
\newtheorem{Definition}{Definition}

\theoremstyle{break}

\newcommand{\Real}{\mathbb{R}}
\newcommand{\Graph}{\mathcal{G}}
\newcommand{\Graphsans}[1]{\Graph \backslash \{#1\}}

\newcommand{\JiA}[3] {J_{#1,#2}{(#3)}}

\newcommand{\argmax}{\text{argmax}}

\newcommand{\comment}[1]{}
\newcommand{\bound}{\mathcal{C}}
\newcommand{\eqdef}{\overset{\Delta}{=}}

\newcommand{\pr}{\mathbb{P}}

\newcommand{\gH}{\mathcal{H}}
\newcommand{\I}{\mathcal{I}}

\newcommand{\ignore}[1]{\relax}

\newcommand{\E}{\mathbb{E}}
\newcommand{\sectionline}{\\ \par \noindent\hfil\rule[.20\baselineskip]{\textwidth}{.7pt}\hfil \\ \par}

\newcommand{\remark}{{\bf Remark : }}
\newcommand{\vvert}{\vspace{.1in}}

\begin{document}


\title{Correlation Decay in Random Decision Networks
\footnote{Preliminary version of this paper will appear in Proceedings of
ACM-SIAM Symposium on Discrete Algorithms, 2010,
Austin, TX }}
\author{
{\sf David Gamarnik }
\thanks{Operations Research Center, LIDS, and Sloan School of Management, MIT,
Cambridge, MA,  02139, e-mail:
{\tt gamarnik@mit.edu}}
\thanks{Research supported by the NSF grant CMMI-0726733}
\and
{\sf David A. Goldberg}
\thanks{Operations Research Center, MIT, Cambridge, MA, 02139, e-mail: {\tt
dag3141@mit.edu}}
\and
{\sf Theophane Weber}
\thanks{Operations Research Center and LIDS, MIT, Cambridge, MA, 02139, e-mail: {\tt
theo\_w@mit.edu}}
}
\maketitle
\begin{abstract}

We consider a decision network on an undirected graph in which each node corresponds to a decision variable, and each node and edge of the graph is associated with a reward function whose value depends only on the variables of the corresponding nodes. The goal is to construct a decision vector which maximizes the total reward.
This decision problem encompasses a variety of models, including maximum-likelihood inference in graphical models (Markov Random Fields),
combinatorial optimization on graphs,  economic team theory and statistical physics.
The network is endowed with a probabilistic structure in which costs are sampled from a distribution.
Our aim is to identify sufficient conditions on the network structure and cost distributions to guarantee average-case polynomiality of the
underlying optimization problem. Additionally, we wish to characterize the efficiency of a decentralized solution generated on the basis of local information.

We construct a new decentralized algorithm called \emph{Cavity Expansion} and establish its theoretical performance for a variety of graph models and reward function distributions. Specifically, for certain classes of models we prove that our algorithm is able to find near optimal solutions with high probability
in a decentralized way. The success of the algorithm is based on the network exhibiting a certain correlation decay (long-range independence) property
and we prove that this property is indeed exhibited by the models of interest.
Our results have the following surprising implications in the area of average case complexity of algorithms. Finding the largest independent (stable) set of a graph is a well known NP-hard optimization problem for which no polynomial time approximation scheme is possible even for graphs with largest connectivity equal to three, unless P=NP. Yet we show that the closely related maximum weighted independent set problem for the same class of graphs admits a PTAS when the weights are independent
identically distributed with the exponential distribution. Namely, randomization of the reward function turns an NP-hard problem into a tractable one.
\end{abstract}

\section{Introduction and literature review}
We consider a team of agents working in a networked structure $(V,E)$, where $V$ is a set of agents, and $E$ the set of edges of the network, each edge indicating potential local interactions between agents. Each agent $v$ has to make a decision $x_v$ from a finite set, and the team incurs a total reward $F(x)=\sum \Phi_v(x_v)+\sum_{u,v \in E} \Phi_{u,v}(x_u,x_v)$. The goal of each agent is to choose its decision $x_v$ so that the total reward $F$ is maximized. This model subsumes many models in a variety of fields including economic team theory, statistical inference, combinatorial optimization on graphs and statistical physics.

As an example, common models in the area of statistical inference are graphical models, Bayesian networks, and Markov Random Fields (MRF) (see~\cite{wainwright2008gme} for an overview of inference techniques for graphical models, and~\cite{mezard:ipa,hartmann2005ptc} for a comprehensive study of the relations between statistical physics, statistical inference, and combinatorial optimization). One of the key objects in such a model is the state which achieves the mode of the density, namely, the state which maximizes the a priori likelihood. The problem of finding such a state can be cast as a problem defined above.

In the economic team theory (see~\cite{marschak1955ett,radner1962tdp,marschak1972ett}), an interesting question was raised in~\cite{rusmevichientong2001abp}: what is the cost of decentralization in a chain of agents. In other words, if we assume that each node only receives local information on the network topology and costs, what kind of performance can the team attain? Cast in our framework, this the problem of finding the maximum of $F(x)$ by means of local (decentralized) algorithms.

Combinatorial optimization problems typically involve the task of finding a solution which minimizes or maximizes some objective function subject to various constraints supported by the underlying graph. Examples include the problem of finding a largest independent set, minimum and maximum cut problems,
max-KSAT problems, etc. Finding an optimal solution in many such problems is a special case of the problem of finding $\max_x F(x)$ described above.

Finally, a key object in statistical physics models is the so-called ground state -- a state which achieves the minimum possible energy. Again, finding such an object reduces to solving the problem described above, namely solving the problem $\max_x F(x)$ ($\min_x -F(x)$ to be more precise).

The combinatorial optimization nature of the decision problem $\max_x F(x)$ implies that the problem of finding $x^*=\argmax_x F(x)$
is generally NP-hard, even for the special case when the decision space for each agent consists only of two elements. This motivates a search for approximate methods which find solutions that theoretically or empirically achieve good proximity to optimality.
Such methods usually differ from field to field. In combinatorial optimization the focus has been on developing methods which achieve some provably guaranteed approximation level using a variety of approaches, including linear programming, semi-definite relaxations and purely combinatorial methods~\cite{Hochbaum}.
In the area of graphical models, researchers have been developing new families of distributed inference algorithms. One of the most studied techniques is the Belief Propagation (BP) algorithm ~\cite{lauritzen1996gm,JordanGraphicalModels,YedidiaFreemanWeiss}. Since the algorithm proposed in the present paper
bears some similarity and is motivated by BP algorithm, we provide below a brief summary of known theoretical facts about BP.

The BP algorithm is known to find an optimal solution $x^*$ when the underlying graph is a tree, but may fail to converge, let alone produce an optimal (or correct) solution when the underlying graph contains cycles. Despite this fact, it often has excellent empirical performance. Also, in some cases, BP can be proven to produce an optimal solution, even when the underlying graph contains cycles. In a framework similar to ours, Moallemi and Van Roy~\cite{moallemi2006cms} show that BP converges and produces an optimal solution when  the action space is continuous and the cost functions $\Phi_{u,v}$ and $\Phi_u$ are quadratic and convex. Some generalization to generally convex functions are obtained in~\cite{moallemi2007cms}. Other cases where BP produces optimal solutions include Maximum Weighted Bipartite Matching \cite{sanghavi2007elr,bayati2008ecm,bayati2008mpm} (for matchings), Maximum Weighted Independent Sets problems where the LP relaxation is tight (~\cite{sanghavi2008mpm}), network flow problems~\cite{GamShahWei}, and more generally, optimization problems defined on totally unimodular constraint matrices~\cite{chertkov2008ebp}.

In this paper, we propose a new message-passing like algorithm for the problem of finding $x^*=\argmax\: F(x)$, which we call the \emph{Cavity Expansion} (CE) algorithm, and obtain sufficient conditions for the asymptotic optimality of our algorithm based on the so-called correlation decay property. Our algorithm draws upon several recent ideas. On the one hand, we rely on a technique used recently for constructing approximate counting algorithms. Specifically, Bandyopadhyay and Gamarnik~\cite{BandyopadhyayGamarnikCounting}, and Weitz~\cite{weitzCounting} 
proposed approximate counting algorithms which are based on local (in the graph-theoretic sense) computation. 
Provided that the model exhibits a form of correlation decay these algorithms are approximate counting algorithms.
The approach was later extended in Gamarnik and Katz~\cite{gamarnik2007cda},\cite{gamarnik2007daa}, Bayati et al \cite{bayati2007sda}, Jung and Shah~\cite{jung2006ibp}. The present work develops a similar approach but for the \emph{optimization} problems. The description of the CE algorithm begins by introducing a notion of a \emph{cavity} $B_v(x)$ for each node/decision pair $(v,x)$ (the notion of cavity was heavily used recently in the statistical physics literature~\cite{mezard2003cmz,MezardIndSets2004}). It is also called \emph{bonus} in the relevant papers~\cite{Aldous:assignment00},\cite{AldousSteele:survey},\cite{gamarnikMaxWeightIndSet}.
$B_v(x)$ is defined as the difference between the optimal reward for the entire network when the action in $v$ is $x$ versus the optimal reward when the action in the same node is $0$ (any other base action can be taken instead of $0$). It is easily shown that knowing $B_v(x)$ is equivalent to solving the original decision problem. We  obtain a recursion expressing the cavity $B_v(x)$ in terms of cavities of the neighbors of $v$ in suitably modified sub-networks of the underlying network. The algorithm then proceeds by expanding this recursion in the breadth-first search manner for some designed number of steps $t$, thus constructing an associated computation tree with depth $t$. At the initialization point the cavity values are assigned some default value. Then the approximation value $\hat B_v(x)$ is computed using this computation tree. If this computation was conducted for $t$ equalling roughly the length $L$ of the longest self-avoiding path of the graph, it would result in exact computation of the cavity values $B_v(x)$. Yet the computation effort associated with this scheme is exponential in  $L$, which itself often grows linearly with the size of the graph.

The key insight of our work is that in many cases, the dependence of the cavity $B_v(x)$ on cavities associated with other nodes in the computation tree dies out exponentially fast as a function of the distance between the nodes. This phenomenon is generally called \emph{correlation decay}. In earlier work~\cite{aldous1992ara,Aldous:assignment00,AldousSteele:survey,gamarnikMaxWeightIndSet,gamarnik2008rga}, it is shown that some optimization problems on locally tree-like graphs with random costs are tractable as they exhibit the correlation decay property. This is precisely our approach: we show that if we compute $B_v(x)$ based on the computation tree with only constant depth $t$, the resulting error $\hat B_v(x)-B_v(x)$ is exponentially small in $r$. By taking $r=O(\log(1/\epsilon))$ for any target accuracy $\epsilon$, this approach leads to an $\epsilon$-approximation scheme for computing the optimal reward $\max_x F(x)$. Thus, the main associated technical goal is establishing the correlation decay property for the associated computation tree.

We indeed establish that the correlation decay property holds for several classes of decision networks associated with random reward functions $\Phi=(\Phi_v,\Phi_{v,u})$. Specifically, we give concrete results for the cases of uniform and Gaussian distributed functions for unconstrained optimization in networks with bounded connectivity (graph degree) $\Delta$. We also consider exponentially distributed (with parameter $1$) weights for the Maximum Weighted Independent Set problem. In this setting, the combination of CE (a message passing style algorithm) and a randomized setting has a particularly interesting implication for the theory of average case analysis of combinatorial optimization. Unlike some other NP-hard problems, finding the MWIS of a graph does not admit a constant factor approximation algorithm for general graphs: Hastad~\cite{hastad1996cha} showed that for every $0<\delta<1$ no $n^{1-\delta}$ approximation algorithm can exist for this problem unless $P=NP$,  where $n$ is the number of nodes. Even for the class of graphs with degree at most $3$, no factor $1.0071$ approximation algorithm can exist, under the same complexity-theoretic assumption, see Berman and Karpinski~\cite{BermanKarpinski}. In contrast, we show when $\Delta\le 3$ and the node weights are independently generated with a parameter $1$ exponential distribution, the problem of finding the maximum weighted independent set admits a PTAS. Thus, surprisingly, introducing random weights translates a combinatorially intractable problem into a tractable one. We further extend these results to the case $\Delta>3$, but for different node weight distributions.

The rest of the paper is organized as follows. In section \ref{section:initial notations}, we describe the general model and notations. In section \ref{sec:MainResults}, we present our main results. In section~\ref{sec:Cav}, we derive the cavity recursion, an exact recursion for computing the cavity of a node in a decision network, and from it develop the Cavity Expansion algorithm. In section~\ref{sec:corrdecay}, we prove that the correlation decay property implies optimality of the cavity recursion and local optimality of the solution. The rest of the paper is devoted to identifying sufficient conditions for correlation decay (and hence, optimality of the CE algorithm): in section~\ref{sec:coupling}, we show how a coupling argument can be used to prove the correlation decay property for the case of uniform and Gaussian weight distributions, and in section~\ref{section:MWIS}, we establish the correlation decay property for the MWIS problem using a different argument based on monotonocity. Concluding thoughts are in section~\ref{section:ccl}.

\section{Model description and notations}\label{section:initial notations}

Consider a decision network $\Graph=(V,E,\Phi,\chi)$. Here $(V,E)$ is an undirected simple graph in which each node $u \in V$ represents an agent, and edges $e \in E$ represent a possible interaction between two agents. Each agent makes a decision $x_u\in \chi\triangleq\{0,1,\ldots,T-1\}$. For every $v\in V$, a function $\Phi_v:\chi\rightarrow \Real$ is given. Also for every edge $e=(u,v)$ a function $\Phi_e:\chi^2\rightarrow \Real\cup \{-\infty\}$ is given. 
The inclusion of $-\infty$ into the range of $\Phi_e$ is needed in order to model the ``hard constraints" in the MWIS problem - prohibiting two ends of an edge to belong to an independent set. 
Functions $\Phi_v$ and $\Phi_e$ will be called \emph{potential functions} and \emph{interaction functions} respectively. Let $\Phi=((\Phi_v)_{v\in V}, (\Phi_e)_{e\in E})$. A vector $\bold{x}=(x_1,x_2,\ldots,x_{|V|})$ of actions is called a solution for the decision network. The value of solution $\bold{x}$ is defined to be $F_\Graph(\bold{x})= \sum_{(u,v)\in E} \Phi_{u,v}(x_u,x_v) + \sum_v \Phi_v(x_v)$. The quantity $J_\Graph\eqdef\max_{\bold{x}} F_\Graph(\bold{x})$ is called the (optimal) value  of the network $\Graph$. A decision $\bold{x}$ is optimal if $F_\Graph(\bold{x})=J_\Graph$.

In a Markov Random Field (MRF), a set of random variables $\bold X=(X_1,\ldots,X_n)$ is assigned a probability $\pr(\bold X=\bold x)$ proportional to $\exp(F_\Graph(\bold x))$. In this context, the quantity $F_\Graph(\bold x)$ can be considered as the log-likelihood of assignment $\bold x$, and maximizing it corresponds to finding a maximum a posterior assignment of the MRF defined by $F_\Graph$.

The main focus of this paper will be on the case where $\Phi_v(x),\Phi_e(x,y)$ are random variables (however, the actual realizations of the random variables are observed by the agents, and their decisions depend on the values  $\Phi_v(x)$ and $\Phi_e(x,y)$). While we will usually assume independence of these random variables when $v$ and $e$ vary, we will allow dependence for the same $v$ and $e$ when we vary the decisions $x,y$. The details will be discussed when we proceed to concrete examples.

\subsection{Examples}

\subsubsection{Independent set}\label{example:IS}

Suppose the nodes of the graph are equipped with weights $W_v\ge 0, v\in V$. A set of nodes $I\subset V$ is an independent set if $(u,v)\notin E$ for every $u,v\in I$. The weight of an (independent) set $I$ is $\sum_{u\in I}W_u$. The maximum weight independent set problem is the problem of finding the independent set $I$ with the largest weight. It can be recast as a decision network problem by setting $\chi=\{0,1\}, \Phi_e(0,0)=\Phi_e(0,1)=\Phi_e(1,0)=0,\Phi_e(1,1)=-\infty, \Phi_v(1)=W_v, \Phi_v(0)=0$.

\subsubsection{Graph Coloring}

An assignment $\phi$ of nodes $V$ to colors $\{1,\ldots,q\}$ is defined to be proper coloring if no monochromatic edges are created. Namely, for every edge $(v,u)$, $\phi(v)\ne \phi(u)$. Suppose each node/color pair $(v,x)\in V\times \{1,\ldots,q\}$ is equipped with a weight $W_{v,x}\ge 0$. The (weighted) coloring problem is the problem of finding a proper coloring $\phi$ with maximum total weight $\sum_v W_{v,\phi(v)}$. In terms of decision network framework, we have $\Phi_{v,u}(x,x)=-\infty, \Phi_{v,u}(x,y)=0, \forall x\ne y\in \chi=\{1,\ldots,q\}, (v,u)\in E$ and $\Phi_v(x)=W_{v,x}, \forall v\in V,x\in \chi$.

\subsubsection{MAX 2-SAT}

Let $(Z_1,\ldots,Z_n)$ be a set of  boolean variables. Let $(C_1,\ldots,C_m)$ be a list of clauses of the form $(Z_i \vee Z_j)$, $(Z_i \vee \overline{Z_j})$, $(\overline{Z_i}\vee Z_j)$ or $(\overline{Z_i} \vee \overline{Z_j})$. The MAX-2SAT problem consists in finding an assignment for binary variables $Z_i$ which maximizes the number of satisfied clauses $C_j$. In terms of a decision network, take $V=\{1,\ldots,n\}$, $E=\{(i,j): \text{$Z_i$ and $Z_j$ appear in a common clause}\}$, and for any $k$, let $\Phi_k(x,y)$ to be $1$ if the clause $C_k$ is satisfied when $(Z_i,Z_j)=(x,y)$ and $0$ otherwise. Let $\Phi_v(x)=0$ for all $v,x$.

\subsubsection{MAP estimation}

In this example, we see a situation in which the reward functions are naturally randomized.
Consider a graph $(V,E)$ with $|V|=n$ and $|E|=m$, a set of real numbers $\bold{p}=(p_1,\ldots,p_n)\in [0,1]^{n}$, and a family $(f_1,\ldots,f_{m})$ of functions such that for each $(i,j)\in E$, $f_{i,j}=f_{i,j}(o,x,y):\Real\times\{0,1\}^2\rightarrow \Real_+$ where $o\in\Real$ and $x,y \in \{0,1\}$. Assume that for each $(x,y)$, $f_{i,j}(\cdot,x,y)$ is a probability density function. Consider two sets $\bold{C}=(C_i)_{1\leq i\leq n}$ and $\bold O=(O_j)_{1 \leq j \leq m}$ of random variables, with joint probability density

$$P(\bold O,\bold C)=\prod_i \: {p_i}^{c_i} (1-p_i)^{1-c_i} \prod_{(i,j)\in E} f_{i,j}(o_{i,j},c_i,c_j)$$

\noindent $\bold{C}$ is a set of Bernoulli random variables (``causes") with probability $P(C_i=1)=p_i$, and $\bold{O}$ is a set of continuous ``observation" random variables. Conditional on the cause variables $\bold C$, the observation variables $\bold O$ are independent, and each $O_{i,j}$ has density $f_{i,j}(o,c_i,c_j)$. Assume the variables $\bold O$ represent observed measurements used to infer on hidden causes $\bold C$. Using Bayes's formula, given observations $\bold O$, the log posterior probability of the causes variables $\bold C$ is equal to:
\begin{align*}
\log P(\bold C = \bold c \: | \: \bold O=\bold o)= K + \sum_i \Phi_i(c_i)+\sum_{i,j\in E} \Phi_{i,j}(c_i,c_j)
\end{align*}
where
\begin{align*}
\Phi_i(c_i)=&\log(p_i/(1-p_i)) c_i\\
\Phi_{i,j}(c_i,c_j)=&\log(f_{i,j}(o_{i,j},c_i,c_j))
\end{align*}
where $K$ is a random number which does not depend on $\bold c$.
Finding the maximum a posteriori values of $\bold C$ given $\bold O$ is equivalent to finding the optimal solution of the decision network $\Graph=(V,E,\Phi,\{0,1\})$. Note that the interaction functions $\Phi_{i,j}$ are naturally randomized, since $\Phi_{i,j}(x,y)$ is a continuous random variable with distribution $$d\pr(\Phi_{i,j}(x,y)= t)=\text{e}^t \sum_{x',y' \in \{0,1\}}  d\pr(f_{i,j}(o,x',y')=\text{e}^t)$$

\subsection{Notations}
For any two nodes $u$,$v$ in $V$, let $d(u,v)$ be  the length (number of edges) of the shortest path between $u$ and $v$. Given a node $u$ and integer $r\ge 0$, let $\mathcal{B}_{\Graph}(u,r)\eqdef\{v\in V: d(u,v)\leq r\}$ and $\mathcal{N}_{\Graph}(u)\eqdef\mathcal{B}(u,1)\backslash \{u\}$ be the set of neighbors of $u$. For any node $u$, let $\Delta_{\Graph}(u)\eqdef |\mathcal{N}_{\Graph}(u)|$ be the number of neighbors of $u$ in $\mathcal{G}$. Let $\Delta_{\Graph}$ be the maximum degree of graph $(V,E)$; namely, $\Delta_{\Graph}=\max_v |\mathcal{N}(v)|$. Often we will omit the reference to the network $\Graph$ when it is obvious from the context.\vvert

For any subgraph $(V',E')$ of $(V,E)$ (i.e. $V'\subset V$, $E'\subset E\cap V'^2$), the subnetwork $\Graph'$ induced by $(V,E)$ is the network $(V',E',\Phi',\chi)$, where $\Phi'=((\Phi_v)_{v\in V'},(\Phi_e)_{e\in E'})$.\vvert

Given a subset of nodes $\bold{v}=(v_1,\ldots,v_k)$, and $\bold{x}=(x_1,\ldots,x_k)\in \chi^k$, let $J_{\Graph,\bold{v}}(\bold{x})$ be the optimal value when the actions of nodes $v_1,\ldots,v_k$ are fixed to be $x_{1},\ldots,x_{k}$ respectively: $J_{\Graph,\bold{v}}(\bold{x})=\max_{\bold{x}: x_{v_i}=x_i, 1\le i\le k}F_\Graph(\bold{x})$. Given $v\in V$ and $x\in \chi$, the quantity $B_{\Graph,v}(x)\eqdef J_{\Graph,v}(x)-J_{\Graph,v}(0)$ is called the \emph{cavity} of action $x$ at node $v$. Namely it is the difference of optimal values when the decision at node $v$ is set to $x$ and $0$ respectively (the choice of $0$ is arbitrary). The cavity function of $v$ is $B_{\Graph,v}=(B_{\Graph,v}(x))_{x\in \chi}$. Since $B_{\Graph,v}(0)=0$, $B_{\Graph,v}$ can be thought of as element of $\Real^{T-1}$. In the important special case $\chi=\{0,1\}$, the cavity function is a scalar $B_{\Graph,v}=J_{\Graph,v}(1)-J_{\Graph,v}(0)$. In this case, if $B_{\Graph,v}>0$ (resp. $B_{\Graph,v}<0$) then $J_{\Graph,v}(1)>J_{\Graph,v}(0)$ and action $1$ (resp. action $0$) is optimal for $v$. When $B_{\Graph,v}=0$ there are optimal decisions consistent both with $x_v=0$ and $x_v=1$. Again, when $\Graph$ is obvious from the context, it will be omitted from the notation. \vvert

For any network $\Graph$, we call $M(\Graph)=\max(|V|,|E|,|\chi|)$ the size of the network. Since we will exclusively consider graphs with degree bounded by a constant, for all practical purposes we can think of $|V|$ as the size of the instance. When we say polynomial time algorithm, we mean that the running time of the algorithm is upper bounded by a polynomial in $|V|$.
An algorithm $\mathcal{A}$ is said to be an $\epsilon$-loss additive approximation algorithm for the problem of finding the optimal decision if for any network $\Graph$ it produces in polynomial time a decision $\hat x$ such that $J_{\Graph}-F(\hat x)<\epsilon$.
If all cost functions are positive, the algorithm $\mathcal{A}$ is said to be an $(1+\epsilon)$-factor multiplicative approximation algorithm if it outputs a solution $\hat x$ such that $J_{\Graph}/F(\hat{x})<1+\epsilon$.
 We call such an algorithm an additive (resp. multiplicative) PTAS (Polynomial Time Approximation Scheme) if it is an  $\epsilon$-loss (resp. $(1+\epsilon)$-factor) additive (resp. multiplicative) approximation factor algorithm for every $\epsilon>0$ and runs in time which is polynomial in $|V|$. An algorithm is called an FPTAS (Fully Polynomial Time Approximation Scheme) if it runs in time which is polynomial in $n$ and $1/\epsilon$. For our purposes another relevant class of algorithms is EPTAS. This is the class of algorithms which produces $\epsilon$ approximation in time $O(|V|^{O(1)}g(\epsilon))$, where $g(\epsilon)$ is some function independent from $n$. Namely, while it is not required that the running time of the algorithm is polynomial in $1/\epsilon$, the $1/\epsilon$ quantity does not appear in the exponent of $n$.  Finally, in our context, since the input is random, we will say that an algorithm is an additive (resp. multiplicative) PTAS with high probability if for all $\epsilon>0$ it
outputs in time polynomial in $|V|$ a solution $\hat{x}$ such that $\pr(J_{\Graph}-F(\hat{x})>\epsilon)<\epsilon$ (resp. $\pr(F(\hat x)/J_{\Graph}>1+\epsilon) \le \epsilon$); FPTAS and EPTAS w.h.p. are similarly defined. Since our algorithm provide probabilistic guarantee, one may wonder whether FPRAS (Fully Polynomial Randomized Approximtion Scheme) would be a more appropriate framework. The typical setting for FPRAS is, however, a deterministic problem input and the randomization is associated purely with algorithm. In our setting however, the setting itself is random, though the algorithms, with the exception of MWIS, are deterministic. 

\section{Main results}\label{sec:MainResults}
In this section we state our main results. The first two results relate to decision networks with uniformly and normally distributed costs, respectively, without any combinatorial constraints on the decisions. The last set of results corresponds to the MWIS problem, which does incorporate the combinatorial constraint of the independence property.

\subsection{Uniform and Gaussian Distributions}\label{sec:resunifom}
Given $\Graph=(V,E,\Phi,\{0,1\})$, suppose that for all $u\in V$, $\Phi_u(1)$ is uniformly distributed on $[-I_1,I_1]$, $\Phi_u(0)=0$, and that for every $e \in E$, $\Phi_e(0,0),\Phi_e(1,0),\Phi_e(0,1)$ and $\Phi_e(1,1)$ are all independent and uniformly distributed on $[-I_2,I_2]$, for some $I_1,I_2>0$. Intuitively, $I_{1}$ quantifies the 'bias' each agent has towards one action or another, while $I_{2}$ quantifies the strength of interactions between agents.

\begin{Theorem}\label{thm:Uniform}
Let $\beta=\frac{5I_{2}}{2I_{1}}$. If $\beta(\Delta-1)^2<1$, then there exists an additive FPTAS for finding $J_\Graph$ with high probability.
\end{Theorem}

Now we turn to the case of Gaussian costs. Assume that for any edge $e=(u,v)$ and any pair of action $(x,y)\in\{0,1\}^{2}$, $\Phi_{u,v}(x,y)$ is a Gaussian random variable with mean $0$ and standard deviation $\sigma_e$. For every node $v\in V$, suppose $\Phi_v(1)=0$ and that $\Phi_v(0)$ is a Gaussian random variable with mean $0$ and standard deviation $\sigma_p$. Assume that all rewards $\Phi_{e}(x,y)$ and $\Phi_{v}(x)$ are independent for all choices of $v,e,x,y$.

\begin{Theorem}\label{thm:Gauss}
Let $\beta=\sqrt{\frac{\sigma_e^2}{\sigma_e^2+\sigma_p^2}}$.
If $\beta(\Delta-1)+\sqrt{\beta(\Delta-1)^3}<1 $, then there exists an additive FPTAS for finding $J_\Graph$ with high probability.
\end{Theorem}
While our main result was stated for the case of independent costs, we have obtained a more general result which incorporates the case of correlated
edge costs. It is given as Proposition~\ref{prop:Gaussiancase} in Section~\ref{sec:coupling}.

\subsection{Maximum Weight Independent Sets}
Here, we consider a variation of the MWIS problem where the nodes of the graph are equipped with random weights $W_i, i\in V$, drawn independently from a common distribution $F(t)=\pr(W\le t), t\ge 0$. Let $I^*=I^*(\Graph)$ be the largest weighted independent set, when it is unique and let $W(I^*)$ be its weight. In our setting it is a random variable. Observe that $I^*$ is indeed almost surely unique when $F$ is a continuous distribution.

\begin{Theorem}\label{theorem:ISMainResult}
If $\Delta_\Graph\leq 3$ and the weights are exponentially distributed with parameter $1$, then there exists a multiplicative EPTAS for finding $J_\Graph$ with high probability. The algorithm runs in time  $O\Big(|V| 2^{O(\epsilon^{-2}\log(1/\epsilon))}\Big)$.
\end{Theorem}

An interesting implication of Theorem~\ref{theorem:ISMainResult} is that while the Maximum (cardinality) Independent Set problem admits neither a polynomial time algorithm nor a PTAS (unless P=NP), even when the degree is bounded to $3$~\cite{BermanKarpinski,trevisan2001nar}, the problem of finding the maximum weight independent set becomes tractable for certain distributions $F$, in the PTAS sense.

The exponential distribution is not the only distribution which can be analyzed in this framework, it is just the easiest to work with. For any phase-type distribution, we can characterize correlation decay and identify sufficient conditions for correlation decay to hold. It is natural to ask if the above result can be generalized, and in particular to wonder if it is possible to find for each $\Delta$ a distribution which guarantees that the correlation decay property holds for graphs with degree bounded by $\Delta$. It is indeed possible, as we extend Theorem~\ref{theorem:ISMainResult}, albeit to the case of mixtures of exponential distributions. Let $\rho>25$ be an arbitrary constant and let $\alpha_j=\rho^j, j\ge 1$.

\begin{Theorem}\label{theorem:ISMainResult2}
Assume $\Delta_\Graph\leq \Delta$, and that the weights are distributed according to
$P(W>t)=\frac{1}{\Delta}\sum_{1\le j \le \Delta} \exp(-\alpha_j t)$.
Then there exists a FPTAS for finding $J_\Graph$ with high probability. The algorithm runs in time  $O\Big(n (\frac{1}{\epsilon})^{\Delta}\Big)$.
\end{Theorem}

Note that for the mixture of exponential distributions described above our algorithm is in fact an FPTAS as opposed to an EPTAS for Theorem~\ref{theorem:ISMainResult}. This is essentially due to the fact that the conditions of Theorem~\ref{theorem:ISMainResult} are at the `boundary' of correlation decay; more technical details are given in \ref{section:MWIS}.

Our final result is a partial converse to the results above; one could conjecture that randomizing the weights makes the problem essentially easy to solve, and that perhaps being able to solve the randomized version does not tell us much about the deterministic version. We show that this is not the case, and that the setting with random weights hits a complexity-theoretic barrier just as the classical cardinality problem does. Specifically, we show that for graphs with sufficiently large degree the problem of finding the largest weighted independent set with i.i.d. exponentially distributed weights does not admit a PTAS.  We need to keep in mind that since we dealing with instances which are random (in terms of weights) and worst-case (in terms of the underlying graph) at the same time, we need to be careful as to the notion of hardness we use.

Specifically, for any $\rho<1$, define an algorithm $\mathcal{A}$ to be a factor-$\rho$ polynomial time approximation algorithm for computing $\E[W(I^*)]$ for graphs with degree at most $\Delta$, if given any graph with degree at most $\Delta$, $\mathcal{A}$ produces a value $\hat w$ such that $\rho \le \hat w/E[W(I^*)]\le 1/\rho$ in time bounded by $O(n^{O(1)})$. Here the expectation is with respect to the exponential weight distribution and the constant exponent $O(1)$ is allowed to depend on $\Delta$.

En route of  Theorems \ref{theorem:ISMainResult} and \ref{theorem:ISMainResult2} we establish similar results for expectations:  there exists an EPTAS and FPTAS respectively for computing the deterministic quantity $E[W(I^{*})]$, the expected weight of the MWIS in the graph $\Graph$ considered.

However, our next result shows that if the maximum degree of the graph is increased, it is impossible to approximate the quantity $E[W(I^{*})]$ arbitrarily closely, unless P=NP. Specifically,

\begin{Theorem}\label{theorem:ISMainResult3}
There exist $\Delta_0$ and $c_1^*,c_2^*$ such that for all $\Delta\ge \Delta_0$ the problem of computing  $\E[W(I^*)]$ to within a multiplicative factor $\rho=\Delta/(c_1^*\,(\log\Delta)\, 2^{c_2^*\sqrt{\log\Delta}})$  for graphs with degree at most $\Delta$ cannot be solved in polynomial time, unless
P=NP.
\end{Theorem}
We could compute a concrete $\Delta_0$ such that for all $\Delta\ge \Delta_0$ the claim of the theorem holds, though such $\Delta_0$ explicitly does not seem to offer much insight. We note that in the related work by Trevisan~\cite{trevisan2001nar}, no attempt is made to compute a similar bound either.

\section{The cavity recursion}\label{sec:Cav}

In this section, we introduce the \emph{cavity recursion}, an exact recursion for computing the cavity functions of each node in a general decision network. We first start by giving the cavity recursion for trees (which is already known as the max-product belief propagation algorithm), and then give a generalization for all networks.

\subsection{Trees}

Given a decision network $\Graph=(V,E,\Phi,\chi)$ suppose that $(V,E)$ is a rooted tree with a root $u$. Using the graph orientation induced by the choice of $u$ as a root, let $\Graph_v$ be the subtree rooted in node $v$ for any node $v\in V$. In particular, $\Graph=\Graph_u$. Denote by $C(u)$ the set of children of $u$ in $(V,E)$. Given a node $u\in V$, a child $v\in C(u)$, and an arbitrary vector $B=(B(x), x\in\chi)$, define

\begin{align}\label{MuDefine}
    \mu_{u \leftarrow v}(x,B)= \max_{y} (\Phi_{u,v}(x,y)+B(y))-\max_{y} (\Phi_{u,v}(0,y)+B(y))
    \end{align}
for every action $x\in \chi$.
    $\mu$ is called partial cavity function.

\begin{Proposition} \label{CavTree}
For every $u\in V$ and $x\in \chi$,
\begin{equation}\label{eq:CavEqTree}
B _u(x)=\Phi_u(x)-\Phi_u(0)+ \sum_{v\in C(u)} \mu_{u \leftarrow v}(x,B _{\Graph_{v},v})
\end{equation}
\end{Proposition}

\begin{proof}
Suppose $C(u)=\{v_1,\ldots,v_d\}$. Observe that the subtrees $\Graph_{v_i}, 1\le i\le d$ are disconnected (see figure \ref{fig:tree})
Thus,
\begin{align*}
B_u (x)&= \Phi_u(x)+ \max_{ x_{1},\ldots,x_{d}} \Big \{\sum_{j=1}^{d} \Phi_{u,v_j}(x,x_{j})+J_{\Graph_{v_j},v_j}(x_{j}) \Big \} \\
       &- \Phi_u(0)- \max_{x_{1},\ldots,x_{d}} \Big \{ \sum_{j=1}^{d} \Phi_{u,v_j}(0,x_{j})+J_{\Graph_{v_j},v_j}(x_{j})) \Big \}\\
&=\Phi_u(x)-\Phi_u(0)\\
        &+  \sum_{j=1}^d  \Big\{ \max_{y} \big(\Phi_{u,v_j}(x,y)+J_{\Graph_{v_j},v_j}(y)\big)
        -\max_{y}  \big(\Phi_{u,v_j}(0,y)+J_{\Graph_{v_j},v_j}(y)\big) \Big \}
\end{align*}
For every $j$
\begin{align*}
&\max_{y} \big(\Phi_{u,v_j}(x,y)+J_{\Graph_{v_j},v_j}(y)\big)
 -  \max_{y} \big(\Phi_{u,v_j}(0,y)+J_{\Graph_{v_j},v_j}(y)\big) =\\
&\max_{y} \big(\Phi_{u,v_j}(x,y)+J_{\Graph_{v_j},v_j}(y)-J_{\Graph_{v_j},j}(0)\big)
 -  \max_{y} \big(\Phi_{u,v_j}(0,y)+J_{\Graph_{v_j},v_j}(y)-J_{\Graph_{v_j},v_j}(0)\big)
\end{align*}
The quantity above is exactly $\mu_{u \leftarrow v_j}(x,B_{v_j,\Graph_{v_j}})$.
\end{proof}

\begin{figure}
    \psfrag{l}[c][c][1]{$\mathcal{G}_u$}
    \psfrag{a}[c][c][1]{$u$}
    \psfrag{b}[c][c][1]{$v_1$}
    \psfrag{c}[c][c][1]{$v_2$}
    \psfrag{d}[c][c][1]{$v_3$}
    \psfrag{e}[c][c][1]{$w_1$}
    \psfrag{f}[c][c][1]{$w_2$}
    \psfrag{g}[c][c][1]{$y_1$}
    \psfrag{h}[c][c][1]{$y_2$}
    \psfrag{i}[c][c][1]{$z_1$}
    \psfrag{j}[c][c][1]{$z_2$}
	 \psfrag{k}[c][c][1]{$\mathcal{G}_{v_1}$}
	 \psfrag{p}[c][c][1]{$B_{\Graph,u}(x)$}
	 \psfrag{q}[c][c][1]{$B_{\Graph_{v_1},v_1}$}
	 \psfrag{m}[c][c][1]{$\mu_{u\leftarrow v_1}(x,B_{\Graph_{v_1},v_1})$}
	 \center
    \includegraphics[scale=0.7]{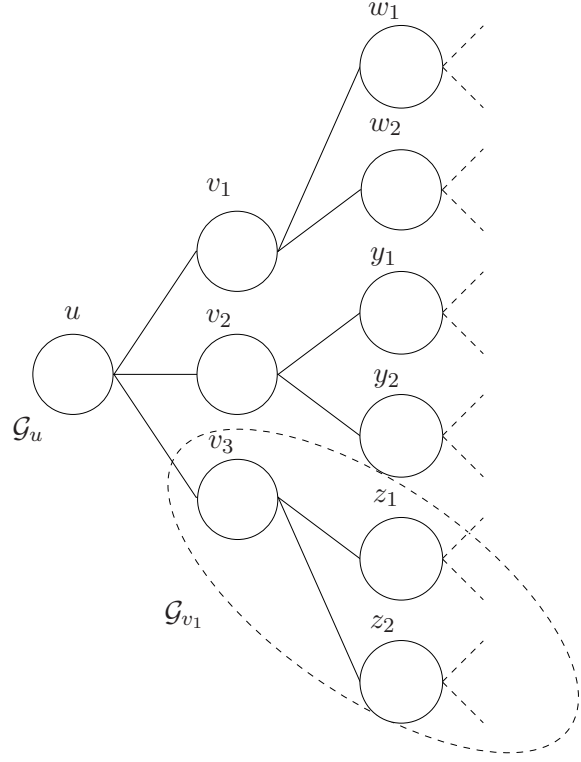}
    \label{fig:tree}
    \caption{Cavity recursion for trees, equivalent to the BP algorithm}
\end{figure}

Iteration \eqref{eq:CavEqTree} constitutes what is known as (max-product) belief propagation. Proposition \ref{CavTree} is the restatement of the well-known fact that BP finds an optimal solution on a tree (citation). BP can be implemented in non-tree like graphs, but then it is not guaranteed to converge, and even when it does it may produce wrong (suboptimal) solutions. In the following section we construct a generalization of BP which is guaranteed to converge to an optimal decision.

\subsection{General graphs}\label{subsection:generalgraphs}

The goal of this subsection is to construct a generalization of identity (\ref{eq:CavEqTree}) for an arbitrary network $\Graph$. This can be achieved by building a sequence of certain auxiliary decision networks $\Graph(u,j,x)$ constructed as follows.

Given  a decision network $\Graph=(V,E,\Phi,\chi)$ where the underlying graph is arbitrary, fix any node $u$ and action $x$ and let $\mathcal{N}(u)=\{v_1,\ldots,v_d\}$. For every $j=1,\ldots,d$ let $\Graph(u,j,x)$ be the decision network $(V',E',\Phi',\chi)$ on the same decision set $\chi$ constructed as follows.
$(V',E')$ is the subgraph induced by $V'= V\setminus \{u\}$. Namely, $E'=E\setminus \{(u,v_1),\ldots,(u,v_d)\}$. Also $\Phi_e'=\Phi_e$ for all $e$ in $E'$ and the potential functions $\Phi_v'$ are defined as follows. For any $v \in V \backslash
\{u,v_1,\ldots,v_{j-1},v_{j+1},\ldots,v_d\}$, $\Phi'_v=\Phi_v$, and
\begin{align}
\Phi'_{v}(y)&=\Phi_{v}(y)+\Phi_{u,v}(x,y)\hspace{20pt}\text{for $v \in \{v_1,\ldots, v_{j-1}\}$}\notag\\
\Phi'_{v}(y)&=\Phi_{v}(y)+\Phi_{u,v}(0,y)\hspace{20pt}\text{for $v \in \{v_{j+1},\ldots,v_d\}$} \label{eq:GraphModified}
\end{align}

\begin{figure}
\psfrag{vc}[c][c][1]{$u$}
\psfrag{va}[c][c][1]{$v_1$}
\psfrag{vb}[c][c][1]{$v_2$}
\psfrag{vd}[c][c][1]{$v_3$}
\psfrag{ca}[l][l][1]{$\Phi_{u,v_1}$}
\begin{center}
\begin{align*}
B_u
    \left(
    \begin{array}{clr}
    \includegraphics[scale=0.21]{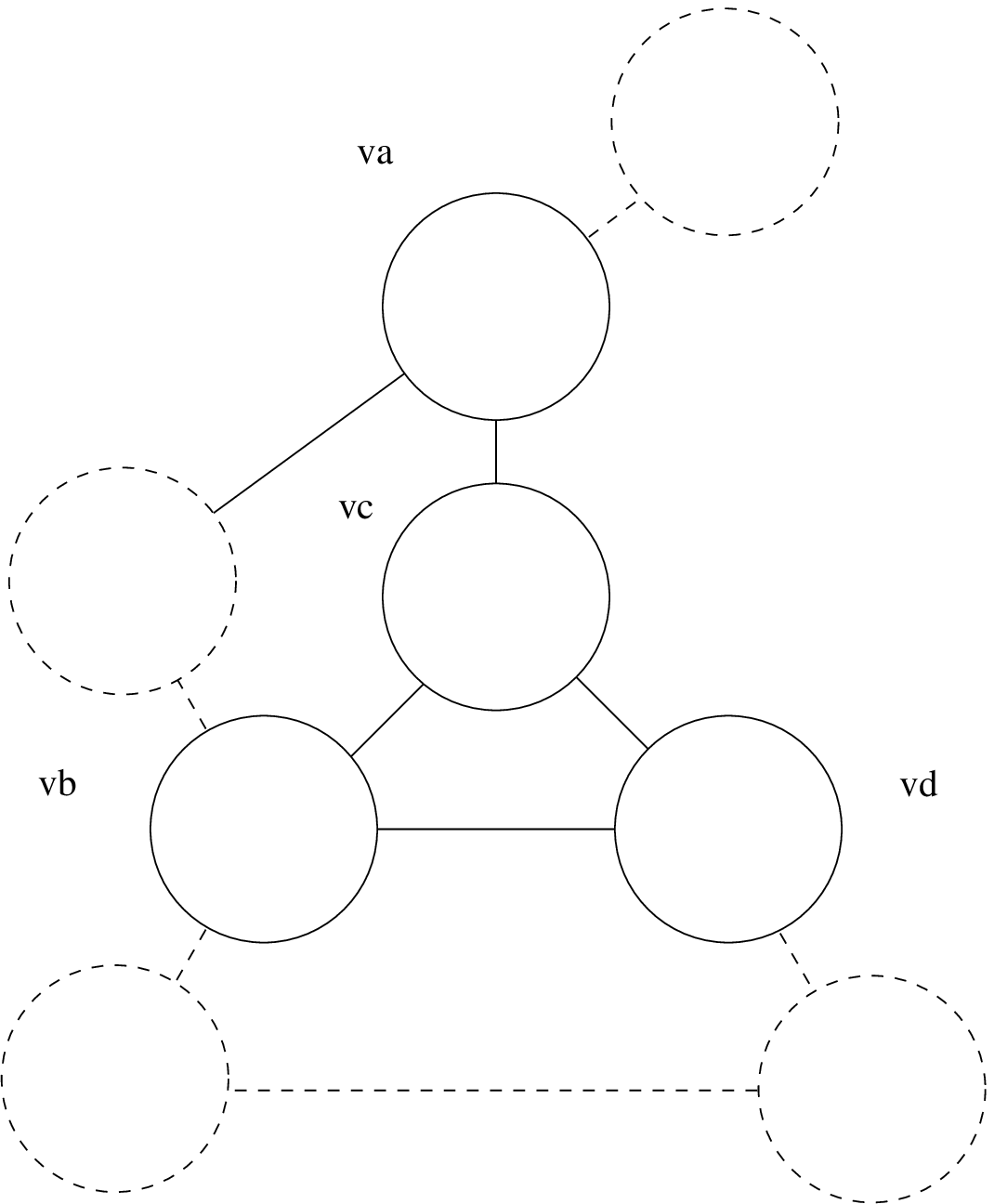}
    \end{array}
    \right)= \Phi_u(x)-\Phi_u(0) &+
    J
    \left(
    \begin{array}{clr}
    \includegraphics[scale=0.21]{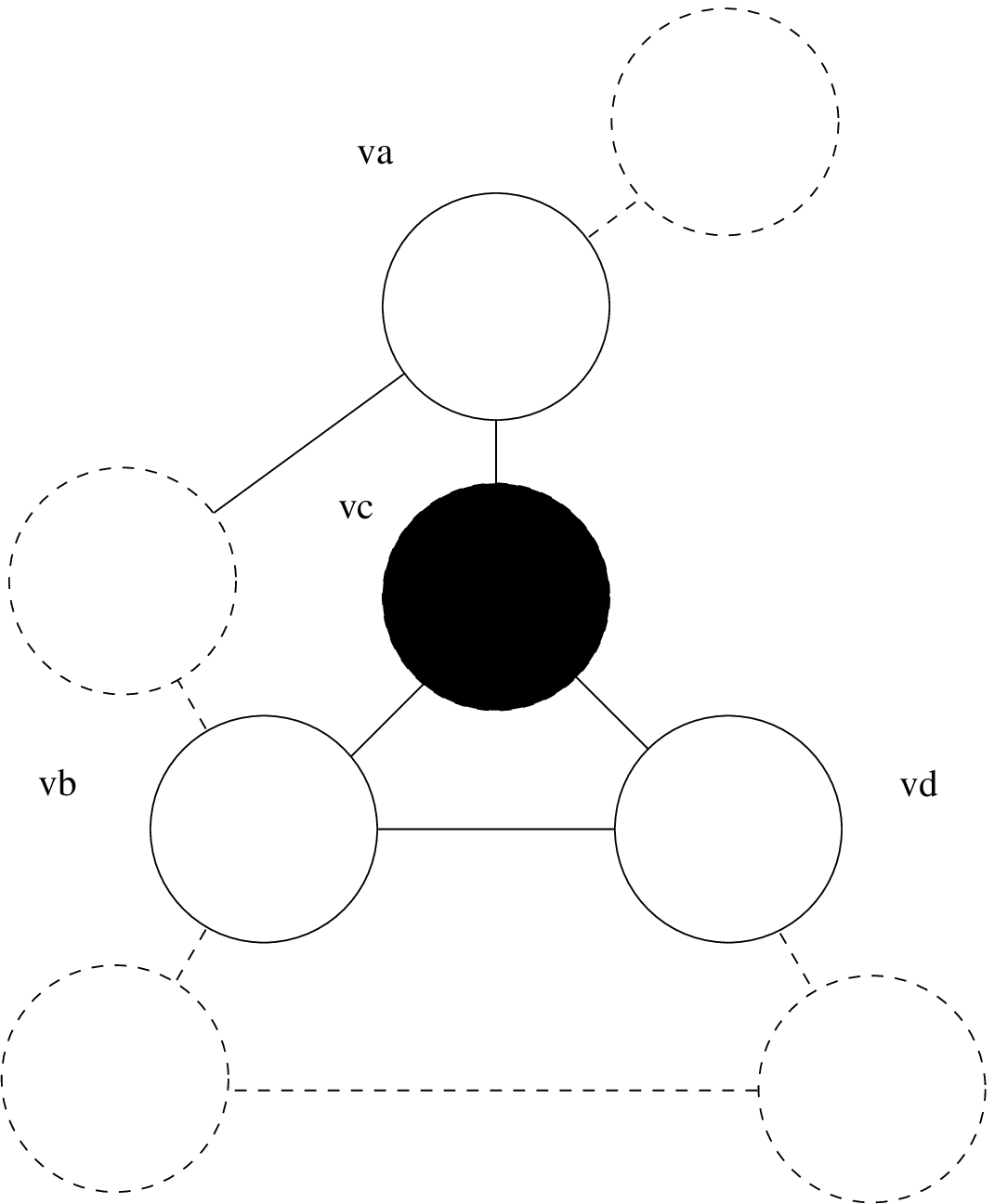}
    \end{array}
    \right)-
    J
    \left(
    \begin{array}{clr}
    \includegraphics[scale=0.21]{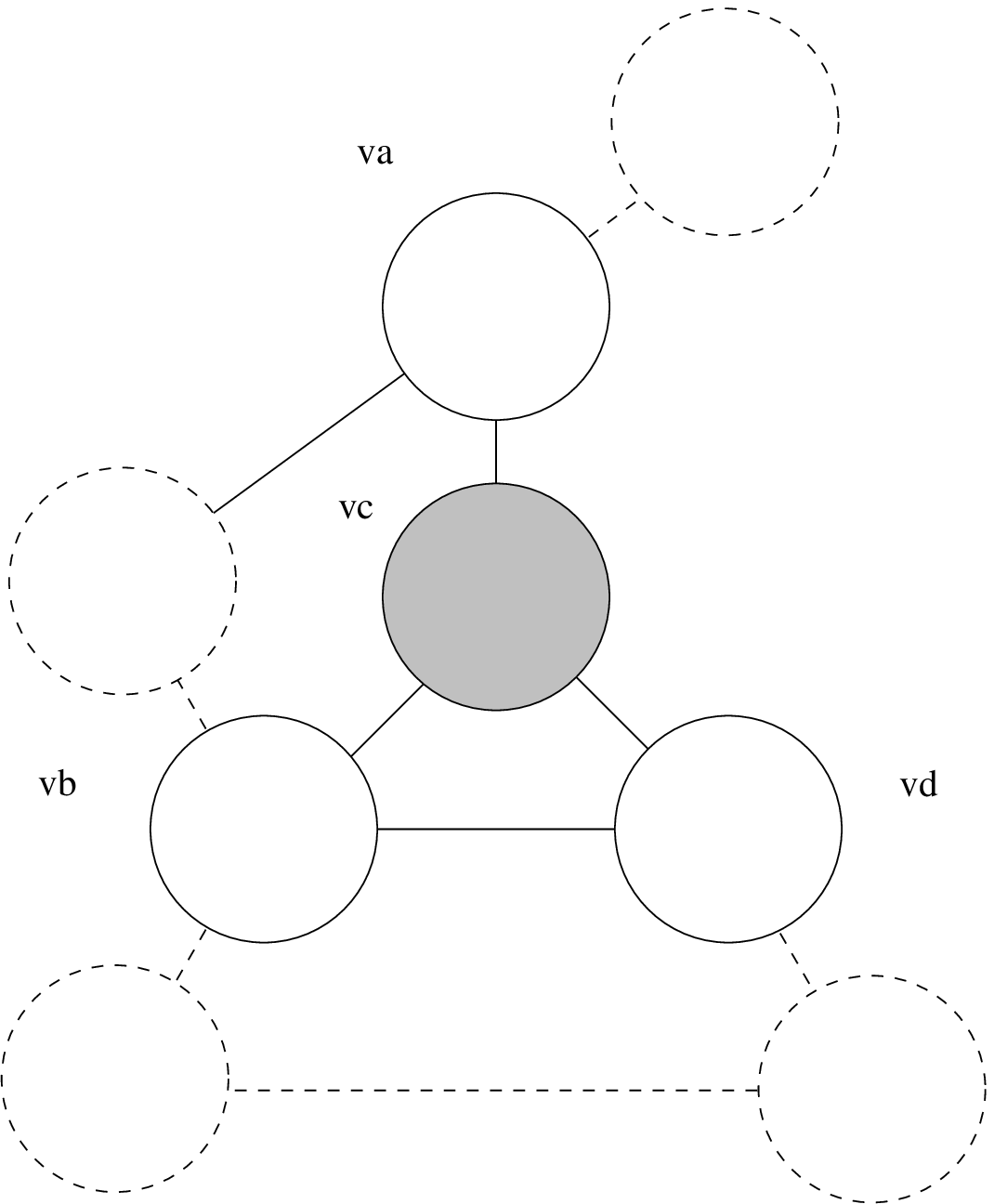}
    \end{array}
    \right)\\
    =
    \Phi_u(x)-\Phi_u(0) &+
    J
    \left(
    \begin{array}{clr}
    \includegraphics[scale=0.21]{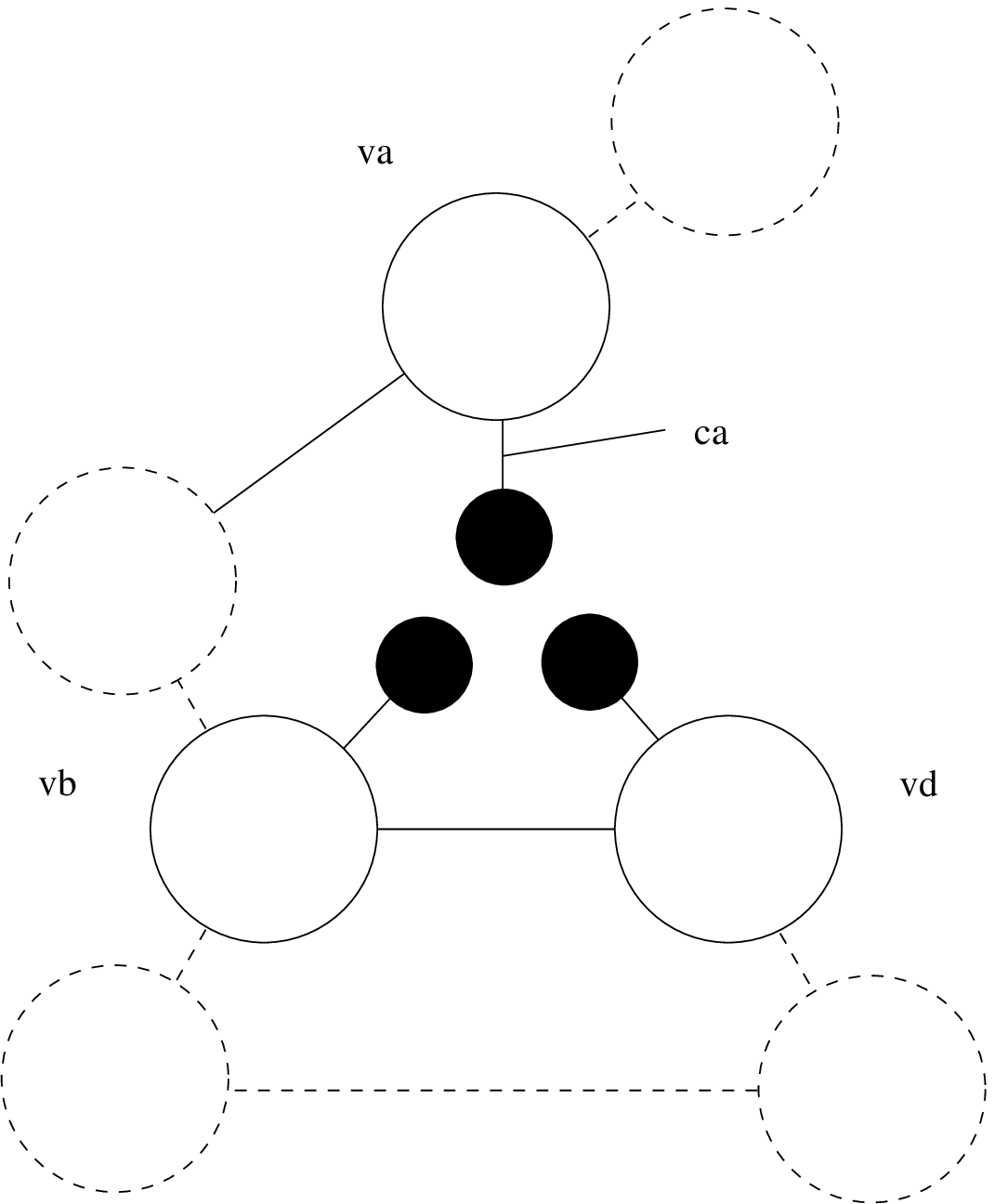}
    \end{array}
    \right)-
    J
    \left(
    \begin{array}{clr}
    \includegraphics[scale=0.21]{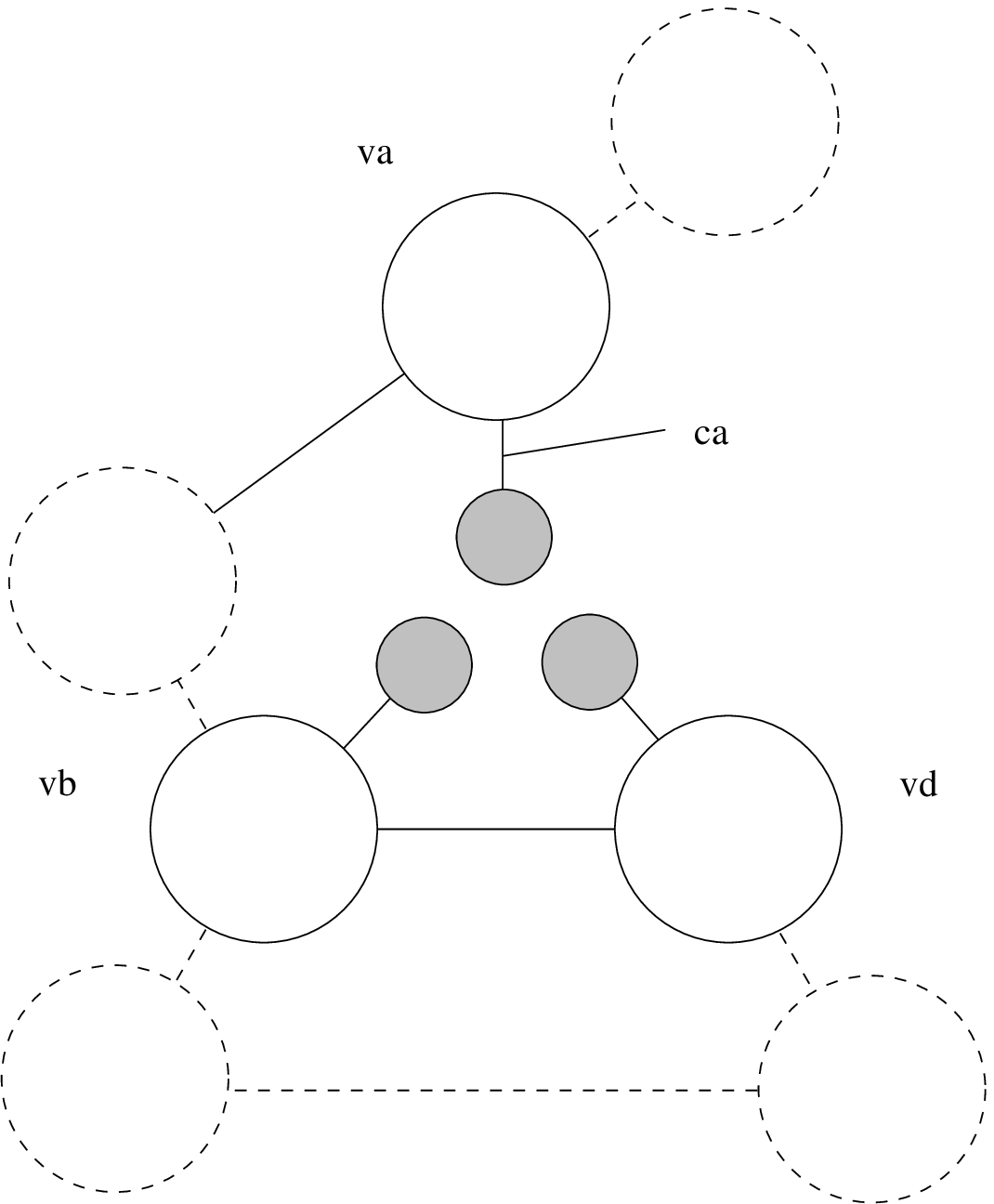}
    \end{array}
    \right)\\
    =
    \Phi_u(x)-\Phi_u(0) &+
    J
    \left(
    \begin{array}{clr}
    \includegraphics[scale=0.21]{cavd.eps}
    \end{array}
    \right)-
    J
    \left(
    \begin{array}{clr}
    \includegraphics[scale=0.21]{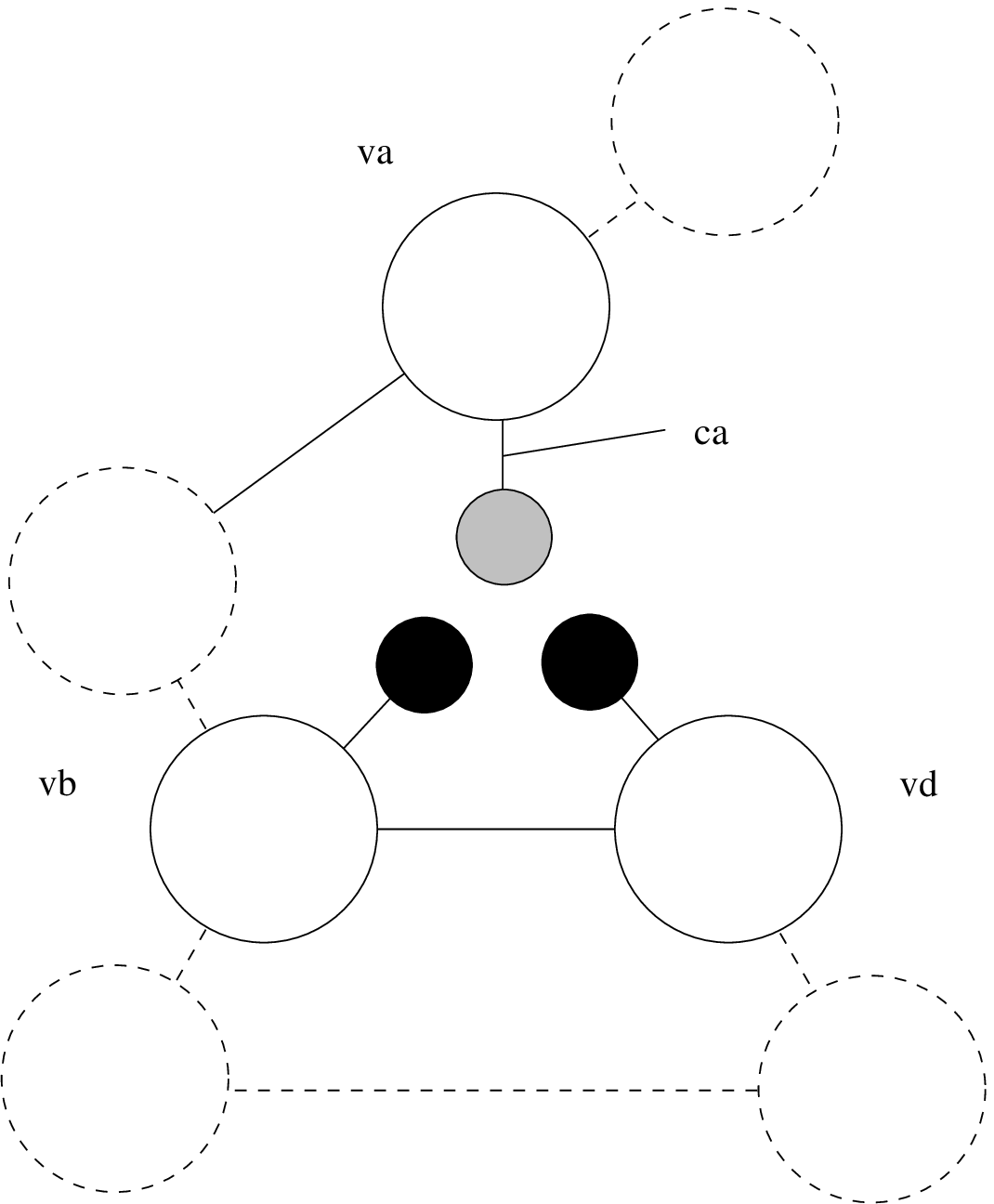}
    \end{array}
    \right)\\
    &+J
    \left(
    \begin{array}{clr}
    \includegraphics[scale=0.21]{cave.eps}
    \end{array}
    \right)-
    J
    \left(
    \begin{array}{clr}
    \includegraphics[scale=0.21]{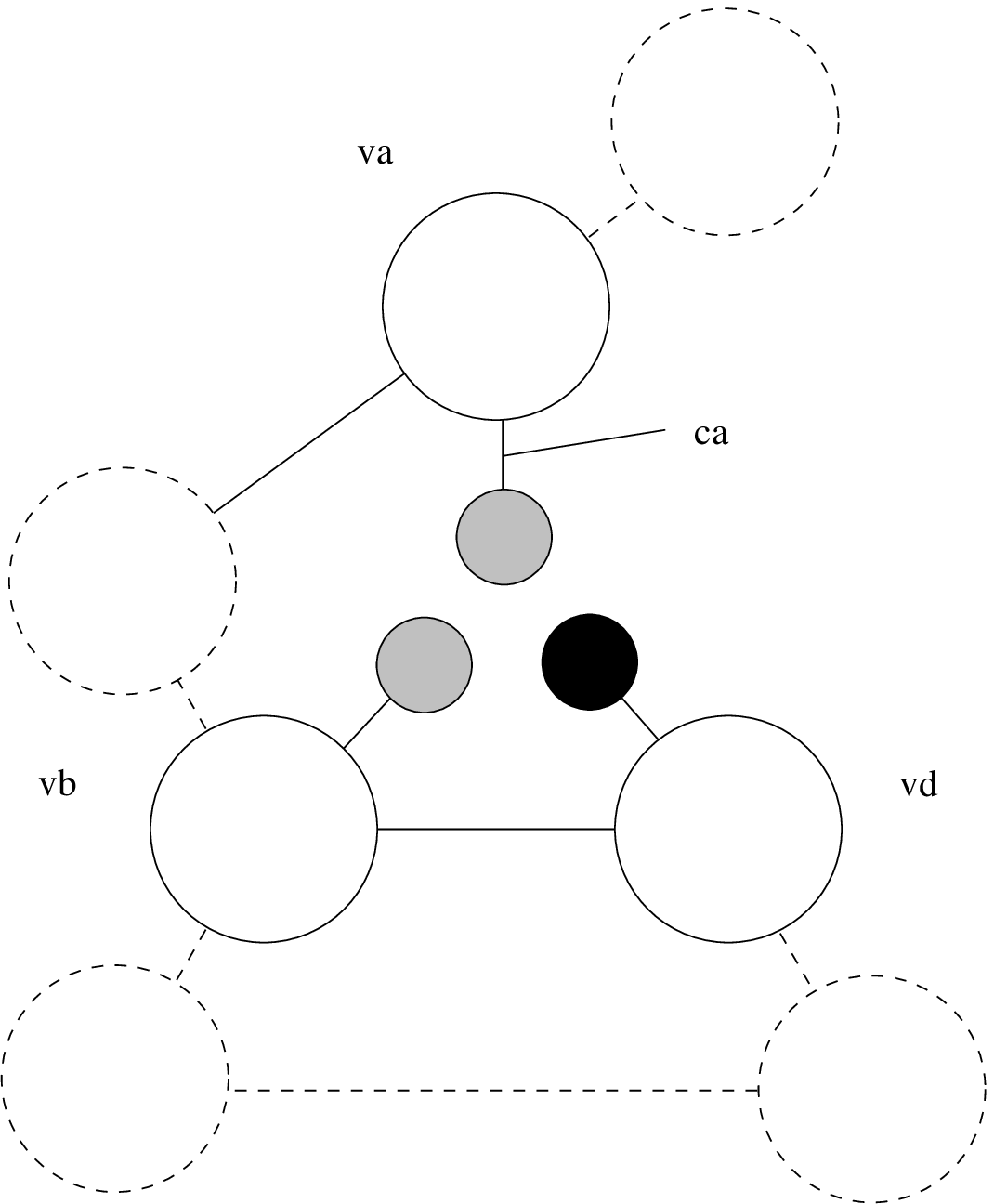}
    \end{array}
    \right)\\
    &+J
    \left(
    \begin{array}{clr}
    \includegraphics[scale=0.21]{cavf.eps}
    \end{array}
    \right)-
    J
    \left(
    \begin{array}{clr}
    \includegraphics[scale=0.21]{cavg.eps}
    \end{array}
    \right)
\end{align*}
\end{center}
\label{fig:firststep}
\caption{First step: building the telescoping sum; black nodes indicate decision $x$, gray node decision $0$; solid circles indicate neighbors of $u$, dotted circles indicate other nodes}
\end{figure}

\begin{figure}
\psfrag{vc}[c][c][1]{$u$}
\psfrag{va}[c][c][1]{$v_1$}
\psfrag{vb}[c][c][1]{$v_2$}
\psfrag{vd}[c][c][1]{$v_3$}
\psfrag{ca}[l][l][1]{$\Phi_{u,v_1}$}
\psfrag{cb}[l][l][1]{$\Phi_{u,v_3}$}
\psfrag{cc}[l][l][1]{$\tilde \Phi_{v_1}$}
\psfrag{cd}[l][l][1]{$\tilde \Phi_{v_3}$}
\begin{center}
\begin{align*}
&J
    \left(
    \begin{array}{clr}
    \includegraphics[scale=0.25]{cave.eps}
    \end{array}
    \right)-
    J
    \left(
    \begin{array}{clr}
    \includegraphics[scale=0.25]{cavf.eps}
    \end{array}
    \right)
    =J
    \left(
    \begin{array}{clr}
    \includegraphics[scale=0.25]{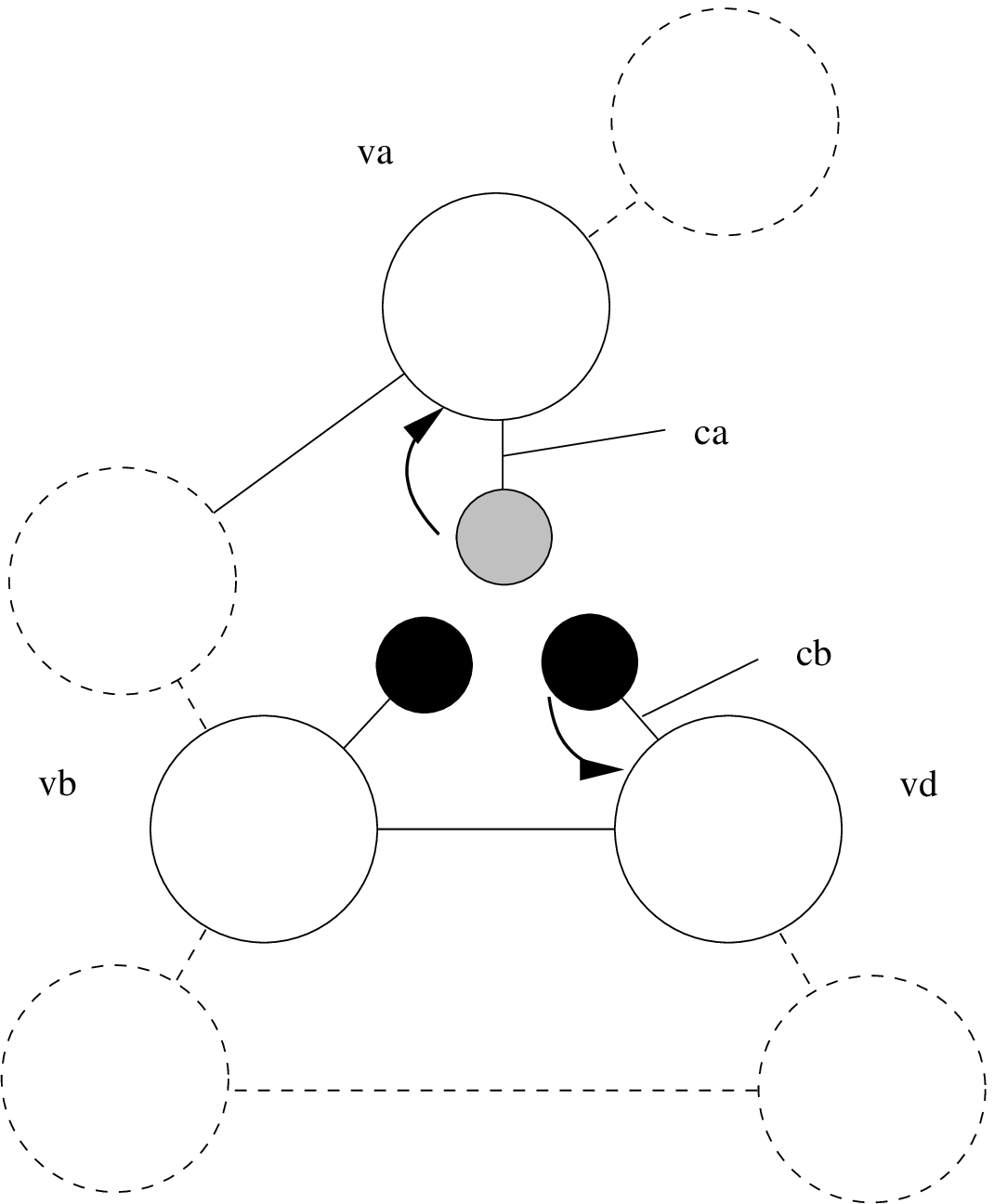}
    \end{array}
    \right)-
    J
    \left(
    \begin{array}{clr}
    \includegraphics[scale=0.25]{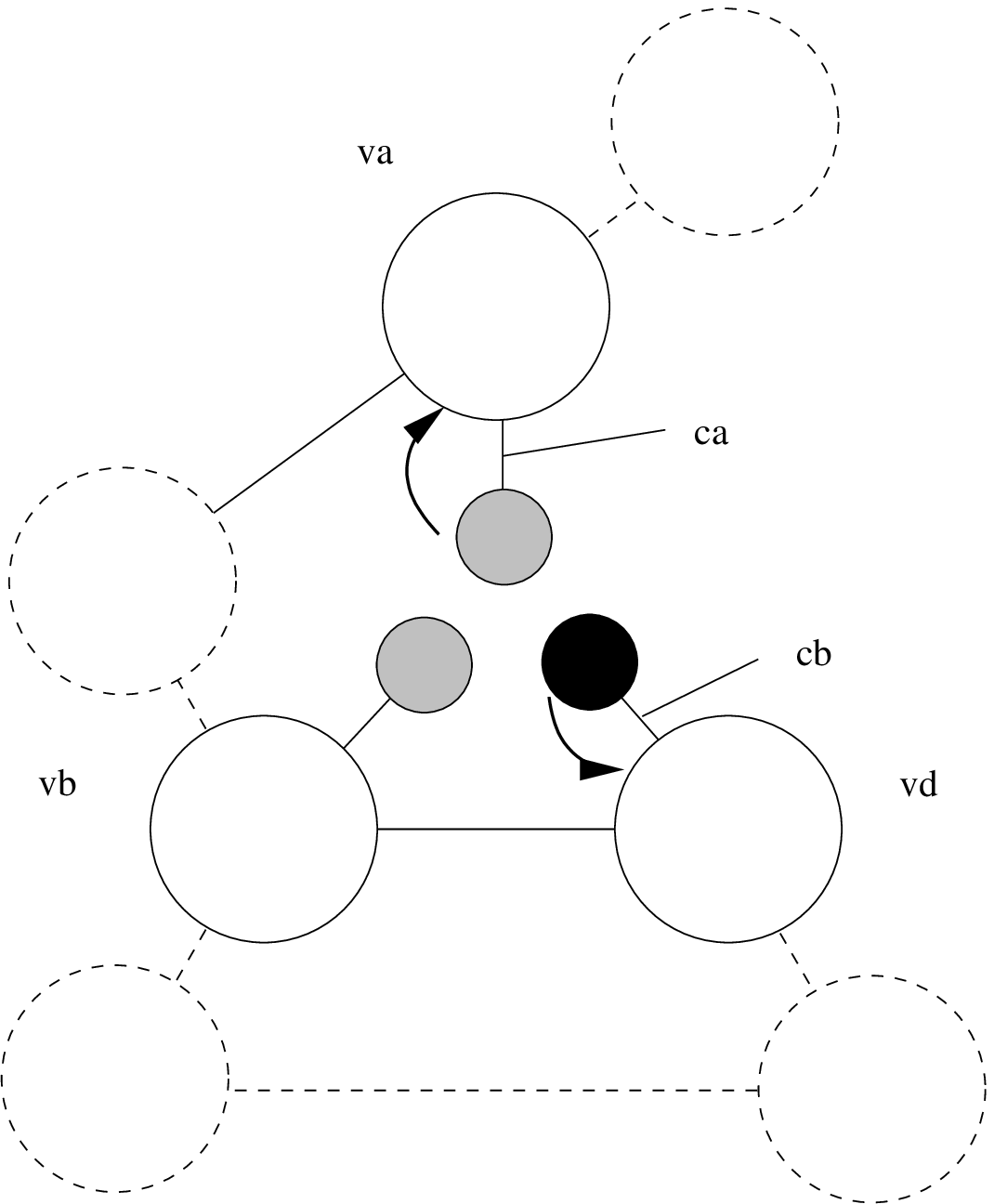}
    \end{array}
    \right)\\
      =&J
    \left(
    \begin{array}{clr}
    \includegraphics[scale=0.25]{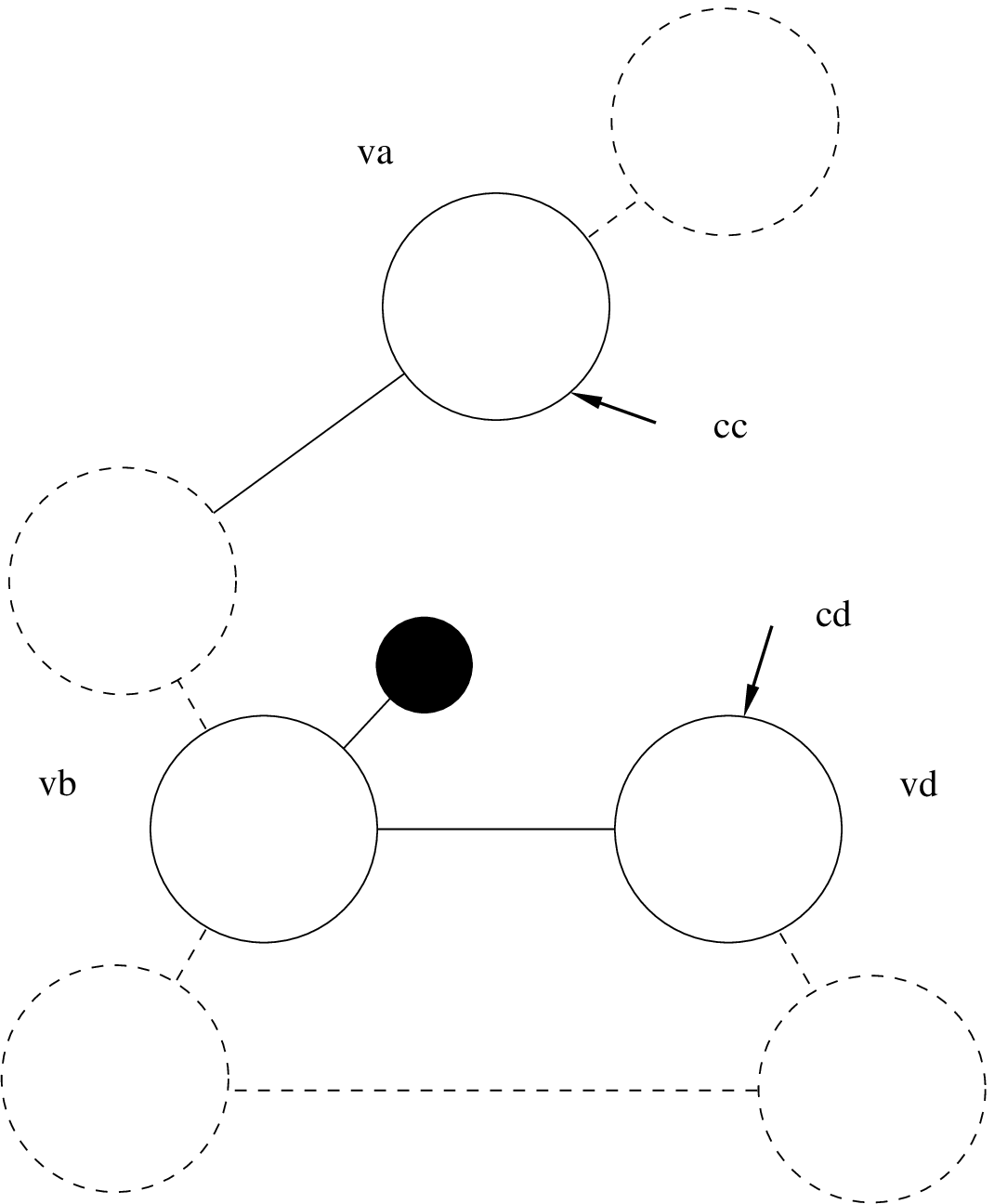}
    \end{array}
    \right)-
    J
    \left(
    \begin{array}{clr}
    \includegraphics[scale=0.25]{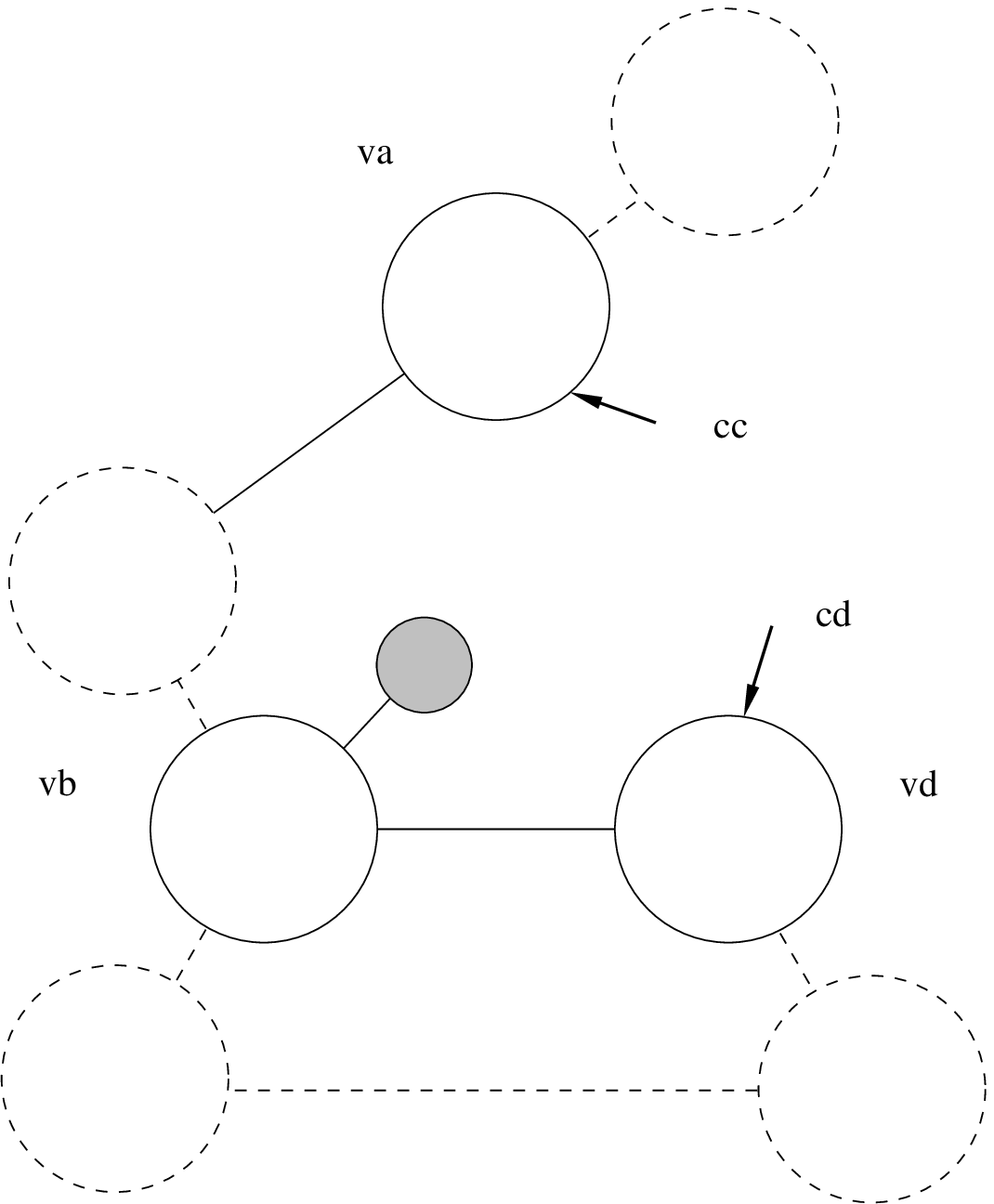}
    \end{array}
    \right)=\mu_{u\leftarrow v_2}(x,B_{\Graph(u,2,x),v_2})
    \end{align*}
    with
    \begin{align*}
    \tilde \Phi_{v_1}(z)=\Phi_{v_1}(z)+\Phi_{u,v_1}(0,z)\\
    \tilde \Phi_{v_3}(z)=\Phi_{v_3}(z)+\Phi_{u,v_3}(x,z)
    \end{align*}
\end{center}
\label{fig:secondstep}
\caption{Second step: build the modified subnetworks (here $\Graph(u,2,x)$); arrow represents modification of the potential function by incorporating interaction function into them}
\end{figure}

\begin{Theorem}[Cavity Recursion]\label{theorem:CavityThm}
For every $x\in \chi$,
\begin{equation}\label{eq:CavEq}
B_u(x)=\Phi_u(x)-\Phi_u(0)+ \sum_{j=1}^d \mu_{u \leftarrow v_j}(x,B_{\Graph(u,j,x),v_j})
\end{equation}
\end{Theorem}

\begin{proof}
For every $k=0,1,\ldots,d$, let $x_{j,k}=x$ when $j\le k$ and $=0$ otherwise. Let $\bold{v}=({v_1},\ldots,{v_d})$, and $\bold{z}=(z_{1},\ldots,z_{d})\in \chi^d$. We have
\begin{align*}
    B_{u}(x)=\Phi_u(x)-\Phi_u(0)
    &+\max_{\bold{z}} \Big \{\sum_{j=1}^{d} \Phi_{u,v_j}(x,z_j)+\JiA{\Graphsans{u}}{\bold{v}}{\bold{z}} \Big \} \\
    &-\max_{\bold{z}} \Big \{ \sum_{j=1}^{d} \Phi_{u,v_j}(0,z_j)+\JiA{\Graphsans{u}}{\bold{v}}{{\bold{z}}}  \Big \}.
\end{align*}
\comment{
=\Phi_u(x)-\Phi_u(0)
&+ \max_{\bold{z}}\Big \{\sum_{j=1}^{d} \Phi_{u,v_j}(\bold{x}^d(j),z_j)+\JiA{\Graphsans{u}}{\bold{v}}{\bold{z}} \Big \} \\
 &- \max_{\bold{z}} \Big \{ \sum_{j=1}^{d} \Phi_{u,v_j}(\bold{x}^0(j),z_j)+\JiA{\Graphsans{u}}{\bold{v}}{{x(\bold{v})}}
              \Big \}
\end{align*}
The variables $\bold{x}^j$ will allow us to construct a telescoping sum. Intuitively, the telescoping sum will be a sum of cavities, where the value at node $u$ will change from $0$ to $x$ incrementally (as opposed to a direct change as in the equations above).} The first step of the proof consists in considering the following telescoping sum (see figure \ref{fig:firststep}):
\begin{align}
     B_{u}(x)=
    \Phi_u(x)-\Phi_u(0)&+\sum_{k=1}^{d} \Bigg[\max_{\bold{z}} \Big \{\sum_{j=1}^{d} \Phi_{u,v_j}(x_{j,k},z_j)+
    \JiA{\Graphsans{u}}{\bold{v}}{\bold{z}} \Big \} \label{Cavity}\\
       \notag
               &-\max_{\bold{z}} \Big \{ \sum_{j=1}^{d} \Phi_{u,v_j}(x_{j,k-1},z_j)+
               \JiA{\Graphsans{u}}{\bold{v}}{{\bold{z}}}  \Big \} \Bigg]
\end{align}
and the $k^{\text{th}}$ difference:
\begin{align}\label{quick5}
    \max_{\bold{z}} \Big \{ \sum_{j=1}^{d} \Phi_{u,v_j}(x_{j,k},z_j)+\JiA{\Graphsans{u}}{\bold{v}}{\bold{z}}  \Big \}
-\max_{\bold{z}} &\Big \{ \sum_{j=1}^{d} \Phi_{u,v_j}(x_{j,k-1},z_j)+\JiA{\Graphsans{u}}{\bold{v}}{\bold{z}}  \Big \}
\end{align}
Let $\bold{z}_{-k}=(z_1,\ldots,z_{k-1},z_{k+1},\ldots,z_d)$. Then,
\begin{align}\label{quick6}
    &\max_{\bold{z}} \Big \{ \sum_{j=1}^{d} \Phi_{u,v_j}(x_{j,k},z_j)+\JiA{\Graphsans{u}}{\bold{v}}{\bold{z}}  \Big \}=\notag \\
    & \max_{z_k} \Big(\Phi_{u,v_k}(x,z_k) + \max_{\bold{z}_{-k}} \big \{ \sum_{j\leq k-1} \Phi_{u,v_j}(x,z_j)+\sum_{j\geq k+1} \Phi_{u,v_j}(0,z_j)+\JiA{\Graphsans{u}}{\bold{v}}{\bold{z}}  \big \} \Big )
\end{align}
Similarly,
\begin{align}\label{quick7}
    &\max_{\bold{z}} \Big \{ \sum_{j=1}^{d} \Phi_{u,v_j}(x_{j,k-1},z_j)+\JiA{\Graphsans{u}}{\bold{v}}{\bold{z}}  \Big \}=\notag \\
    & \max_{z_k}\Big( \Phi_{u,v_k}(0,z_k) + \max_{\bold{z}_{-k}} \big \{
    \sum_{j\leq k-1} \Phi_{u,v_j}(x,z_j)+\sum_{j\geq k+1} \Phi_{u,v_j}(0,z_j)+\JiA{\Graphsans{u}}{\bold{v}}{\bold{z}}  \big \}\Big)
\end{align}
For each $z_k$, we have (see figure \ref{fig:secondstep}):
$$\max_{\bold{z}_{-k}} \Big \{ \sum_{j\leq k-1}\Phi_{u,v_j}(x,z_j)+\sum_{j\geq k+1} \Phi_{u,v_j}(0,z_j)+\JiA{\Graphsans{u}}{\bold{v}}{\bold{z}}  \Big \}=J_{\Graph(u,k,x),v_k}(z_k)$$
By adding and substrating $J_{\Graph(u,k,x),v_k}(0)$, expression (\ref{quick5}) can therefore be rewritten as
\begin{eqnarray*}
     \max_{y} (\Phi_{u,v_k}(x,y) + B_{\Graph(u,k,x)}(y))- \max_{y} (\Phi_{u,v_k}(0,y) + B_{\Graph(u,k,x)}(y))
\end{eqnarray*}
which is exactly $\mu_{u \leftarrow v_k}(x,B_{\Graph(u,k,x)})$. Finally, we obtain: $$B_u(x)=\Phi_u(x)-\Phi_u(0)+ \sum_{k=1}^d \mu_{u \leftarrow v_k}(x,B_{\Graph(u,k,x),v_k})$$

\end{proof}

\subsection{Computation tree and the Cavity Expansion algorithm}\label{sec:CE}

Given a decision network $\Graph$, a node $u \in V$ with $\mathcal{N}_u=\{v_1,\ldots,v_d\}$, and $r\in \mathbb{Z}_+$,  introduce a vector $\text{CE}[\Graph,u,r]=(\text{CE} [\Graph,u,r,x], x\in \chi)\in \Real^T$ defined recursively as follows.

\begin{enumerate}
    \item $\text{CE}[\Graph,u,0,x]=0$
    \item For every  $r=1,2,\ldots$, and every $x\in \chi$,
    \begin{align}\label{eq:CEeq}
    \text{CE}[\Graph,u,r,x]=\Phi_u(x)-\Phi_u(0)+
    \sum_{j=1}^d \mu_{u \leftarrow v_j}\Big(x,\text{CE}[\Graph(u,j,x),v_j,r-1]\Big),
    \end{align}
\end{enumerate}
where $\Graph(u,k,x)$ is defined in Subsection~\ref{subsection:generalgraphs}, and the sum $\sum_{j=1}^d$ is equal to $0$ when $\mathcal{N}_u=\emptyset$. Note that from the definition of $\Graph(u,k,x)$, the definition and output of $\text{CE}[\Graph,u,r]$ depend on the order in which the neighbors $v_j$ of $u$ are considered. $\text{CE}[\Graph,u,r]$ serves as an $r$-step approximation, in some appropriate sense to be explained later, of the cavity vector $B_{\Graph,u}$. The motivation for this definition is relation (\ref{eq:CavEq}) of Theorem~\ref{theorem:CavityThm}. The local cavity approximation can be computed using an algorithm described below, which we call \emph{Cavity Expansion (CE)} algorithm.
\sectionline
{\label{Code:CE}\tt {Cavity Expansion: CE}[$\Graph,u,r,x$]\\
INPUT: A network $\Graph$, a node $u$ in $\Graph$, an action $x$ and a computation depth $r\geq 0$\\
BEGIN\\
If $r=0$ return $0$\\
else do\\
Find neighbors $\mathcal{N}(u)=\{v_1,v_2,\ldots,v_d\}$ of $u$ in $\Graph$.\\
If $\mathcal{N}(u)=\emptyset$,
return $\Phi_u(x)-\Phi_u(0)$. \\
Else\\
For each $j=1,\ldots, d$, construct the network $\Graph(u,j,x)$.\\
For each $j=1,\ldots, d$, and  $y\in\chi$, compute $\text{CE}[\Graph(u,j,x),v_j,r-1,y]$\\
For each $j=1,\ldots, d$, compute $\mu_{u\leftarrow v_j}(x,\text{CE}[\Graph(u,j,x),v_j,r-1,y])$\\
Return $\Phi_u(x)-\Phi_u(0) + \sum_{1\le j\le d} \mu_{u\leftarrow v_j}(x,\text{CE}[\Graph(u,j,x),v_j,r-1,y])$ as $\text{CE}[\Graph,u,r,x]$.
}
\sectionline

The algorithm above terminates because $r$ decreases by one at each recursive call of the algorithm. As a result, an initial call to $\text{CE}[\Graph,u,r,x]$ will result in a finite number of recursive calls to some $\text{CE}[\Graph_j,u_j,k_j,x_j]$, where $k_j<r$. Let $(\Graph_i,v_i,x_i)_{1\leq i \leq m}$ be the subset of arguments for the calls used in computing $\text{CE}[\Graph,u,r,x]$ for which $k_i=0$. In the algorithm above, the values returned for $r=0$ are $0$, but it can be generalized by choosing a value $\mathcal{C}_i$ for the call $\text{CE}[\Graph_i,v_i,0,x_i]$.

The set of values $\mathcal{C}=(\mathcal{C}_i)_{1\leq i \leq m}$ will be  called a \emph{boundary condition}. We denote by $\text{CE}[\Graph,u,r,x,\mathcal{C}]$ the output of the cavity algorithm with boundary condition $\mathcal{C}$. The interpretation of $\text{CE}[\Graph,u,r,x,\mathcal{C}]$ is that it is an estimate of the cavity $B_{\Graph,u}(x)$ via $r$ steps of recursion (\ref{eq:CavEqTree}) when the recursion is initialized by setting $\text{CE}[\Graph_i,u_i,0,x_i]=\mathcal{C}_i$ and is run $r$ steps. We will sometimes omit $\mathcal{C}$ from the notation when such specification is not necessary. Call $\mathcal{C^{*}}=(\mathcal{C}^{*}_i)\eqdef (B_{\Graph_i,v_i}(x_i))$  the ``true boundary condition''. The justification comes from the following proposition, the proof of which follows directly from Theorem~\ref{theorem:CavityThm}.

\begin{Proposition}
Given node $u$ and $\mathcal{N}(u)=\{v_1,\ldots,v_d\}$, suppose
for every $j=1,\ldots, d$ and $y\in \chi$,
$\text{CE}[\Graph(u,j,x),v_j,r-1,y]=B_{\Graph(u,j,x),v_j}(y)$; then,
$\text{CE}[\Graph,u,r,x]=B_{\Graph,u}(x)$.
\end{Proposition}

As a result, if $\mathcal{C}$ is the ``correct'' boundary condition, then $\text{CE}[\Graph,u,r,x,\mathcal{C}]=B_{\Graph,u}(x)$ for every $u,r,x$. The execution of the Cavity Expansion algorithm can be visualized as a computation on a tree, due to its recursive nature. This has some similarity with a computation tree associated with the performance of the Belief Propagation algorithm, \cite{tatikonda2002lbp,sanghavi2008mpm,bayati2008mpm}.  The important difference with \cite{tatikonda2002lbp}  is that the presence of cycles is incorporated via the construction $\Graph(u,j,x)$ (similarly to \cite{weitzCounting,jung2006ibp,bayati2007sda, gamarnik2007daa,gamarnik2007cda}.  As a result, the computation tree of the CE is finite (though often extremely large), as opposed to the BP computation tree.

An important lemma, which we will use frequently in the rest of the paper, states that in the computation tree of the cavity recursion, the cost function of an edge cost is statistically independent from the subtree below that edge.

\begin{Proposition}\label{IndependenceFixLemma}
Given  $u,x$ and  $\mathcal{N}(v)=\{v_1,\ldots,v_d\}$, for every $r,j=1,\ldots,d$ and $y \in \chi$, $CE[\Graph(u,j,x),v_j,r-1,y]$ and $\Phi_{u,v_j}$ are independent.
\end{Proposition}

Note however that $\Phi_{u,v_j}$ and $CE[\Graph(u,k,x),v_k,r-1,y]$ are generally dependent when $j\not = k$

\begin{proof}
The proposition follows from the fact that for any $j$, the interaction function $\Phi_{u,v_j}$ does not appear in $\Graph(u,j,x)$, because node $u$ does not belong to $\Graph(u,j,x)$), and does not modify the potential functions of $\Graph(u,j,x)$ in the step (\ref{eq:GraphModified}).
\end{proof}
Our last proposition analyzes the complexity of running the Cavity Expansion algorithm.

\begin{Proposition}\label{prop:Complexity}
For every $\Graph,u,r,x$, the value $\text{CE}[\Graph,u,r,x]$ can be computed in time $O\big(r (\Delta T)^r\big)$.
\end{Proposition}

\begin{proof}
The computation time required to construct the networks $\Graph(u,j,x)$, compute the messages $\mu_{u\leftarrow v_j}(x,\hat B_{v_j})$, and return $\Phi_u(x)-\Phi_u(0) + \sum_{1\le j\le d} \mu_{u\leftarrow v_j}(x,\hat B_{v_j})$, is $O(\Delta T)$. Let us prove by induction that that for any subnetwork $\Graph'$ of $\Graph$, $\text{CE}[\Graph',u,r,x]$ can be computed in time bounded by $O(r (\Delta T)^r)$.
The values for $r=1$ can be computed in time bounded by $M$, since $\Graph'$ is a subnet of $\Graph$ and therefore of smaller size.
For $r>1$, the computations of $\text{CE}[\Graph',u,r,x]$ requires a fixed cost of $O(\Delta T)$, as well as $(\Delta T)$ calls to $\text{CE}$ with depth $(r-1)$. The total cost is therefore bounded by $O(\Delta T+(\Delta T)\: (r-1) (\Delta T)^{r-1})$, which is $O(r(\Delta T)^r)$.
\end{proof}

\section{Correlation decay and decentralized optimization}\label{sec:corrdecay}

In this section, we investigate the relations between the correlation decay phenomenon and the existence of near-optimal decentralized decisions. \comment{Intuitively, correlation decay indicates a diminishing dependence between the cavity of a node $v$ and the structure of the network faraway from $v$.} When a network exhibits the correlation decay property, the cavity functions of faraway nodes are weakly related, implying a weak dependence between their optimal decisions as well.  Thus one can expect that good decentralized decisions exist.  We will show that this is indeed the case.

\begin{Definition}

Given a function $\rho(r)\ge 0, r\in\mathbb{Z}_+$ such that  $\lim_{r\rightarrow\infty}\rho(r)=0$, a decision network $\Graph$ is said to satisfy the correlation decay property with rate $\rho$ if for every two boundary conditions $\bound$, $\bound'$
\begin{align*}\max_{u,x}
\E\big|CE[\Graph,u,r,x,\bound]-CE[\Graph,u,r,x,\bound']\big|\leq \rho(r).
\end{align*}

If there exists $K_c>0$ and $\alpha_c<1$ such that $\rho(r)\leq K_c\alpha_c^r$ for all $r$, then we say that $\Graph$ satisfies the exponential correlation decay property with rate $\alpha_c$.

\end{Definition}

The correlation decay property implies that for every $u,x$,

\begin{align*}
\E\big|CE[\Graph,u,r,x]-B_{\Graph,u}(x)\big|\leq \rho(r).
\end{align*}

The following assumptions will be frequently used in future.
\begin{Assume}\label{asm}
For all $v\in V, x,y\in\chi$, $B_v(x)-B_v(y) $ is a continuous random variable with density  bounded above by a constant $g>0$.
\end{Assume}
We will also assume the costs functions are bounded in $L_2$ norm:
\begin{Assume}\label{asm2}
There exists $K_\Phi$ such that for any $e \in E$,
$\big(\sum_{x,y \in \chi} \E|\Phi_e(x,y)|^2\big)^{1/2}\leq {K_\Phi}$ and for any $v\in V$, $\big(\sum_{x \in \chi}\E|\Phi_v(x)|^2\big)^{1/2}\leq K_\Phi$
\end{Assume}
Assumption $1$ is designed to lead to the following two properties: (a) There is a unique optimal action in every node with probability 1. (b) The suboptimality gap between the optimal action and the second best action is large enough so that there is a ``clear winner'' among actions.

\subsection{Correlation decay implies near-optimal decentralized decisions}

Under Assumption $1$ let $\bold{x}=(x_v)_{v\in V}$ be the unique (with probability one) optimal solution for the network $\Graph$. For every $v\in V$, $x\in \chi$, let $x^r_v=\argmax_x CE[\Graph,v,r,x]$, ties broken arbitrarily, and $\bold{x}^r=(x^r_v)$. The main relation between correlation decay property, Cavity Expansion algorithm and the optimization problem is given by the following result.

\begin{Proposition}\label{prop_subopt}
Suppose $\Graph$ exhibits the correlation decay property with rate $\rho(r)$ and that Assumption \ref{asm} holds. Then,
\begin{eqnarray}\label{Suboptbest}
\pr(x^r_u\neq x_u) \leq 2T^2\sqrt{2g \rho(r)} ,\hspace{0.5cm} \forall u \in V, r\ge 1.
\end{eqnarray}
\end{Proposition}

\begin{proof}
For simplicity, let $B^r_u(x)$ denote $CE[\Graph,u,r,x]$. We will first prove that
\begin{eqnarray}\label{Subopt1}
    \pr(x^r_u\neq x_u) \leq T^2 ( g\epsilon +\frac{2 \rho(r)}{\epsilon})
\end{eqnarray}
The proposition will follow by choosing $\epsilon=\sqrt{2 \rho(r)g^{-1}}$. Consider a node $u$, and notice that if
\begin{align*}
(B_u(x)-B_u(y))(B^r_u(x)-B^r_u(y))>0, \qquad \forall x\ne y,
\end{align*}
then $x^r_u=x_u$. Indeed, since  $B_u(x_{u})-B_u(y)>0$ for all $y\not = x_{u}$, the property implies the same for $B^r_u$, and the assertion holds.
Thus, the event $\{x^r_u\neq x_u\}$ implies the event
$$\{\exists (x,y), y \neq x : (B_u(x)-B_u(y))(B^r_u(x)-B^r_u(y)) \leq 0\}$$
Fix  $\epsilon>0$ and note that for two real numbers $r$ and $s$, if  $|r|>\epsilon$ and $|r-s| \leq \epsilon$,
then $rs>0$. Applying this to $r=B_u(x)-B_u(y)$ and $s=B^r_u(x)-B^r_u(y)$, we find that the events $|B_u(x)-B_u(y)|> \epsilon$
and  $$(|B_u(x)-B^r_u(x)| < \epsilon/2) \cap (|B_u(y)-B^r_u(y)| < \epsilon/2)$$ jointly imply $$(B_u(x)-B_u(y))(B^r_u(x)-B^r_u(y)) > 0$$
Therefore, the event $(B_u(x)-B_u(y))(B^r_u(x)-B^r_u(y)) \leq 0$ implies
 $$\{|B_u(x)-B_u(y)|\leq \epsilon\} \cup \{|B_u(x)-B^r_u(x)|\geq\epsilon/2\}\cup \{|B_u(y)-B^r_u(y)|\geq\epsilon/2\}$$
Applying the union bound, for any two actions $x\neq y$,
\begin{align}\label{eq:3cases}
\pr\Big((B_u(x)-B_u(y))(B^r_u(x)-B^r_u(y))\leq 0\Big) &\leq \pr(|B_u(x)-B_u(y)|\leq \epsilon)+\pr(|B_u(x)-B^r_u(x)|\geq\epsilon/2)\notag\\
&+\pr(|B_u(y)-B^r_u(y)|\geq\epsilon/2).
\end{align}

Now $\pr(|B_u(x)-B_u(y)|\leq \epsilon)$ is at most $2g \epsilon$ by Assumption \ref{asm}. Using the Markov inequality, we find that the second summand in (\ref{eq:3cases}) is at most $2\E|B_u(x)-B^r_u(x)| / \epsilon \leq 2\rho(r) / \epsilon$. The same bound applies to the third summand. Finally, noting  there are $T(T-1)/2$ different pairs $(x,y)$ with $x\neq y$ and  applying the union bound, we obtain:
\begin{eqnarray*}
\pr(x^r_u\neq x_u) &\leq & (T(T-1)/2) (2g \epsilon+4\rho(r) / \epsilon)\\
&\leq&{T^2}(g\epsilon +\frac{2 \rho(r)}{\epsilon}).
\end{eqnarray*}
\end{proof}

For the special case of exponential correlation decay, we obtain the following result, the proof of which immediately  follows from Proposition~\ref{prop_subopt}.

\begin{Corollary}
Suppose $\Graph$ exhibits the exponential correlation decay property with rate $\alpha_c$, and suppose Assumption~\ref{asm} holds. Then
$$\pr(x^r_u\neq x_u) \leq 2T^2 \sqrt{2gK_c} \alpha_c ^{r/2}, \qquad \forall u\in V,r\ge 1.$$
In particular, for any $\epsilon>0$, if
$$r\geq 2 \frac{|\log K'_c|+|\log\epsilon|}{|\log(\alpha_c)|}$$
then
$$\pr(x^r_u\neq x_u) \leq \epsilon$$
where $K'_c=2T^2\sqrt{2gK_c}$
\end{Corollary}

In summary, correlation decay - and in particular fast (i.e. exponential) correlation decay - implies that the optimal action in a node depends with high probability only on the structure of the network in a small radius around the node. As in \cite{rusmevichientong2003ddm}, we call such a property \emph{decentralization} of optimal actions. Note that the radius required to achieve an $\epsilon$ error does not depend on the size of the entire network; moreover, for exponential correlation decay, it grows only as a logarithm of the accepted error.

The main caveat of Proposition $\ref{prop_subopt}$  is that the Assumption~\ref{asm} does not necessarily hold. For instance, it definitely does not apply to models with discrete random variables $\Phi_u$ and $\Phi_{u,v}$. In fact, assumption~\ref{asm} is not really necessary, as we show in an online appendix that a regularization technique allows to relax this assumption. Note that Assumption ~\ref{asm2} is not needed for Proposition \ref{prop_subopt} to hold.

\subsection{Correlation decay and efficient decentralized optimization}

Proposition \ref{prop_subopt} illustrates how optimal actions are decentralized under the correlation decay property. In this section, we use this result to show that the resulting optimization algorithm is both near-optimal and computationally efficient.

As before, let before $\bold{x}=(x_u)$ denote the optimal solution for the network $\Graph$, and let $\bold{x}^r=(x_u^r)$ be the decisions resulting from the Cavity Expansion algorithm with depth $r$. Let $\tilde{\bold{x}}=(\tilde x_u)$ denote (any) optimal solution for the perturbed network $\tilde\Graph$. Let $K_1=10 K_\Phi \,T  ( |V| +|E|)$, and $K_2= K_1 \, (g\,K_c)^{1/4}$, where $K_c$ is defined under the assumption of exponential correlation decay.

\begin{Theorem} \label{CorrMainThM}
Suppose a decision network $\Graph$ satisfies correlation decay
property with rate $\rho(r)$. Then, for all $r>0$
\begin{equation}\label{CorrMainEq}
\E[F(\bold{x})-F(\bold{x}^{r,\delta})] \leq K_1(g\rho(r))^{1/4}
\end{equation}
\end{Theorem}

\begin{Corollary}\label{CorrMainCor}
Suppose $\Graph$ exhibits exponential correlation decay property with rate $\alpha_c$.
Then, for any $\epsilon>0$, if $$r\geq \big( 8|\log \epsilon|+4|\log(K_2)| \big) |\log(\alpha_c)|^{-1}$$
then
$$\pr \big(F(\bold{x})-F(\bold{x}^{r})>\epsilon \big) \leq \epsilon$$
and $\bold x^{r}$ can be computed in time polynomial in $|V|$, $1/\epsilon$.
\end{Corollary}

\begin{proof}
By applying the union bound on Proposition~\ref{prop_subopt}, for every $(u,v)$, we have:
$\pr\big((x^{r}_u,x^{r}_v) \neq ({x}_u,{x}_v)\big) \leq 4T^2\sqrt{2g \rho(r)}$. We have
$$\E|{F}({\bold{x}})-{F}(\bold{x}^{r})|\leq  \sum_{u \in V} \E|{\Phi}_u({x}_u)-{\Phi}_u(x^{r}_u)| +
\sum_{(u,v) \in E} \E|\Phi_{u,v}({x}_u,{x}_v)-\Phi_{u,v} (x^{r}_u,x^{r}_v)| $$
For any $u,v \in V$,
\begin{align*}
    \E[\Phi_{u,v}({x}_u,{x}_v)-\Phi_{u,v} (x^{r}_u,x^{r}_v)]
    &\leq \E\Big[ 1_{(x^{r}_u,x^{r}_v) \neq ({x}_u,{x}_v)}\, \Big(\big|\Phi_{u,v}({x}_u,{x}_v)\big|+\big|\Phi_{u,v}(x^{r}_u,x^{r}_v)\big|\Big)\Big]\\
    &\leq 2K_\Phi \, \pr\big((x^{r}_u,x^{r}_v) \neq ({x}_u,{x}_v)\big)^{1/2}\\
    &\leq 4 K_\Phi T \,({2g\rho(r)})^{1/4}
\end{align*}
where the second inequality follows from Cauchy-Schwarz. Similarly, for any $u$ we have
\begin{align*}
\E|{\Phi}_u({x}_u)-{\Phi}_u(x^{r}_u)| \leq \leq& 4 K_\Phi T\, {(2g \rho(r))}^{1/4}
\end{align*}
By summing over all nodes and edges, we get: $\E[{F}({\bold{x}})-{F}(\bold{x}^{r}) \leq 8K_\Phi\, T (2g\rho(r))^{1/4}\leq K_1 (g\rho(r))^{1/4}$, and equation \eqref{CorrMainEq} follows. The corollary is then proved using Markov Inequality; injecting the definition of exponential correlation decay into equation \eqref{CorrMainEq}, we obtain
\begin{align*}
\pr(J_\Graph-F(\hat x) \ge \epsilon) \le E[J_\Graph-F(\hat x)]/\epsilon \le K_2 \alpha^{r/4}/\epsilon
\end{align*}
Since $r\geq (4|\log(K_2)|+8 |\log(\epsilon)|) |\log(\alpha)|^{-1}$, we have $K_2 \alpha^{r/4}\le \epsilon^2$ and the result follows.

\end{proof}

\section{Establishing the correlation decay property. Coupling technique}
\label{sec:coupling}
The previous section motivates the search for conditions implying the correlation decay property. This section is devoted to the study of a coupling argument which can be used to show that correlation decay holds.
Results in this section are for the case $|\chi|=2$. They can be extended to the case $|\chi|\geq 2$ at the expense of heavier notations, but not much additional insight gain. For this special case $\chi=\{0,1\}$, we introduce a set of simplifying notations as follows.

\subsection{Notations}\label{subsec:Notations}

Given $\Graph=(V,E,\Phi,\{0,1\})$ and $u\in V$, let $v_1,\ldots,v_d$ be the neighbors of $u$ in $V$. For any $r>0$ and boundary conditions $\bound$, $\bound'$, define:
\begin{enumerate}
\item $B(r)\eqdef \text{CE}[\Graph,u,r,1,\bound]$ and $B'(r)\eqdef \text{CE}[\Graph,u,r,1,\bound'] $
\item
    For $j=1,\ldots d$, let $\Graph_j=\Graph(u,j,1)$, and let
$B_j(r-1) \eqdef \text{CE}[\Graph_j,v_j,r-1,1,\bound]$ and
$B'_j(r-1) \eqdef \text{CE}[\Graph_j,v_j,r-1,1,\bound']$. Also let $\bold B(r-1)=(B_j(r-1))_{1\leq j \leq d}$ and  $\bold B'(r-1)=(B'_j(r-1))_{1\leq j \leq d}$

\item For $k=1,\ldots n_j$, let $(v_{j1},\ldots,v_{j n_j})$ be the neighbors of $v_j$ in $\Graph_j$, and let
$B_{jk}(r-2)=\text{CE}[\Graph_j(v_{j},k,1),v_{jk},r-2,1,\bound]$
and $B'_{jk}(r-2)=\text{CE}[\Graph_j(v_{j},k,1),v_{j},r-2,1,\bound']$ for all $k=1\ldots n_j$. Also let $\bold{B_j}(r-2)=(B_{jk}(r-2))_{1\leq k\leq n_j}$ and $\bold{B_j'}(r-2)=(B'_{jk}(r-2))_{1\leq k\leq n_j}$.
\item For simplicity, since $1$ is the only action different from the reference action $0$, we denote $\mu_{u \leftarrow v_j}(z) \eqdef \mu_{u\leftarrow v_j}(1,z)$.\\ From equation (\ref{MuDefine}), note the following alternative expression for $\mu_{u\leftarrow v_j}(z)$
\begin{align}\label{nicerformbeta}
\mu_{u\leftarrow v_j}(z)= \Phi_{u,v_j}(1,1)-\Phi_{u,v_j}(0,1)+&\max(\Phi_{u,v_j}(1,0)-\Phi_{u,v_j}(1,1),z)\\ -& \max(\Phi_{u,v_j}(0,0)-\Phi_{u,v_j}(0,1),z) \notag
\end{align}

\item Similarly, for any $j=1\ldots d$ and $k=1\ldots n_j$, let $\mu_{v_j\leftarrow v_{jk}}(z)\eqdef \mu_{v_j\leftarrow v_{jk}}(1,z)$.

\item For any $\bold{z}=(z_1,\ldots,z_d)$, let $\mu_{u}(\bold{z})=\sum_j \mu_{u\leftarrow v_j}(z_j)$. Also, for any $j$, and any $\bold{z}=(z_1,\ldots,z_{n_j})$, let $\mu_{v_j}(\bold{z})=\sum_{1\leq k\leq n_j} \mu_{v_j\leftarrow v_{jk}}(z_k)$.

\item For any directed edge $e=(u \leftarrow v)$, denote
\begin{eqnarray*}
\Phi_e^1&\eqdef&\Phi_{u,v}(1,0)-\Phi_{u,v}(1,1)\\
\Phi_e^2&\eqdef& \Phi_{u,v}(0,0)-\Phi_{u,v}(0,1)\\
\Phi_e^3&\eqdef&\Phi_{u,v}(1,1)-\Phi_{u,v}(0,1)\\
X_e&\eqdef& \Phi_e^1+\Phi_e^2\\
Y_e&\eqdef&\Phi^2_e-\Phi^1_e=\Phi_{u,v}(1,1)-\Phi_{u,v}(1,0)-\Phi_{u,v}(0,1)+\Phi_{u,v}(0,0)
\end{eqnarray*}
Note that $Y_{u \leftarrow v}=Y_{v \leftarrow u}$, so we simply denote it $Y_{u,v}$.

\end{enumerate}
Note that for any $e$, $\E|Y_e|\leq K_\Phi$ (see Assumption $2$).
Equation \eqref{eq:CEeq} can be rewritten as
\begin{eqnarray}\label{eq:muu:a}
B(r)&=&\mu_u(\bold B (r-1))+ \Phi_u(1)-\Phi_u(0) \\
B'(r)&=&\mu_u(\bold B' (r-1)) + \Phi_u(1)-\Phi_u(0)\label{eq:muu:b}
\label{Cavitybin2}
\end{eqnarray}
Similarly, we have
\begin{eqnarray}\label{eq:muu2:a}
B_j(r-1)=\mu_{v_j}(\bold{B_j}(r-2))+ \Phi_{v_j}(1)-\Phi_{v_j}(0)\\
B'_j(r-1)=\mu_{v_j}(\bold{B'_j}(r-2))+ \Phi_{v_j}(1)-\Phi_{v_j}(0)\label{eq:muu2:b}
\end{eqnarray}
Finally,
equation \eqref{nicerformbeta} can be rewritten
\begin{eqnarray}\label{nicerform}
\mu_{u\leftarrow v}(z) &=& \Phi_{u\leftarrow v}^3+\max(\Phi_{u\leftarrow v}^1,z)-\max(\Phi_{u\leftarrow v}^2,z)
\end{eqnarray}

\comment{
We call $Y_e$ the \emph{interaction coupling}; }$Y_e$  represents how strongly the interaction function $\Phi_{u,v}(x_u,x_v)$ is ``coupling" the variables $x_u$ and $x_v$. In particular, if $Y_e$ is zero,  the interaction function $\Phi_{u,v}(x_u,x_v)$ can be decomposed into a sum of two potential functions $\Phi_u(x_u)+\Phi_{v}(x_{v})$, that is, the edge between $u$ and $v$ is then be superfluous and can be removed. To see why this is the case, take $\Phi_u(0)=0$, $\Phi_u(1)=\Phi_{u,v}(1,0)-\Phi_{u,v}(0,0)$, $\Phi_{v}(0)=\Phi_{u,v}(0,0)$ and $\Phi_{v}(1)=\Phi_{u,v}(0,1)$, which is also equal to $\Phi_{u,v}(1,1)-\Phi_{u,v}(1,0)+\Phi_{u,v}(0,0)$, since $Y_e=0$.

\comment{Suppose now we have two sets $(B_j)_{j=1\ldots d},(B'_j)_{j=1\ldots d}$ of values for the cavities at the children $(v_1,\ldots,v_d)$. Both sets are obtained through cavity recursions which differ only by their boundary conditions (i.e. the value returned by $CE$ when $r=0$). For instance, one can choose $B_j$ to be the optimal cavity $B_j$ and $B'_j$ to be the output of the cavity algorithm with boundary condition equal to $0$.}

\subsection{Distance-dependent coupling and correlation decay}

\comment{ In this section, we present a sufficient condition  for
correlation decay. The condition depends on the parameters of a
particular form of coupling: For any neighbor $v_j$ of $u$, it is
possible that the partial cavities $\mu_{u\leftarrow v_j}$ and $\mu'_j$ depending on
two different boundary conditions $\bound$ and $\bound'$ be equal
even when $B_j\not = B'_j$. The probability that this coupling
occurs decreases as the distance between $B_j$ and $B_j'$ grows
bigger.}

\begin{Definition}\label{DDcouple} A network $\Graph$ is said to exhibit $(a,b)$-coupling with parameters $(a,b)$ if for
every edge $e=(u,v)$, and every two real values $x$, $x'$:
\begin{align}\label{DDcouple-eq}
\pr\Big(\mu_{u\leftarrow v}(x+\Phi_{v}(1)-\Phi_{v}(0)) =  \mu_{u\leftarrow v}(x'+\Phi_{v}(1)-\Phi_{v}(0))\Big) \geq (1-a)-b |x-x'|
\end{align}
\end{Definition}
The probability above, and hence the coupling parameters, depend on both $\Phi_v(1)-\Phi_v(0)$ and the values $\Phi_{u,v}(x,y)$. Note that if for all $x,x'$
\begin{align}\label{DDcouple-eq2}
\pr\Big(\mu_{u\leftarrow v}(x) =  \mu_{u\leftarrow v}(x')\Big) \geq (1-a)-b |x-x'|
\end{align}
then $\Graph$ exhibits $(a,b)$ coupling, but in general the tightest coupling values found for equation \eqref{DDcouple-eq2} are much weaker than the ones we would find by analyzing condition \eqref{DDcouple-eq}.
This form of distance dependent coupling is a useful tool in proving that correlation decay occurs, as illustrated by the following theorem:

\begin{Theorem}\label{PropCouplingCorr}
	Suppose $\Graph$ exhibits $(a,b)$-coupling. If
    \begin{align}\label{FirstCondition}
a(\Delta-1)+ \sqrt{b K_\Phi} (\Delta-1)^{3/2}<1
\end{align}
    then the exponential correlation decay property holds with $K=\Delta^2 K_\Phi$ and $\alpha=a(\Delta-1)+ \sqrt{b K_\Phi} (\Delta-1)^{3/2}$.

\vspace{5 mm}

\noindent Suppose $\Graph$ exhibits $(a,b)$-coupling and that there exists $K_Y>0$ such that $|Y_e|\leq K_Y$ with probability $1$. If
        \begin{align}\label{ThirdCondition}
a(\Delta-1)+bK_Y (\Delta-1)^2 <1
\end{align}
then the exponential correlation decay property holds with $\alpha=a(\Delta-1)+bK_Y (\Delta-1)^2$
\comment{
\vspace{5 mm}

\noindent
Suppose $\Graph$ exhibits $(a,b)$-coupling, that the network is locally tree-like ($\mathcal{B}(u,r)$ is a tree for every $u\in U$ and depth $r$ used for the cavity recursion) and that for all edges $e=(u,v)\in E$, the random variables $(\Phi_e(x,y))_{x,y \in \{0,1\}^2}$ are i.i.d. If
\begin{align}\label{SecondCondition}
(\Delta-1)(a +\sqrt{bK_\Phi})<1
\end{align}
then the exponential correlation decay property holds with $\alpha=(\Delta-1)(a +\sqrt{bK_\Phi})$.
}
\end{Theorem}

\subsubsection{Proof of Theorem \ref{PropCouplingCorr}}\label{MainProof}

We begin by proving several useful lemmas.

\begin{Lemma}\label{ContLemma}
For every $(u,v)$, and every two real values $x,x'$
\begin{equation}\label{continuity}
|\mu_{u\leftarrow v}(x)-\mu_{u\leftarrow v}(x')|\leq |x-x'|.
\end{equation}
\end{Lemma}

\begin{proof}
From (\ref{nicerformbeta}) we obtain
\begin{align*}
\mu_{u\leftarrow v}(x)-\mu_{u\leftarrow v}(x') & =  \max\Big(\Phi_{u,v}(1,0)-\Phi_{u,v}(1,1),x\Big)
-\max\Big(\Phi_{u,v}(0,0)-\Phi_{u,v}(0,1),x\Big)\\
        &   -\max\Big(\Phi_{u,v}(1,0)-\Phi_{u,v}(1,1),x'\Big) +\max\Big(\Phi_{u,v}(0,0)-\Phi_{u,v}(0,1),x'\Big).
\end{align*}
Using twice the relation $\max_x f(x)-\max_x g(x)\leq \max_x (f(x)-g(x))$, we obtain:
\begin{align*}
\mu_{u\leftarrow v}(x)-\mu_{u\leftarrow v}(x') &\leq \max(0,x-x')+\max(0,x'-x) \\
& =|x-x'|
\end{align*}
The other inequality is proved similarly.
\end{proof}

\begin{Lemma}\label{CouplLemma}
For every $u,v\in V$ and every two real values $x,x'$
\begin{equation}\label{funccomp}
|\mu_{u\leftarrow v}(x)-\mu_{u\leftarrow v}(x')|\leq |Y_{u,v}|
\end{equation}
\end{Lemma}
\begin{proof}
\comment{Let $Y_j=\Phi_{u,v_j}(1,1)-\Phi_{u,v_j}(1,0)-\Phi_{u,v_j}(0,1)+\Phi_{u,v_j}(0,0)$.}
Using $\eqref{nicerformbeta}$ and $\eqref{Cavitybin2}$, we have
\begin{align*}
\mu_{u\leftarrow v}(x)-(\Phi_{u,v}(1,1)-\Phi_{u,v}(0,1)) &=
\max(\Phi_{u,v}(1,0)-\Phi_{u,v}(1,1),x) \\
&-\max(\Phi_{u,vx}(0,0)-\Phi_{u,v}(0,1),x).
\end{align*}
By using the relation $\max_x f(x)-\max_x g(x)\leq \max_x (f(x)-g(x))$ on the right hand side, we obtain
\begin{align*}
\mu_{u\leftarrow v}(x)-(\Phi_{u,v}(1,1)-\Phi_{u,v}(0,1)) &\leq \max(0,-Y_{u,v}).
\end{align*}
Similarly
\begin{align*}
-\mu_{u\leftarrow v}(x')+(\Phi_{u,v}(1,1)-\Phi_{u,v}(0,1)) \leq \max(0,Y_{u,v}).
\end{align*}
Adding up
\begin{align*}
\mu_{u\leftarrow v}(x)-\mu_{u\leftarrow v}(x') \leq |Y_{u,v}|.
\end{align*}
The other inequality is also proven similarly.
\end{proof}

\comment{
In particular, this implies that if $Y_{u,v_j}=0$ then $\mu_{u\leftarrow v_j}=\mu'_j$.
Distance-dependent coupling induces a form of continuity
between the cavity difference $|B-B'|$ at node $u$ and the cavity differences $|B_j-B'_j|$ at the children of $u$.
}

\begin{Lemma}\label{CorrLemma}
    Suppose $(a,b)$-coupling holds. Then,
\begin{align}\label{Corr}
    \E|B(r)-B'(r)|\leq a \sum_{1\le j\le d} \E|B_j(r-1)-B'_j(r-1)|+b \sum_{1\le j\le d} \E\big[|B_j(r-1)-B'_j(r-1)|^2\big].
\end{align}
\end{Lemma}

\begin{proof}
Using \eqref{eq:CEeq}, we obtain:
\begin{align*}
\E|B(r)-B'(r)| &= \E\Big[\big| \Phi_u(1)-\Phi_u(0) +\sum_j \mu_{u\leftarrow v_j}(B_j(r-1)) - (\Phi_u(1)-\Phi_u(0))
-\sum_j \mu_{u\leftarrow v_j}(B_j'(r-1))\big|\Big] \\
&\leq  \sum_j \E  \big|\mu_{u\leftarrow v_j}(B_j(r-1))-\mu_{u\leftarrow v_j}(B_j'(r-1))\big|\\
 &= \sum_j  \mathbb {E} \Big [  \E\big[|\mu_{u\leftarrow v_j}(B_j(r-1))-\mu_{u\leftarrow v_j}(B_j'(r-1))|\big| \mu_{v_j}(\bold{B_j}(r-2),\mu_{v_j}(\bold{B'_j}(r-2)\big] \Big]
\end{align*}
By Lemma \ref{ContLemma}, we have $|\mu_{u\leftarrow v_j}(B_j(r-1)) - \mu_{u\leftarrow v_j}(B_j'(r-1))| \leq  |B_j(r-1)-B_j'(r-1)|$.
Also note from that from equation \eqref{eq:muu2:a} and \eqref{eq:muu2:b}, $|B_j(r-1)-B_j'(r-1)|=|\mu_{v_j}(\bold{B_j}(r-2))-\mu_{v_j}(\bold{B'_j}(r-2))|$; hence conditional on both  $\mu_{v_j}(\bold{B_j}(r-2)$ and $\mu_{v_j}(\bold{B'_j}(r-2)$, $|B_j(r-1)-B_j'(r-1)|$ is a constant.
 Therefore,
\begin{align}\label{Corr1}
\notag \E\Big[\big|\mu_{u\leftarrow v_j}(B_j(r-1)) - \mu_{u\leftarrow v_j}(B_j'(r-1))\big|\:\:\Big|  \mu_{v_j}(\bold{B_j}(r-2),\mu_{v_j}(\bold{B'_j}(r-2) \Big] \\
\leq|B_j(r-1)-B_j'(r-1)| \: \pr(\mu_{u\leftarrow v_j}(B_j(r-1))\not=\mu_{u\leftarrow v_j}(B_j'(r-1))\:|\:\mu_{v_j}(\bold{B_j}(r-2),\mu_{v_j}(\bold{B'_j}(r-2))
\end{align}
Note that in the (a,b) coupling definition, the probability is over the values of the functions $\Phi_{u,v_j}$, and $\Phi_v$. By proposition \ref{IndependenceFixLemma}, these are independent from $\mu_{v_j}(\bold{B_j}(r-2)$ and $\mu_{v_j}(\bold{B'_j}(r-2))$. Thus, by the (a,b) coupling assumption, $\pr(\mu_{u\leftarrow v_j}(B_j(r-1))\not=\mu_{u\leftarrow v_j}(B_j'(r-1))\:|\:\mu_{v_j}(\bold{B_j}(r-2),\mu_{v_j}(\bold{B'_j}(r-2))\leq a + b |B_j(r-1)-B_j'(r-1)|$.
The result then follows.
\end{proof}

Fix an arbitrary node $u$ in $\Graph$. Let $\mathcal{N}(u)=\{v_1,\ldots,v_d\}$. Let $d_j=|\mathcal{N}(v_j)|-1$ be the number of neighbors of $v_j$ in $\Graph$ other than $u$ for $j=1,\ldots,d$. We need to establish that for every two boundary condition $\bound,\bound'$
\begin{align}\label{aim}
\E|\text{CE}(\Graph,u,r,\bound)-\text{CE}(\Graph,u,r,\bound')|\leq K\alpha^r
\end{align}
We first establish the bound inductively for the case $d\leq \Delta-1$. Let $e_d$ denote the supremum of the left-hand side of \eqref{aim}, where the supremum is over all networks $\Graph'$ with degree at most $\Delta$, such that the corresponding constant $K_{\Phi'}\leq K_{\Phi}$, over all nodes $u$ in $\Graph$ with degree $|\mathcal{N}(u)|\leq \Delta-1$ and all over all choices of boundary conditions $\bound,\bound'$. Each condition corresponds to a different recursive inequality for $e_r$.

\paragraph{Condition \eqref{FirstCondition}}

Under \eqref{FirstCondition}, we claim that
\begin{align}\label{recursive-er}
e_r\leq a(\Delta-1)e_{r-1}+b(\Delta-1)^3 K_{\Phi} e_{r-2}
\end{align}
Applying \eqref{eq:muu2:a} and \eqref{eq:muu2:b}, we have
\begin{align*}
|B_j(r-1)-B'_j(r-1)|\leq \sum_{1\leq k\leq d_j} |\mu_{v_j\leftarrow v_{jk}}(B_{jk}(r-2))-\mu_{v_j\leftarrow v_{jk}}(B'_{jk}(r-2))|
\end{align*}
Thus,
\begin{align*}
|B_j(r-1)-B'_j(r-1)|^2\leq \Big(\sum_{1\leq k\leq d_j} |\mu_{v_j\leftarrow v_{jk}}(B_{jk}(r-2))-\mu_{v_j\leftarrow v_{jk}}(B'_{jk}(r-2))|\Big)^2\\
\leq d_j \sum_{1\leq k\leq d_j} |\mu_{v_j\leftarrow v_{jk}}(B_{jk}(r-2))-\mu_{v_j\leftarrow v_{jk}}(B'_{jk}(r-2))|^2
\end{align*}
By Lemmas \ref{ContLemma} and \ref{CouplLemma}
we have $|\mu_{v_j\leftarrow v_{jk}}(B_{jk}(r-2))-\mu_{v_j\leftarrow v_{jk}}(B'_{jk}(r-2))|\leq |B_{jk}(r-2)-B'_{jk}(r-2)|$ and $|\mu_{v_j\leftarrow v_{jk}}(B_{jk}(r-2))-\mu_{v_j\leftarrow v_{jk}}(B'_{jk}(r-2))|\leq |Y_{jk}|$. Also, $d_j\leq \Delta-1$.Therefore,
\begin{align}
|B_j(r-1)-B'_j(r-1)|^2 \leq (\Delta-1) \sum_{1\leq k\leq d_j} |B_{jk}(r-2)-B'_{jk}(r-2)|\:.\: |Y_{jk}|
\end{align}
By Proposition \ref{IndependenceFixLemma}, the random variables $|B_{jk}(r-2)-B'_{jk}(r-2)|$ and $|Y_{jk}|$ are independent. We obtain:
\begin{align}
\E |B_j(r-1)-B'_j(r-1)|^2 \leq& (\Delta-1) \sum_{1\leq k\leq d_j} \E|B_{jk}(r-2)-B'_{jk}(r-2)|\:.\: \E|Y_{jk}|\\
\leq& (\Delta-1) K_\Phi (\sum_{1\leq k \leq d_j} \E|B_{jk}(r-2)-B'_{jk}(r-2)|)\notag\\
\leq& (\Delta-1)^2 K_\Phi e_{r-2}\notag
\end{align}
where the second inequality follows from the definition of $K_\Phi$ and the third inequality follows from the definition of $e_r$ and the fact that the neighbors $v_{jk}$, $1\leq k \leq d_j$ of $v_j$ have degrees at most $\Delta-1$ in the corresponding networks for which $B_{jk}(r-2)$ and $B'_{jk}(r-2)$ were defined.
Applying Lemma \ref{CorrLemma} and the definition of $e_r$, we obtain
\begin{align*}
\E|B(r)-B'(r)|\leq& \: a \sum_{1 \leq j \leq d} \E|B_j(r-1)-B'_j(r-1)| + b \sum_{1\leq j \leq d} \E\big[|B_j(r-1)-B'_j(r-1)|^2\big]\\
\leq& \: a(\Delta-1) e_{r-1} + b (\Delta-1)^3 K_\Phi e_{r-2}
\end{align*}
This implies \eqref{recursive-er}.

From \eqref{recursive-er} we obtain that $e_r\leq K\alpha^r$ for $K=\Delta K_\Phi $ and $\alpha$ given as the largest in absolute value root of the quadratic equation $\alpha^2=a(\Delta-1) \alpha + b(\Delta-1)^3 K_\Phi$. We find this root to be
\begin{align*}
a=&\:\frac{1}{2}(a(\Delta-1)+\sqrt{a^2(\Delta-1)^2+4b(\Delta-1)^3K_\Phi})\\
\leq&\:a(\Delta-1)+\sqrt{b (\Delta-1)^3 K_\Phi}\\
<& \: 1
\end{align*}
where the last inequality follows from assumption \eqref{FirstCondition}. This completes the proof for the case that the degree $d$ of $u$ is at most $\Delta-1$.

Now suppose $d=|\mathcal{N}(u)|=\Delta$. Applying \eqref{eq:muu:a} and \eqref{eq:muu:b} we have
\begin{align*}
|B(r)-B'(r)|\leq \sum_{1\leq j \leq d} |\mu_{u\leftarrow v_j}(B_j(r-1)-\mu_{u\leftarrow v_j}(B'_j(r-1))|
\end{align*}
Applying again Lemma \ref{ContLemma}, the right-hand side is at most
$$\sum_{1\leq j \leq d} |B_j(r-1)-B'_j(r-1)| \leq \Delta e_{r-1}$$
since $B_j(r-1)$ and $B'_j(r-1)$ are defined for $v_j$ in a subnetwork $\Graph_j=\Graph(u,j,1)$, where $v_j$ has degree at most $\Delta-1$. Thus again the correlation decay property holds for $u$ with $\Delta K$ replacing $K$.

\paragraph{Condition \eqref{ThirdCondition}}

Recall from lemma \ref{CorrLemma} that for all $r$, we have:
\begin{align*}
    \E|B(r)-B'(r)|\leq a \sum_{1\le j\le d} \E|B_j(r-1)-B'_j(r-1)|+b \sum_{1\le j\le d} \E\big[|B_j(r-1)-B'_j(r-1)|^2\big].
\end{align*}
For all $j$, $|B_j(r-1)-B'_j(r-1)|=|\sum_k (\mu_{v_j\leftarrow v_{jk}}(B_{jk})- \mu_{v_j\leftarrow v_{jk}}(B_{jk}'))|$. Moreover, for each $j,k$, $|\mu_{v_j\leftarrow v_{jk}}(B_{jk})-\mu_{v_j\leftarrow v_{jk}}(B'_{jk})|\leq |Y_{jk}| \leq K_Y$ (the second inequality follows from Lemma \ref{CouplLemma}, the third by assumption).
As a result,
$$|B_j(r-1)-B_j'(r-1)|^2\leq (\Delta-1)K_Y|B_j(r-1)-B_j(r-1)|$$
We obtain:
\begin{align*}
    e_r\leq (a+bK_Y(\Delta-1))\:(\Delta-1)e_{r-1}
\end{align*}

Since $a(\Delta-1)+bK_Y(\Delta-1)^2<1$, $e_r$ goes to zero exponentially fast. The same reasoning as previously shows that this property implies correlation decay.

\subsection{Establishing coupling bounds}
\subsubsection{Coupling Lemma}

Theorem~\ref{PropCouplingCorr} details sufficient condition under which the distance-dependent coupling induces correlation decay (and thus efficient decentralized algorithms, vis-\`a-vis Proposition \ref{prop:Complexity} and Theorem \ref{CorrMainThM}). It remains to show how can we prove coupling bounds. The following simple observation can be used to achieve this goal.

\noindent For any edge $(u,v)\in \Graph$, and any two real numbers $x,x'$, consider the following events
$$E^+_{u\leftarrow v}(x,x')=\{\min(x,x')+\Phi_v(1)-\Phi_v(0)\geq \max(\Phi^1_{u\leftarrow v},\Phi^2_{u\leftarrow v})\}$$
$$E^-_{u\leftarrow v}(x,x')=\{\max(x,x')+\Phi_v(1)-\Phi_v(0)\leq \min(\Phi^1_{u\leftarrow v},\Phi^2_{u\leftarrow v})\}$$
$$E_{u\leftarrow v}(x,x')=E^+_{u,v}(x,x') \cup E^-_{u,v}(x,x')$$

\begin{Lemma}\label{LemmaCoupling}
If $E_{u\leftarrow v}(x,x')$ occurs, then $\mu_{u\leftarrow v}(x+\Phi_v(1)-\Phi_v(0))=\mu_{u\leftarrow v}(x'+\Phi_v(1)-\Phi_v(0))$.
Therefore $$P(\mu_{u\leftarrow v}(x+\Phi_v(1)-\Phi_v(0))=\mu_{u\leftarrow v}(x'+\Phi_v(1)-\Phi_v(0))\geq P(E_{u\leftarrow v}(x,x'))$$
\end{Lemma}

\begin{proof}
From representation \eqref{nicerform}, we have $\mu_{u\leftarrow v}(x) = \Phi_{u\leftarrow v}^3+\max(\Phi_{u\leftarrow v}^1,z)-\max(\Phi_{u\leftarrow v}^2,z)$; let $x,x'$ be any two reals. If both $x$ and $x'$ are greater than both $\Phi_{u\leftarrow v}^1$ and $Phi_{u\leftarrow v}^2$, then $\mu_{u\leftarrow v}(x)=\Phi_{u\leftarrow v}^3=\mu_{u\leftarrow v}(x')$. If both $x$ and $x'$ are smaller than both $\Phi_{u\leftarrow v}^1$ and $Phi_{u\leftarrow v}^2$, then $\mu_{u\leftarrow v}(x)=\Phi_{u\leftarrow v}^3+\Phi_{u\leftarrow v}^1-\Phi_{u\leftarrow v}^2=\mu_{u\leftarrow v}(x')$. The result follows from applying the above observation to $x+\Phi_v(1)-\Phi_v(0)$ and $x'+\Phi_v(1)-\Phi_v(0)$.
\end{proof}

Note that Lemma \ref{LemmaCoupling} implies that the probability of coupling not occuring $P(\mu_{u\leftarrow v}(x+\Phi_v(1)-\Phi_v(0))\not = \mu_{u\leftarrow v}(x'+\Phi_v(1)-\Phi_v(0)))$ is upper bounded by the probability of $(E_{u\leftarrow v}(x,x'))^{c}$.
When obvious from context, we drop the subscript $u\leftarrow v$. We will often use the following description of $(E_(x,x'))^{c}$: for two real values $x\geq x'$,
\begin{eqnarray}\label{eq:Ecxx}
(E(x,x'))^c=\{\min(\Phi^1,\Phi^2)+ \Phi_{v}(0)-\Phi_v(1) < x < \max(\Phi^1,\Phi^2)+\Phi_v(0)-\Phi_v(1)+x-x'\}
\end{eqnarray}

\comment{
Recall that $\widehat {\mu_{u\leftarrow v}}(B)=\Phi^3_{u \leftarrow v}+\max(\Phi^1_{u\leftarrow v}+\partial \Phi_{v},z)-\max(\Phi^2_{u\leftarrow v}+\Phi_v(0)-\Phi_v(1),z)$.
Therefore, if $E^+$ occurs, $\widehat{\mu_{u\leftarrow v}}(B)=\widehat{\mu_{u\leftarrow v}}(B')=A^3$; if $E^-$ occurs, $\widehat{\mu_{u\leftarrow v}}(B)=\widehat{\mu_{u\leftarrow v}}(B')=A^3+\Phi^1_{u\leftarrow v}-\Phi^2_{u\leftarrow v}$.
The Lemma follows.
}

\subsubsection{Uniform Distribution. Proof of Theorem \ref{thm:Uniform}}

In order to prove Theorem \ref{thm:Uniform}, we compute the coupling parameters $a,b$ for this distribution and apply the second form of Theorem \ref{PropCouplingCorr}.

\begin{Lemma}\label{UnifLemma}
The network with uniformly distributed rewards described in section \ref{sec:resunifom}  exhibits $(a,b)$ coupling with $a=\frac{I_2}{2I_1}$ and $b=\frac{1}{2I_1}$.
\end{Lemma}

\begin{proof}
For any fixed edge $(u,v)\in \Graph$, $\Phi_{u\leftarrow v}^1$ and $\Phi_{u\leftarrow v}^2$ are i.i.d. random variables with a triangular distribution with support $[-2I_2,2I_2]$ and mode $0$. Because $\Phi_{u\leftarrow v}^1$ and $\Phi_{u\leftarrow v}^2$ are i.i.d., by symmetry we obtain:
\begin{align*}
\pr(&(E(x,x'))^c)=\\
&= 2\int_{-2I_2}^{2I_2} d\pr_{\Phi^1}(a_1) \int_{a_1}^{2I_2}  \: d\pr_{\Phi^2} (a_2)\: P(a_1+\Phi_v(0)-\Phi_v(1) < x <\Phi_v(0)-\Phi_v(1)+a_2+x-x')\\
&=2 \int_{-2I_2}^{2I_2}d\pr_{\Phi^1}(a_1) \int_{a_1}^{2I_2}  \: d\pr_{\Phi^2} (a_2)\: P(x'-a_2 < \Phi_v(0)-\Phi_v(1) < x-a_1)
\end{align*}
$P(x'-a_2 < \Phi_v(0)-\Phi_v(1) < x-a_1)$ can be upper bounded by $\frac{a_2-a_1+x-x'}{2I_1}$, and we obtain:
\begin{eqnarray*}
P(E(x,x')^c) &\leq& \frac{x-x'}{2I_1} +\frac{1}{I_1} \int_{-2I_2}^{2I_2} d\pr_{\Phi^1}(a_1) \int_{a_1}^{2I_2}  \: d\pr_{\Phi^2} (a_2) (a_2-a_1)
\end{eqnarray*}

Note that $d\pr_{\Phi^2} (a_2)=\frac{1}{4I_2^2}(a_2+2I_2) d(a_2)$ for $a_2\leq 0$, and $d\pr_{\Phi^2} (a_2)=\frac{1}{4I_2^2}(2I_2-a_2) d(a_2)$ for $a_2\geq 0$; identical expressions hold for $d\pr_{\Phi^1} (a_1)$.
Therefore, for $a_1\geq 0$,
\begin{eqnarray*}
\int_{a_1}^{2I_2}  \: d\pr_{\Phi^2} (a_2) (a_2-a_1)&=&\frac{1}{4I_2^2}\int_{a_1}^{2I_2} (2I_2-a_2) (a_2-a_1)\: d(a_2)\\
&=&\frac{1}{4I_2^2}\big( -\int_{a_1}^{2I_2} (2I_2-a_2)^2 d(a_2)+(2I_2-a_1)\int_{a_1}^{2I_2}(2I_2-a_2)d(a_2)  \big)\\
&=&\frac{1}{4I_2^2}\big( -\frac{1}{3}(2I_2-a_1)^3+\frac{1}{2} (2I_2-a_1)^3\big)=\frac{1}{24I_2^2}(2I_2-a_1)^3\\
\end{eqnarray*}
Similarly, for $a_1\leq 0$,
\begin{eqnarray*}
&&\int_{a_1}^{2I_2}  \: d\pr_{\Phi^2} (a_2) (a_2-a_1)=\comment{\\
&&\int_{a_1}^{0}  \: d\pr_{\Phi^2} (a_2) (a_2-a_1)+\int_{0}^{2I_2}  \: d\pr_{\Phi^2} (a_2) (a_2-a_1)=\\
&&\int_{a_1}^{0}  \: d\pr_{\Phi^2} (a_2) (a_2-a_1)+\frac{1}{3}I_2-\frac{1}{2}a_1=\\
&&\int_{a_1}^{0}  \: \frac{1}{4I_2^2}(a_2+2I_2)(a_2-a_1)d(a_2)+\frac{1}{3}I_2-\frac{1}{2}a_1=\\
&&\int_{a_1}^{0}  \: \frac{1}{4I_2^2}(a_2+2I_2)^2d(a_2)-(a_1+2I_2) \int_{a_1}^{0}  \: \frac{1}{4I_2^2}(a_2+2I_2)d(a_2)+\frac{1}{3}I_2-\frac{1}{2}a_1=\\
&&\frac{1}{12I_2^2}((2I_2)^3-(a_1+2I_2)^3)-(a_1+2I_2) \frac{1}{8I_2^2}(4I_2^2-(a_1+2I_2)^2)+\frac{1}{3}I_2-\frac{1}{2}a_1=\\
&&} -a_1+\frac{1}{24I_2^2} (a_1+2I_2)^3
\end{eqnarray*}
The final integral is therefore equal to:
\begin{eqnarray*}
 &&\int_{-2I_2}^{2I_2} d\pr_{\Phi^1}(a_1) \int_{a_1}^{2I_2}  \: d\pr_{\Phi^2} (a_2) (a_2-a_1)\\
 &=&\frac{1}{4I_2^2}\Big(\int_{-2I_2}^0\big((a_1+2I_2)(-a_1+\frac{1}{24I_2^2}(a_1+2I_2)^3\big)d(a_1)+\int_{0}^{2I_2} \frac{1}{24I_2^2}(2I_2-a_1)^4 d(a_1)\Big)\\
  &=&\frac{1}{4I_2^2}\Big(\frac{24}{15}I_2^3+\frac{4}{15}I_2^3\Big)=\frac{7}{15}I_2
\end{eqnarray*}
Finally,
\begin{eqnarray*}
P((E(x,x')^c)\leq \frac{7I_2}{15I_1}+\frac{|x-x'|}{2I_1}\leq \frac{I_2}{2I_1}+\frac{|x-x'|}{2I_1}
\end{eqnarray*}
Therefore, the system exhibits coupling with parameters $(\frac{I_2}{2I_1},\frac{1}{2I_1})$.
\end{proof}

We can now finish the proof of Theorem \ref{thm:Uniform}. For all $(u,v)\in E$ and $x,y \in \chi$, $|\Phi_{u,v}(x,y)|\leq I_2$. Therefore, for any $(u,v)$, $|Y_{u,v}|=|\Phi_{u,v}(1,1)-\Phi_{u,v}(0,1)-\Phi_{u,v}(1,0)+\Phi_{u,v}(0,0)|\leq 4 I_2$.

Note that for all edges, $|Y_e|\leq 4I_2$, so that the condition $\beta (\Delta-1)^{2}<1$ implies $\frac{I_2}{2I_1} (\Delta-1)+\frac{4I_2}{2I_1}(\Delta-1)^{2}<1$. Since $(\Delta-1)\leq (\Delta-1)^2$, if $\beta (\Delta-1)^2<1$ we also have $\frac{I_2}{2I_1} (\Delta-1)+ \frac{4 I_2}{2 I_1} (\Delta-1)^2<1$. This is exactly condition \eqref{ThirdCondition} with $a,b$ as given by
Lemma \ref{UnifLemma} and $K_Y=4 I_2$. It follows that $\Graph$ exhibits exponential correlation decay, and since Assumptions $1$ and $2$ hold, all conditions of Corollary \ref{CorrMainCor} are satisfied, and there exists an additive FPTAS for computing $J_\Graph$.

\subsubsection{Gaussian distribution. Proof of Theorem \ref{thm:Gauss}}

In this section, we compute the coupling parameters for Gaussian distributed reward functions. Rather than considering only the assumptions of Theorem \ref{thm:Gauss}, we place ourselves in a more general framework. The proof will then follow from the application of Theorem \ref{PropCouplingCorr} (first condition) and a special case of the computation detailed below (see Corollary \ref{Cor:gaussiid}). Assume that for every edge $e=(u,v)$ the value functions $(\Phi_{u,v}(0,0),\Phi_{u,v}(0,1),\Phi_{u,v}(1,0),\Phi_{u,v}(1,1))$ are  independent, identically distributed four-dimensional Gaussian random variables, with mean $\bold{\mu}=(\mu_i)_{i \in \{ 00,01,10,11\}}$, and covariance matrix $S=(S_{ij})_{i,j \in \{00,01,10,11\}}$.
For every node $v\in V$, suppose $\Phi_v(1)=0$ and that $\Phi_v(0)$ is a Gaussian random variable with mean $\mu_p$ and standard deviation $\sigma_p$.
Moreover, suppose all the $\Phi_v$ and $\Phi_e$ are independent for $v\in V$, $e \in E$.
Let
\begin{align*}
\sigma_1^2&=S_{10,10}-2S_{10,11}+S_{11,11}+\sigma_p^2 &\sigma_2^2=&S_{00,00}-2S_{00,01}+S_{01,01}+\sigma_p^2 \\
\rho&=(\sigma_1\sigma_2)^{-1} (S_{00,10}-S_{00,11}-S_{01,10}+S_{01,11}+\sigma_p^2) &C=&\frac{\sigma_2^2-\sigma_1^2}{\sqrt{(\sigma_1^2+\sigma_2^2)^2-4\rho^2\sigma_1^2\sigma_2^2}} \\
\sigma_X^2&=\sigma_1^2+\sigma_2^2+2\rho \sigma_1\sigma_2  &\sigma_Y^2=&\sigma_1^2+\sigma_2^2-2\rho \sigma_1\sigma_2
\end{align*}
\begin{Proposition}\label{prop:Gaussiancase}
Assume $C<1$. Then the network exhibits coupling with parameters $(a,b)$ equal to:
\begin{align*}
a=&\frac{1}{\pi}\arctan\Big(\sqrt{\frac{1}{1-C^2}}\frac{\sigma_Y}{\sigma_X}\Big)+\sqrt{\frac{2}{\pi}}\frac{|\mu_{00}+\mu_{11}-\mu_{10}-\mu_{01}|}{\sigma_X}\\
b=& \sqrt{\frac{2}{\pi}}\frac{1}{\sigma_X}
\end{align*}
\end{Proposition}

\begin{Corollary}\label{Cor:gaussiid}
Suppose that for each $e$,$({\Phi}_e(0,0),{\Phi}_e(0,1),{\Phi}_e(1,0),{\Phi}_e(1,1))$ are i.i.d. Gaussian variables with mean $0$ and standard deviation $\sigma_e$.
Let $\beta=\sqrt{\frac{\sigma_e^2}{\sigma_e^2+\sigma_p^2}}$
Then $a\leq \beta$ and $b K_{\Phi}\leq \beta$.
\end{Corollary}
\begin{proof}
Under the conditions of corollary \ref{Cor:gaussiid}, we have $\sigma_{Y}^2=4\sigma_{e}^{2}$, $\sigma_{X}^2=4\sigma_{p}^{2}+4\sigma_{e}^{2}$, and $C=0$. Note also that $K_\Phi\leq 2 \sigma_e$
By Proposition~\ref{prop:Gaussiancase}, the network exhibits coupling with parameters
\begin{align*}
a=&\frac{1}{\pi}\arctan\Big(\sqrt{\frac{\sigma_e^2}{\sigma_e^2+\sigma_p^2}}\Big)\leq \frac{1}{\pi} \beta \leq \beta \\
b=& \sqrt{\frac{1}{2\pi}}\frac{1}{\sqrt{\sigma_e^2+\sigma_p^2}} \text{       and so,      } b K_\Phi \leq \sqrt{\frac{2}{\pi}} \beta \leq \beta
\end{align*}
\end{proof}

Remark that when $\sigma_e\rightarrow 0$, $\beta\rightarrow 0$ and correlation decay takes place; moreover, combining Corollary \ref{Cor:gaussiid} and Theorem \ref{PropCouplingCorr} (condition \eqref{FirstCondition}) directly yields Theorem \ref{thm:Gauss}.

\begin{proof}[Proof of Proposition \ref{prop:Gaussiancase} ]

Fix an edge $(u,v)$ in $E$; for simplicity, in the rest of this section denote $\bar{\Phi}^1=\Phi^1_{u\leftarrow v}+\Phi_v(0)-\Phi_v(1)$ and $\bar{\Phi}^2=\Phi^2_{u\leftarrow v}+\Phi_v(0)-\Phi_v(1)$.
It follows that $(\overline{\Phi}^1,\overline{\Phi}^2)$ follows a bivariate Gaussian distribution with mean $(\mu_1,\mu_2)$:
\begin{align*}
\mu_1=\mu_{10}-\mu_{11}+\mu_p \text{ and } \mu_2=\mu_{00}-\mu_{01}+\mu_p
\end{align*}
and covariance matrix
\[ S_A=\left( \begin{array}{cc}
\sigma_1^2 & \rho \sigma_1 \sigma_2 \\
\rho \sigma_1 \sigma_2 & \sigma_2^2    \end{array} \right)\]
Let $X=\overline{\Phi}^1+\overline{\Phi}^2$, $Y=\overline{\Phi}^2-\overline{\Phi}^1$. Then, $(X,Y)$ is a bivariate Gaussian vector with means $\E[X]=\mu_1+\mu_2$ and $\E[Y]=\mu_2-\mu_1$, standard deviations $\sigma_X,\sigma_Y$ and correlation $C$ as defined previously.
Denote also $\overline{X}\eqdef X-E[X]$ and $\overline{Y}\eqdef Y-E[Y]$ the centered versions of $X$ and $Y$. Consider two real numbers $x\geq x'$, and let $(b,t)$ be the two real numbers such that $x=b+t/2$, $x'=b-t/2$. From equation \eqref{eq:Ecxx}, we have
\begin{equation*}
(E(x,x'))^c=\{\min(\overline{\Phi}^1,\overline{\Phi}^2)-t/2<b<\max(\overline{\Phi}^1,\overline{\Phi}^2)+t/2\}
\end{equation*}
The first step of the proof consists in rewriting the event $(E(x,x'))^c$ in terms of the variables ${X},{Y}$:
\begin{Lemma}
\begin{equation*}
(E(x,x'))^c=\{| Y| \geq |X-2b|-t \}
\end{equation*}
\end{Lemma}
\begin{proof}
\begin{align*}
(E(x,x'))^c=&\{\min(\overline{\Phi}^1,\overline{\Phi}^2)-t/2<b<\max(\overline{\Phi}^1,\overline{\Phi}^2)+t/2\}\\
=&\{\overline{\Phi}^1-t/2<b<\overline{\Phi}^2+t/2,	 \overline{\Phi}^1\leq\overline{\Phi}^2\}\:\cup\: \{\overline{\Phi}^2-t/2<b<\overline{\Phi}^1+t/2,Y\leq 0,	 \overline{\Phi}^2\leq\overline{\Phi}^1\}\\
=&\{2\overline{\Phi}^1-t<2b<2\overline{\Phi}^2+t,	 \overline{\Phi}^1\leq\overline{\Phi}^2\}\:\cup\: \{2\overline{\Phi}^2-t<2b<2\overline{\Phi}^1+t,	 \overline{\Phi}^2\leq\overline{\Phi}^1\}\\
=&\{X-Y-t<2b<X+Y+t,Y\geq0\}\:\cup\: \{X+Y-t<2b<X-Y+t,Y\leq0\}\\
=&\{(X-2b)-|Y|-t<0<(X-2b)+|Y|+t\}\\
=&\{|Y|\geq(X-2b-t)\}\cap\{|Y|\geq(2b-X-t)\}\\
=&\{|Y|\geq |X-2b|-t\}
\end{align*}	
\end{proof}
For any $b$ and $t\geq0$, let $S(t)=\{x,y: |y|\geq |x|-t\}$, and for any real $x$, let $S(t,y)=\{x:|y|\geq |x|-t\}$. Note $S(t,y)$ is symmetric and convex in $x$ for all $y$.
Using the lemma, we obtain:
\begin{align}
\pr((E)^c(x,x'))&=\frac{1}{2\pi \sigma_x \sigma_y\sqrt{1-C^2}} \int_{S(t)} \exp(-\frac{1}{2(1-C^2)}(\frac{(x-\mu_1-\mu_2+2b)^2}{\sigma_x^2} + \frac{(y-\mu_2+\mu_1)^2}{ \sigma_y^2}\notag \\
& -2C \frac{(x-\mu_1-\mu_2+2b) (y+\mu_2-\mu_1)}{\sigma_x\sigma_y})) dx dy\notag\\
&=\frac{1}{2\pi \sigma_x \sigma_y\sqrt{1-C^2}} \int_{y} \exp(-\frac{1}{2(1-C^2)}\frac{(y-\mu_2+\mu_1)^2}{ \sigma_y^2})
\:g(y) dy\label{AndersonFirstStep}
\end{align}
where:
\begin{align*}
g(y)= \int_{x\in S(t,y)} \exp(-\frac{1}{2(1-C^2)}(\frac{(x-\mu_1-\mu_2+2b)^2}{\sigma_x^2} -2C \frac{(x-\mu_1-\mu_2+2b) (y-\mu_2+\mu_1)}{\sigma_x\sigma_y})) dx
\end{align*}
Let $\tilde{x_b}=\frac{(x-\mu_1-\mu_2+2b)}{\sigma_x} $ and $\tilde{y}= \frac{(y-\mu_2+\mu_1)}{ \sigma_y}$.
Then:
\begin{align*}
g(y)=& \: \exp(\frac{C^2}{2(1-C^2)} \tilde{y}^2)   \int_{x\in S(t,y)} \exp(-\frac{1}{2(1-C^2)}(\tilde{x_b}-C \tilde{y})^2 ) dx
\end{align*}
Now, $\tilde{x_b}-C \tilde{y}=\frac{x-\mu_1-\mu_2+2b - \frac{C \sigma_x (y-\mu_2+\mu_1)}{\sigma_y}}{\sigma_x}$.
Recall Anderson's inequality~\cite{dudley1999ucl}: let $\gamma$ be a centered Gaussian measure on $\mathbb{R}^k$, and $S$ be a convex, symmetric subset of $\mathbb{R}^k$. Then, for all $z$, $\gamma(S)\geq \gamma(S+z)$. Since $S(t,y)$ is a convex symmetric subset, by setting $2b=\mu_1+\mu_2+\frac{C \sigma_x (y-\mu_2+\mu_1)}{\sigma_y}$, it follows that
\begin{align*}
g(y)\leq \: \exp(\frac{C^2}{2(1-C^2)}\tilde{y}^2)   \int_{x\in S(t,y)} \exp(-\frac{1}{2\sigma_x^2(1-C^2)}x^2)dx
\end{align*}

Injecting that bound in equation (\ref{AndersonFirstStep}), we obtain:
\begin{align*}
\pr((E)^c(x,x')) \leq &\frac{1}{2\pi \sigma_x \sigma_y\sqrt{1-C^2}} \int_{y} \exp(-\frac{1}{2(1-C^2)}\frac{(y-\mu_2+\mu_1)^2}{ \sigma_y^2}) \\
&\Big( \exp(\frac{C^2}{2(1-C^2)} \frac{(y-\mu_2+\mu_1)^2}{ \sigma_y^2})   \int_{x\in S(t,y)} \exp(-\frac{1}{2\sigma_x^2(1-C^2)}x^2 ) dx  \Big ) dy\\
\leq & \frac{1}{2\pi \sigma_x \sigma_y\sqrt{1-C^2}} \int_{S(t)} \exp(-\frac{1}{2(1-C^2)}(\frac{x^2}{\sigma_x^2} +(1-C^2) \frac{(y-\mu_2+\mu_1)^2}{ \sigma_y^2})) dx dy
\end{align*}
Finally, note that the triangular inequality, for any $\alpha$ we have $S(t)\subset S_\alpha(t)\eqdef\{(x,y): |y-\alpha|\geq|x|-t-|\alpha|\}$.
We obtain:
\begin{align*}
\pr((E)^c(x,x')) \leq & \frac{1}{2\pi \sigma_x \sigma_y\sqrt{1-C^2}} \int_{S_{\mu_2-\mu_1}(t)} \exp(-\frac{1}{2(1-C^2)}(\frac{x^2}{\sigma_x^2} +(1-C^2) \frac{(y-\mu_2+\mu_1)^2}{ \sigma_y^2})) dx dy\\
\leq & \frac{1}{2\pi \sigma_x \sigma_y\sqrt{1-C^2}} \int_{S(t+|\mu_2-\mu_1|)} \exp(-\frac{1}{2(1-C^2)}(\frac{x^2}{\sigma_x^2} + (1-C^2)\frac{y^2}{ \sigma_y^2}))dx dy
\end{align*}
where the second inequality follows from a simple change of variable.
Let $t'=t+|\mu_2-\mu_1|$
Finally, we decompose $S(t')$ as the union of two sets: $S(t)=S_\text{int}(T)\cup S_\text{out}(t)$, where:
\begin{align*}
S_\text{int}(t')=&\{(X,Y) : |X|<t'\}\\
S_\text{out}(t')=&\{(X,Y) : |X| \geq t' \text{ and } |Y| \geq (|X|-t')\},
\end{align*}
and note that $S_\text{int}(t')\cap S_\text{out}(t')=\emptyset$. We have:
\begin{align*}
\pr(S_\text{int}(t'))\leq &\frac{2t'}{\sqrt{2 \pi(1-C^2) } \sigma_x}
\end{align*}
and, by symmetry of $S_\text{out}(t')$ in $X$ and $Y$,
\begin{align*}
\pr(S_\text{out}(t'))=&4 \pr(\{(x,y) : x\geq t, y\geq0, y \geq x-t\})\\
 = & \frac{2}{\pi \sigma_x \sigma_y\sqrt{1-C^2}} \int_{\{(x,y) : x\geq t, y\geq0, y \geq x-t\}} \:\: \exp(-\frac{1}{2(1-C^2)}(\frac{x^2}{\sigma_x^2} + (1-C^2)\frac{y^2}{ \sigma_y^2}))\:dx dy
\end{align*}
Using the change of variables $(x',y')=(\frac{x-t}{\sqrt{1-C^2}\sigma_x},\frac{y}{\sigma_y})$,we get:
\begin{align*}
\pr(S_\text{out}(t')) = & \frac{2}{\pi} \int_{\{(x',y') : x'>0, y'>0, y' \geq \frac{\sigma_x\:\sqrt{1-C^2}}{\sigma_y}x'\}} \Big( \exp(-(x'+\frac{t'}{\sqrt{1-C^2}\:\sigma_x})^2 - y'^2) \Big)dx' dy'
\end{align*}
Since $(x'+\frac{t'}{\sqrt{1-C^2}\:\sigma_x})^2 \geq x'^2$, it follows that:
\begin{align*}
\pr(S_\text{out}(t')) \leq & \frac{2}{\pi}  \int_{\{(x',y') : x'>0, y'>0, y' \geq \frac{\sigma_x\:\sqrt{1-C^2}}{\sigma_y}x'\}} \Big( \exp(-x'^2 - y'^2) \Big)dx dy
\end{align*}
By using a radial change of variables $(x',y')=(r\cos(\theta),r\sin(\theta))$ we can compute exactly the expression above, and find:
\begin{align*}
\pr(S_\text{out}(t')) \leq & \frac{2}{\pi} \int_{\{(r,\theta) : r>0, \arctan(\frac{\sigma_x\sqrt{1-C^2}}{\sigma_y}) \leq \theta \leq \frac{\pi}{2}\}}  \exp(-r^2) rdr d\theta\\
=&\frac{1}{\pi} \arctan(\frac{\sigma_y}{\sigma_x\:\sqrt{1-C^2}})
\end{align*}

\begin{align}\label{generalcouple}
\pr((E)^c(x,x')) \leq \Bigg (\frac{1}{\pi} \arctan(\frac{\sigma_y}{\sigma_x \sqrt{1-C^2}})+\sqrt{\frac{2}{\pi (1-C^2)}} \frac{|\mu_2-\mu_1|}{\sigma_x}\Bigg)+\sqrt{\frac{2}{\pi (1-C^2)}} \frac{t}{\sigma_x}
\end{align}
which gives us the desired bounds on $(a,b)$.
\end{proof}

\section{Maximum weighted independent sets}\label{section:MWIS}

\subsection{Cavity expansion and the algorithm}\label{section:algorithm}
In this section, we show how the correlation decay framework also applies to MWIS problems and prove theorems $3$, $4$, and $5$. There are additional challenges in achieving this goal. First, the bounded costs assumption required for the results of section \ref{sec:corrdecay}  does not hold for constrained optimization problems, as the underlying problem has infinite costs. Second, the coupling technique of section \ref{sec:coupling} is not readily applicable for MWIS. We therefore develop a different approach.

As for unconstrained optimization problems, we follow three steps. First, we detail the Cavity Expansion algorithm. Second, we establish the correlation decay property. Finally, we show that the correlation decay property implies that near-optimal, decentralized optimization can be performed in polynomial time.

Consider a general weighted graph $\Graph=(V,E,W)$, where $(V,E)$ is a graph whose nodes are equipped with arbitrary non-negative weights $W_i, i\in V$; no probabilistic assumption on $W_i$ is adopted yet. Note that for Independent Sets problems, we have $ J_{\Graph}=W(I^*)$, and for any $(i_1,\ldots,i_d)$, $J_{\Graph,(i_1,\ldots,i_d)}(\bold 0)=J_{\Graph\setminus\{i_1,\ldots,i_d\}}$, where $\Graph\setminus\{i_1,\ldots,i_d\}$ is the subgraph induced by nodes $V\setminus\{i_1,\ldots,i_d\}$. Consider  a given node $i\in V$ and let $N(i)=\{i_1,\ldots,i_d\}$. From Theorem \ref{theorem:CavityThm}, we have
\begin{align}
B_{\Graph,i}=J_{\Graph,i}(1)-J_{\Graph,i}(0)=W_i+\sum_l \mu_{i \leftarrow i_l}(1,B_{\Graph(i,l),i_l})
\end{align}
Recall that for MWIS, we have $\Phi_e(x,y)=-\infty$ for $(x,y)=(1,1)$ and $\Phi_e(x,y)=0$, otherwise. Therefore, by definition of $\mu_{i\leftarrow i_l}$, we have
\begin{align*}
\mu_{i \leftarrow i_l}(1,B_{\Graph(i,l),i_l})=&\max(-\infty+B_{\Graph(i,l),i_l},0)-\max(B_{\Graph(i,l),i_l},0)=-\max(B_{\Graph(i,l),i_l},0)
\end{align*}
Thus, \begin{align*}
B_{\Graph,i}=W_i-\sum_{l=1}^d \max(B_{\Graph(i,l),i_l},0)
\end{align*}

Let $l\le d$; recall the definition of $\Graph(i,l)$: $\Graph(i,l)$ is the network $\Graph\setminus\{i\}$, where the potential functions of the neighbors of $i$ have been modified as follows:
\begin{itemize}
\item for $v\in \{i_1,\ldots,i_{l-1}$, $\phi'_v(0)=\phi_v(0)+\phi_{i,v}(1,0)=0$, and $\phi'_v(1)=W_v+\phi_{i,v}(1,1)=W_v-\infty=-\infty$. Since the new weight of $v$ is $-\infty$, it is equivalent to removing this node from the graph.
\item for $v\in \{i_{l+1},\ldots,i_d$,$\phi'_v(0)=\phi_v(0)+\phi_{i,v}(0,0)=0$, and $\phi'_v(1)=W_v+\phi_{i,v}(0,1)=W_v$
\end{itemize}

We thus observe that in $\Graph(i,l)$, the nodes $\{i,i_1,\ldots,i_{l-1}\}$ can be removed, while the weights of nodes $\{i_{l+1},\ldots,i_d\}$ are unchanged; equivalently, we have $\Graph(i,l)=\Graph\setminus\{i,i_1,\ldots,i_{l-1}\}$. Therefore, we obtain
 \begin{align*}
B_{\Graph,i}=W_i-\sum_{l=1}^d \max(B_{\Graph\setminus\{i,i_1,\ldots,i_{l-1}\}},0)
\end{align*}

We further modify the cavity recursion by the following change of variable: for any graph $\Graph$ and node $i$, let $C_\Graph(i)= \max(B_{\Graph,i},0)$; note we have $C_\Graph(i)=\max(J_{\Graph,i}(1),J_{\Graph,i}(0))-J_{\Graph,i}(0)=J_\Graph-J_{\Graph\setminus\{i\}}$. The variables $C$ will be called \emph{cavities}. It turns out that in the case of IS problems, working with cavities $C$ is more convenient  than with cavities $B$. We obtain the cavity recursion for MWIS:
\begin{Proposition}\label{prop:cavityrecursion}
For any $i\in V$, let $N(i)=\{i_1,\ldots,i_d\}$. Then
\begin{align}\label{eq:cavityrecursion}
C_{\Graph}(i)=\max\Big(0,W_i-\sum_{1\le l\le d}C_{\Graphsans{i,i_1,\ldots,i_{l-1}}}(i_l)\Big),
\end{align}
where $\sum_{1\le l\le d}=0$ when $N(i)=\emptyset$.
If $W_i-\sum_{1\le l\le d}C_{\Graph\setminus\{i,i_1,\ldots,i_{l-1}\}}(i_l)>0$, namely $C_{\Graph}(i)>0$, then every largest weight independent set must contain  $i$. Similarly
if $W_i-\sum_{1\le l\le d}C_{\Graph\setminus\{i,i_1,\ldots,i_{l-1}\}}(i_l)<0$, implying $C_{\Graph}(i)=0$, then every largest weight independent set does not contain  $i$.
\end{Proposition}

\remark The proposition leaves out a "fuzzy" case $W_i-\sum_{1\le l\le d}C_{\Graph\setminus\{i,i_1,\ldots,i_{l-1}\}}(i_l)=0$. This will not be a problem in our setting since, due to the continuity of the weight distribution, the probability of this event is zero. Modulo this tie, the event $C_{\Graph}(i)>0\:(C_{\Graph}(i)=0)$ determine  whether $i$ must (must not) belong to the largest weighted independent set.

Using the special form of the cavity recursion \eqref{eq:cavityrecursion}, the cavity expansion algorithm for MWIS is very similar as the one defined in section \ref{sec:CE}. For any induced subgraph $\gH$ of $\Graph$ and node $i$, let $C^-_\gH(i,r)=\max(0,\text{CE}[\gH,i,r])$ with boundary condition $\text{CE}[\gH,i,0]=0$, and let $C^+_\gH(i,r)$ be the same quantity for the boundary condition $\text{CE}[\gH,i,0]=W_i$. Alternatively, $C^-$ and $C^+$ can be defined by the following recursions:
\begin{align}\label{eq:cavityrecursionapproximate}
C_{\gH}^-(i,r)=\left\{
                \begin{array}{ll}
                  0, & \hbox{$r=0$;} \\
                  \max\Big(0,W_i-\sum_{1\le l\le d}C^-_{\gH\setminus\{i,i_1,\ldots,i_{l-1}\}}(i_l,r-1)\Big), & \hbox{$r\ge 1$.}
                \end{array}
              \right.
\end{align}
\begin{align}\label{eq:cavityrecursionapproximate2}
C_{\gH}^+(i,r)=\left\{
                \begin{array}{ll}
                  W_i, & \hbox{$r=0$;} \\
                  \max\Big(0,W_i-\sum_{1\le l\le d}C^+_{\gH\setminus\{i,i_1,\ldots,i_{l-1}\}}(i_l,r-1)\Big), & \hbox{$r\ge 1$.}
                \end{array}
              \right.
\end{align}

The two boundaries condition were chosen so that $C_{\gH}^-(i,r)$ and $C_{\gH}^+(i,r)$ provide valid bounds on the true cavities $C_{\gH}(i)$, as detailed by the following Lemma.

\begin{Lemma}\label{lemma:Bbounds}
For every even $r$
\begin{align*}
C_{\gH}^-(i,r)\le C_{\gH}(i)\le C_{\gH}^+(i,r),
\end{align*}
and for every odd $r$
\begin{align*}
C_{\gH}^+(i,r)\le C_{\gH}(i)\le C_{\gH}^-(i,r),
\end{align*}
\end{Lemma}

\begin{proof}
The proof is by induction in $r$. The assertion holds by definition of $C^-,C^+$ for $t=0$. The induction
follows from (\ref{eq:cavityrecursion}), definitions of $C^-,C^+$ and
since the function $x\rightarrow\max(0,W-x)$ is non-increasing.
\end{proof}


We now describe our algorithm for producing a large weighted independent set. Our algorithm runs in two stages. Fix $\epsilon>0$. In the first stage we take an input graph $\Graph=(V,E)$ and delete every node (and incident edges) with probability $\epsilon^2/16$, independently for all nodes. We denote the resulting (random) subgraph by ${\Graph(\epsilon)}$. In the second stage we compute $C_{{\Graph(\epsilon)}}^-(i,r)$ for every node $i$ for the graph ${\Graph(\epsilon)}$ for some target even number of steps $r$. We set $\I(r,\epsilon)=\{i:C_{{\Graph(\epsilon)}}^-(i,r)>0\}$. Let $I^*_\epsilon$ be the largest weighted independent set of ${\Graph(\epsilon)}$.
\begin{Lemma}\label{lem:feasibleIS}
$\I(r,\epsilon)$ is an independent set.
\end{Lemma}

\begin{proof}
By Lemma~\ref{lemma:Bbounds}, if $C_{{\Graph(\epsilon)}}^-(i,r)>0$ then $C_{{\Graph(\epsilon)}}>0$, and therefore $\I\subset I^*_\epsilon$. Thus our algorithm produces an independent set in ${\Graph(\epsilon)}$ and therefore in $\Graph$.
\end{proof}

We finish this section by mentioning that due to Proposition~\ref{prop:Complexity}, the complexity of running both stages of the algorithm is $O(nr\Delta^r)$. As it will be apparent from the analysis, we could take $C_{{\Graph(\epsilon)}}^+$ instead of $C_{{\Graph(\epsilon)}}^-$ and arrive at the same result using an odd number $r$.

\subsection{Proof of Theorem~\ref{theorem:ISMainResult}}\label{section:proofmainresult}

\subsubsection{Correlation decay property}

The main bulk of the proof of Theorem~\ref{theorem:ISMainResult} will be to show that $\I(r,\epsilon)$ is close to $I^*_\epsilon$ in the set-theoretic sense. We will use this to show that $W(\I(r,\epsilon))$ is close to $W(I^*_\epsilon)$. It will be then straightforward to show that $W(I^*_\epsilon)$ is close to $W(I^*)$, which will finally give us the desired result, theorem $5$. The key step therefore consists in proving that the correlation decay property holds. It is the object of our next proposition.

First, we introduce for any arbitrary induced subgraph $\gH$ of ${\Graph(\epsilon)}$, and any node $i$ in $\gH$, introduce $M_{\gH}(i)=\E[\exp(-C_{\gH}(i))],M^-_{\gH}(i,r)=\E[\exp(-C^-_{\gH}(i,r))],M^+_{\gH}(i,r)=\E[\exp(-C^+_{\gH}(i,r))]$.

\begin{Proposition}\label{prop:BonusBounded}
Let ${\Graph(\epsilon)}=(V_\epsilon,E_\epsilon)$ be the graph obtained from the original underlying graph as a result of the first phase of the algorithm (namely deleting every node with probability $\delta=\epsilon^2/16$ independently for all nodes).  Then, for every node $i$ in ${\Graph(\epsilon)}$ and every $r$
\begin{align}\label{eq:missclassify1}
\pr(C_{{\Graph(\epsilon)}}(i)=0,C^+_{{\Graph(\epsilon)}}(i,2t)>0)\le  3(1-\epsilon^2/16)^{2r},
\end{align}
and
\begin{align}\label{eq:missclassify2}
\pr(C_{{\Graph(\epsilon)}}(i)>0,C^-_{{\Graph(\epsilon)}}(i,2t)=0)\le  3(1-\epsilon^2/16)^{2r}.
\end{align}
\end{Proposition}

\begin{proof}
Consider a subgraph $\gH$ of $\Graph$, node $i\in \gH$ with neighbors $\mathcal{N}_\gH(i)=\{i_1,\ldots,i_d\}$, and suppose for now that the number of neighbors of $i$ in $\Graph$ is less than $2$.

Examine the recursion (\ref{eq:cavityrecursion}) and observe that all the randomness in terms $C_{\gH\setminus\{i,i_1,\ldots,i_{l-1}\}}(i_l)$ comes from the subgraph $\gH\setminus\{i,i_1,\ldots,i_{l-1}\}$, and thus $W_j$ is independent from the vector \\$(C_{\gH\setminus\{i,i_1,\ldots,i_{l-1}\}}(i_l),  1\le l\le d)$.
A similar assertion applies when we replace $C_{\gH\setminus\{i,i_1,\ldots,i_{l-1}\}}(i_l)$ with \\$C^-_{\gH\setminus\{i,i_1,\ldots,i_{l-1}\}}(i_l,r)$ and $C^+_{\gH\setminus\{i,i_1,\ldots,i_{l-1}\}}(i_l,r)$ for every $r$. Using the memoryless property of the exponential distribution, denoting $W$ a standard exponential random variable, we obtain:
\begin{align}\label{eq:representation}
\E[\exp(-C_{\gH}(i))|\sum_{1\le l\le d}C_{\gH\setminus\{i,i_1,\ldots,i_{l-1}\}}(i_l)=x]
=&\pr(W_i\le x) E[\exp(0)]+\notag \\
&\vspace{-1cm}\E[\exp(-(W_i-x))\,|\,W_i>x]\pr(W_i>x)\notag \\
=&(1-\pr(W_i> x))+\E[\exp(-W)]\pr(W_i>x) \notag \\
=&(1-\pr(W_i>x))+(1/2)\pr(W_i>x) \notag\\
=&1-(1/2)\pr(W_i>x) \notag\\
=&1-(1/2)\exp(-x)
\end{align}
It follows
\begin{align*}
\E[\exp(-C_{\gH}(i))]=1-(1/2)\E\exp \left(-\sum_{1\le l\le d}C_{\gH\setminus\{i,i_1,\ldots,i_{l-1}\}}(i_l) \right).
\end{align*}
Similarly we obtain
\begin{align*}
\E[\exp(-C^-_{\gH}(i,r))]&=1-(1/2)\E\exp(-\sum_{1\le l\le d}C^-_{\gH\setminus\{i,i_1,\ldots,i_{l-1}\}}(i_l,r-1)\Big), \\
\E[\exp(-C^+_{\gH}(i,r))]&=1-(1/2)\E\exp(-\sum_{1\le l\le d}C^+_{\gH\setminus\{i,i_1,\ldots,i_{l-1}\}}(i_l,r-1)\Big).
\end{align*}

Since $i$ had two neighbors or less in $\Graph$, it also has two neighbors or  less in $\gH$.
For $d=0$, we have trivially $M_{\gH}(i)=M^-_{\gH}(i)=M^+_{\gH}(i)$. Suppose $d=1: N_{\gH}(i)=\{i_1\}$. Then,
\begin{align}
M^-_{\gH}(i,r)-M^+_{\gH}(i,r)&=(1/2)\Big(\E[\exp(-C^+_{\gH\setminus\{i\}}(i_1,r-1))]-\E[\exp(-C^-_{\gH\setminus\{i\}}(i_1,r-1))]\Big) \notag \\
&=(1/2)\Big(M^+_{\gH\setminus\{i\}}(i_1,r-1)-M^-_{\gH\setminus\{i\}}(i_1,r-1)\Big)\label{degone}
\end{align}
Finally, suppose $d=2$: $N(i)=\{i_1,i_2\}$. Then
\begin{align*}
M^-_{\gH}(i,r)&-M^+_{\gH}(i,r) \\
&=(1/2)\E[\exp(-C^+_{\gH\setminus\{i\}}(i_1,r-1)-C^+_{\gH\setminus\{i,i_1\}}(i_2,r-1))] \\
&-(1/2)\E[\exp(-C^-_{\gH\setminus\{i\}}(i_1,r-1)-C^-_{\gH\setminus\{i,i_1\}}(i_2,r-1))] \\
&=(1/2)\E[\exp(-C^+_{\gH\setminus\{i\}}(i_1,r-1))(\exp(-C^+_{\gH\setminus\{i,i_1\}}(i_2,r-1))
-\exp(-C^-_{\gH\setminus\{i,i_1\}}(i_2,r-1))]\\
&+(1/2)\E[\exp(-C^-_{\gH\setminus\{i,i_1\}}(i_2,r-1))(\exp(-C^+_{\gH\setminus\{i\}}(i_1,r-1))
-\exp(-C^-_{\gH\setminus\{i\}}(i_1,r-1))]
\end{align*}
Using the non-negativity of $C^-,C^+$ and applying Lemma~\ref{lemma:Bbounds} we obtain for odd  $r$
\begin{align}
0\le M^+_{\gH}(i,r)-M^-_{\gH}(i,r)&\le (1/2)\E[\exp(-C^-_{\gH\setminus\{i,i_1\}}(i_2,r-1))
-\exp(-C^+_{\gH\setminus\{i,i_1\}}(i_2,r-1))]\notag \\
&+(1/2)\E[\exp(-C^-_{\gH\setminus\{i\}}(i_1,r-1))
-\exp(-C^+_{\gH\setminus\{i\}}(i_1,r-1))], \notag \\
&=(1/2)\big(M^-_{\gH\setminus\{i,i_1\}}(i_2,r-1)-M^+_{\gH\setminus\{i,i_1\}}(i_2,r-1)\big)\notag \\
&+(1/2)\big(M^-_{\gH\setminus\{i\}}(i_1,r-1)-M^+_{\gH\setminus\{i\}}(i_1,r-1)\big)\label{degtwoeven}
\end{align}
and for even $r$
\begin{align}
0\le M^-_{\gH}(i,r)-M^+_{\gH}(i,r)
&\le(1/2)\big(M^+_{\gH\setminus\{i,i_1\}}(i_2,r-1)-M^-_{\gH\setminus\{i,i_1\}}(i_2,r-1)\big) \notag \\
&+(1/2)\big(M^+_{\gH\setminus\{i\}}(i_1,r-1)]-M^-_{\gH\setminus\{i\}}(i_1,r-1)\big)\label{degtwoodd}
\end{align}
Summarizing the three cases we conclude
\begin{align}\label{eq:withouteps}
|M^+_{\gH}(i,r)-M^-_{\gH}(i,r)|\le
(d/2)\max_{\gH',j}\Big|M^+_{\gH'}(j,r-1)-M^-_{\gH'}(j,r-1)\Big|,
\end{align}
where the maximum is over subgraphs $\gH'$ of $\Graph$ and nodes $j\in \gH'$ with degree at most $2$ in $\gH'$. The reason for this is that in equations \eqref{degone},\eqref{degtwoeven}, and \eqref{degtwoodd}; the moments $M^+_{\gH'}(j,r-1)$ in the right hand side are always computed in a node $j$ which has lost at least one of its neighbors (namely, $i$) in graph $\gH$. Since the degree of $j$ was at most $3$ in $\Graph$ and one neighbor at least is removed, $j$ has at most two neighbors in $\gH$.
By considering $\gH \cap {\Graph(\epsilon)}$ in all previous equations, equation \eqref{eq:withouteps} implies
\begin{align}\label{eq:witheps}
|M^+_{\gH \cap  {\Graph(\epsilon)}}(i,r)-M^-_{\gH\cap  {\Graph(\epsilon)}}(i,r)|\le
(d(\epsilon)/2)\max_{\gH',j}\Big|M^+_{\gH'\cap  {\Graph(\epsilon)}}(j,r-1)-M^-_{\gH'\cap  {\Graph(\epsilon)}}(j,r-1)\Big|,
\end{align}
where $d(\epsilon)$ denotes the number of neighbors of $i$ in $\gH \cap  {\Graph(\epsilon)}$. By definition of $\Graph(\epsilon)$, $d(\epsilon)$ is a binomial random variable with $d$ trials and probability of success $(1-\epsilon^2/16)$, where $d$ is the degree of $i$ in $\gH$. Since $d\leq 2$, $E[d(\epsilon)]\leq 2(1-\epsilon^2/16)$. Moreover, this randomness is independent from the randomness of the random weights of $\gH$. Therefore,
\begin{align}\label{eq:ineqeps}
\E[|M^+_{\gH \cap  {\Graph(\epsilon)}}(i,r)-M^-_{\gH \cap  {\Graph(\epsilon)}}(i,r)|]\le
(1-\epsilon^2/16)\max_{\gH,j}\E\Big|M^+_{\gH \cap  {\Graph(\epsilon)}}(j,r-1)-M^-_{\gH \cap  {\Graph(\epsilon)}}(j,r-1)\Big|
\end{align}
where the external expectation is w.r.t. the randomness of the first phase of the algorithm (deleted nodes). Let $e_{r-1}$ denote the right-hand side of \eqref{eq:ineqeps}. By taking the max of the left-hand side of \eqref{eq:ineqeps} over all $(\gH,j)$ where $j$ has degree less than or equal to $2$ in $\gH$, we obtain the inequality $e_r\leq (1-\epsilon^2/16) e_{r-1}$. Iterating on $r$ and using $0\le M\le 1$, this implies that $e_r \leq (1-\epsilon^2/16)^r$ for all $r \ge 0$. Finally, it is easy to show using the same techniques that equation \eqref{eq:withouteps} holds for $r=3$ as well. This finally implies that for an arbitrary node $i$ in $\Graph(\epsilon)$
\begin{align*}
\E[|M^+_{{\Graph(\epsilon)}}(i,r)-M^-_{{\Graph(\epsilon)}}(i,r)|]\le 3/2 (1-\epsilon^2/16)^r.
\end{align*}
Applying Lemma~\ref{lemma:Bbounds}, we conclude for every $r$
\begin{align*}
0&\le \E[\exp(-C^-_{{\Graph(\epsilon)}}(i,2r))-\exp(-C^+_{{\Graph(\epsilon)}}(i,2r))]\le 3/2(1-\epsilon^2/16)^{2r}.
\end{align*}
Recalling (\ref{eq:representation}) we have
\begin{align*}
\E[\exp(-C_{{\Graph(\epsilon)}}(i))]=1-(1/2)\pr(W>\sum_{1\le l\le d}C_{{\Graph(\epsilon)}\setminus\{i,i_1,\ldots,i_{l-1}\}}(i_l))
=1-(1/2)\pr(C_{{\Graph(\epsilon)}}(i)>0).
\end{align*}
Similar expressions are valid for $C^-_{{\Graph(\epsilon)}}(i,r)), C^+_{{\Graph(\epsilon)}}(i,r))$. We obtain
\begin{align*}
0&\le \pr(C^+_{{\Graph(\epsilon)}}(i,2r)=0)-\pr(C^-_{{\Graph(\epsilon)}}(i,2r)=0)\le 3(1-\epsilon^2/16)^{2r}.
\end{align*}
Again applying Lemma~\ref{lemma:Bbounds}, we obtain
\begin{align*}   
\pr(C_{{\Graph(\epsilon)}}(i)=0,C^+_{{\Graph(\epsilon)}}(i,2r)>0)\le \pr(C^-_{{\Graph(\epsilon)}}(i,2r)=0,C^+_{{\Graph(\epsilon)}}(i,2r)>0)\le 3(1-\epsilon^2/16)^{2r},
\end{align*}
and
\begin{align*}  
\pr(C_{{\Graph(\epsilon)}}(i)>0,C^-_{{\Graph(\epsilon)}}(i,2r)=0)\le \pr(C^-_{{\Graph(\epsilon)}}(i,2r)=0,C^+_{{\Graph(\epsilon)}}(i,2r)>0)\le 3(1-\epsilon^2/16)^{2r}.
\end{align*}
This completes the proof of the proposition.
\end{proof}

\subsubsection{Concentration argument}\label{subsubsec:concentration}

We can now complete the proof of Theorem~\ref{theorem:ISMainResult}. We need to bound $|W(I^*)-W(I^*_\epsilon)|$ and $W(I^*_\epsilon\setminus \I(r,\epsilon))$ and show that both quantities are small.

Let $\Delta V_\epsilon$ be the set of nodes in $\Graph$ which are not in ${\Graph(\epsilon)}$. Trivially, $|W(I^*)-W(I^*_\epsilon)|\le W(\Delta V_\epsilon)$. We have $\E[\Delta V_\epsilon]=\epsilon^2/16 n$, and since the nodes were deleted irrespectively of their weights, then $\E[W(\Delta V_\epsilon)]=\epsilon^2/16 n$.

To analyze $W(I^*_\epsilon\setminus \I(r,\epsilon))$, observe that by (second part of) Proposition~\ref{prop:BonusBounded}, for every node $i, \pr(i\in I^*_\epsilon\setminus \I(r,\epsilon))\le 3(1-\epsilon^2/16)^r\eqdef \delta_1$.  Thus $\E|I^*_\epsilon\setminus \I(r,\epsilon)|\le \delta_1n$. In order to obtain a bound on $W(I^*_\epsilon\setminus \I(r,\epsilon))$ we derive a crude bound on the largest weight of a subset with cardinality $\delta_1n$. Fix a constant $C$ and consider the set $V_C$ of all nodes in ${\Graph(\epsilon)}$ with weights greater than $C$. We have $\E[W(V_C)]\le (C+E[W-C|W>C])\exp(-C)n=(C+1)\exp(-C)n$. The remaining nodes have a weight at most $C$. Therefore,
\begin{align*}
\E[W(I^*_\epsilon\setminus \I(r,\epsilon))]&\le \E[W\Big(\big((I^*_\epsilon\setminus \I(r,\epsilon)\big)\cap V_C) \cup V_C^c\Big)] \le C \E[|I^*_\epsilon\setminus \I(r,\epsilon)|] +\E[V_C]\\
&\le C\delta_1n+(C+1)\exp(-C)n.
\end{align*}
We conclude
\begin{align}\label{eq:expWIstarI}
\E[|W(I^*)-W(\I(r,\epsilon))|]\le \epsilon^2/16 n+C\delta_1n+(C+1)\exp(-C)n.
\end{align}
Now we obtain a lower bound on $W(I^*)$. Consider the standard greedy algorithm for generating an independent set: take arbitrary node, remove neighbors, and repeat. It is well known and simple to see that this algorithm produces an independent set with cardinality at least $n/4$, since the largest degree is at most 3. Since the algorithm ignores the weights, then  also the expected weight of this set is at least $n/4$. The variance of that weight is upper bounded by $n$. By Chebyshev's inequality
\begin{align*}
\pr(W(I^*)<n/8)\le {n\over (n/4-n/8)^2}=64/n.
\end{align*}
We now summarize the results.
\begin{align*}
\pr({W(\I(r,\epsilon))\over W(I^*)}\le 1-\epsilon)&\le \pr({W(\I(r,\epsilon))\over W(I^*)}\le 1-\epsilon,W(I^*)\ge n/8)
+\pr(W(I^*)< n/8) \\
&\le \pr({|W(I^*)-W(\I(r,\epsilon))|\over W(I^*)}\ge \epsilon,W(I^*)\ge n/8)+64/n \\
&\le \pr({|W(I^*)-W(\I(r,\epsilon))|\over n/8}\ge \epsilon)
+64/n \\
&\le {\epsilon^2/16 +4C(1-\epsilon^2/16)^r+(C+1)\exp(-C)\over \epsilon/8}+64/n,
\end{align*}
where we have used Markov's inequality in the last step and $\delta_1=4(1-\epsilon^2/16)^r$. Thus it suffices to arrange $C$ so that the first ratio is at most $2\epsilon/3$ and assuming, without the loss of generality, that $n\ge 192/\epsilon$, we will obtain that the sum is at most $\epsilon$. It is a simple exercise to show that by taking $r=O(\log(1/\epsilon)/\epsilon^2)$ and $C=O(\log(1/\epsilon))$, we obtain the required result. This completes the proof of Theorem~\ref{theorem:ISMainResult}. \qed

\subsection{Generalization to higher degrees. Proof of Theorem~\ref{theorem:ISMainResult2}}
In this section we present the proof of Theorem~\ref{theorem:ISMainResult2}. The mixture of $\Delta$ exponential distributions with rates $\alpha_j, 1\le j\le \Delta$ and equal weights $1/\Delta$
can be viewed as first randomly generating a rate $\alpha$ with the probability law $\pr(\alpha=\alpha_j)=1/\Delta$ and then
randomly generating exponentially distributed random variable with rate $\alpha_j$, conditional on the rate being $\alpha_j$.

For every subgraph $\gH$ of $\Graph$, node $i$ in $\gH$ and $j=1,\ldots,\Delta$, define
$M_{\gH}^j(i)=\E[\exp(-\alpha_j\:  C_{\gH}(i))]$, $M_{\gH}^{-,j}(i,r)=\E[\exp(-\alpha_j \: C_{\gH}^-(i,r))]$ and $M_{\gH}^{+,j}(i,r)=\E[\exp(-\alpha_j\:C_{\gH}^+(i,r))]$, where $C_{\gH}(i)),C_{\gH}^+(i,r))$ and $C_{\gH}^-(i,r))$
are defined as in Section~\ref{section:algorithm}.

\begin{Lemma}
Fix any subgraph $\gH$, node $i\in\gH$ with $N_{\gH}(i)=\{i_1,\ldots,i_d\}$. Then
\begin{eqnarray*}
\E[\exp(-\alpha_j C_{\gH}(i))]&=&1-\frac{1}{\Delta} \sum_{1\le k\le m} \frac{\alpha_j}{\alpha_j+\alpha_k} \E[\exp(- \sum_{1\le l\le d} \alpha_k C_{\gH\setminus\{i,i_1,\ldots,i_{l-1}\}}(i_l))]\\
\E[\exp(-\alpha_j C_{\gH}^+(i,r))]&=&1-\frac{1}{\Delta} \sum_{1\le k\le m} \frac{\alpha_j}{\alpha_j+\alpha_k} \E[\exp(- \sum_{1\le l\le d} \alpha_k C_{\gH\setminus\{i,i_1,\ldots,i_{l-1}\}}^+(i_l,r-1))]\\
\E[\exp(-\alpha_j C_{\gH}^-(i,r))]&=&1-\frac{1}{\Delta} \sum_{1\le k\le m} \frac{\alpha_j}{\alpha_j+\alpha_k} \E[\exp(- \sum_{1\le l\le d} \alpha_k C_{\gH\setminus\{i,i_1,\ldots,i_{l-1}\}}^-(i_l,r-1))]
\end{eqnarray*}
\end{Lemma}

\begin{proof}
Let $\alpha(i)$ be the random rate associated with node $i$. Namely, $\pr(\alpha(i)=\alpha_j)=1/\Delta$.
We condition on the event $\sum_{1\le l\le d}C_{\gH\setminus\{i,i_1,\ldots,i_{l-1}\}}(i_l)=x$. As $C_{\gH}(i)=\max(0,W_i-x)$, we obtain:
\begin{eqnarray*}\
\E[-\alpha_j C_{\gH}(i)|x]&=&{1\over \Delta}\sum_k  \E[-\alpha_j C_{\gH}(i)|x,\alpha(i)=\alpha_k]\\
&=& {1\over \Delta}\sum_k \Big(\pr(W_i\leq x| \alpha(i)=\alpha_k)\\
&\indent &+\pr(W_i> x| \alpha(i)=\alpha_k)\E[\exp(-\alpha_j(W_i-x))|W_i>x,\alpha(i)=\alpha_k]\Big)\\
&=&{1\over \Delta}\sum_k \bigg(1-\exp(-\alpha_k x)+\exp(-\alpha_k x)\frac{\alpha_k}{\alpha_j+\alpha_k}\bigg)\\
&=&1-{1\over \Delta}\sum_k\frac{\alpha_j}{\alpha_j+\alpha_k} \exp(-\alpha_k x)
\end{eqnarray*}
Thus,
\begin{eqnarray*}
\E[-\alpha_j C_{\gH}(i)]&=&1-{1\over \Delta}\sum_k \frac{\alpha_j}{\alpha_j+\alpha_k} \E[\exp(- \sum_{1\le l\le d} \alpha_k C_{\gH\setminus\{i,i_1,\ldots,i_{l-1}\}}(i_l))]
\end{eqnarray*}
The other equalities follow identically.
\end{proof}
By taking differences, we obtain
\begin{eqnarray*}
&&M_{\gH}^{-,j}(i,r)-M_{\gH}^{+,j}(i,r)=\\ &&{1\over \Delta}\sum_k \frac{\alpha_j}{\alpha_j+\alpha_k}\Bigg(\E[\prod_{1\le l\le d} \exp(- \alpha_k C_{\gH\setminus\{i,i_1,\ldots,i_{l-1}\}}^+(i_l,r-1))] -\E[\prod_{1\le l \le d}\exp(-\alpha_k C_{\gH\setminus\{i,i_1,\ldots,i_{l-1}\}}^-(i_l,r-1))]\Bigg)
\end{eqnarray*}
We now use the identity
\begin{align*}
\prod_{1\leq l \leq r} x_l - \prod_{1 \leq l \leq r} y_l=
 \sum_{1\le l \le r} \Big( \big(\prod_{1\leq k \leq l-1} x_k \big) (x_l-y_l) \big(\prod_{l+1\leq k \leq r} y_k\big) \Big),
\end{align*}
which further implies
\begin{align*}
\Big|\prod_{1\leq l \leq r} x_l - \prod_{1 \leq l \leq r} y_l\Big|\le
\sum_{1\leq l\leq r}|x_l-y_l|,
\end{align*}
when $\max_l |x_l|,|y_l|<1$.
By applying this inequality with $x_l=\exp(-\alpha_k C_{\gH\setminus\{i,i_1,\ldots,i_{l-1}\}}^+(i_l,r-1))$  and $y_l=\exp(-\alpha_k C_{\gH\setminus\{i,i_1,\ldots,i_{l-1}\}}^-(i_l,r-1))$, we obtain
\begin{eqnarray*}
|M_{\gH}^{-,j}(i,r)-M_{\gH}^{+,j}(i,r)|\leq {1\over \Delta}\sum_{1 \le k \le m} \frac{\alpha_j}{\alpha_j+\alpha_k} \sum_{1\le l \le d} \big|M_{\gH\setminus\{i,i_1,\ldots,i_{l-1}\}}^{-,k}(i_l,r-1)-M_{\gH\setminus\{i,i_1,\ldots,i_{l-1}\}}^{+,k}(i_l,r-1) \big|.
\end{eqnarray*}
This implies
\begin{eqnarray}\label{eq:general_1}
|M_{\gH}^{-,j}(i,r)-M_{\gH}^{+,j}(i,r)|\leq  {r\over \Delta}\sum_{1 \le k \le m} \frac{\alpha_j}{\alpha_j+\alpha_k} \max_{1\le l \le d} \big|M_{\gH\setminus\{i,i_1,\ldots,i_{l-1}\}}^{-,k}(i_l,r-1)-M_{\gH\setminus\{i,i_1,\ldots,i_{l-1}\}}^{+,k}(i_l,r-1) \big|.
\end{eqnarray}
For any $t\geq0$ and $j$, define $e_{r,j}$ as follows
\begin{equation}
e_{r,j}=\sup_{\gH \subset \Graph,i \in \gH} |M_{\gH}^{-,j}(i,r)-M_{\gH}^{+,j}(i,r)|
\end{equation}
By taking maximum on the right and left hand side successively, inequality \eqref{eq:general_1} implies
\begin{equation*}
e_{r,j}\leq {r\over \Delta}\sum_{1 \le k \le m} \frac{\alpha_j}{\alpha_j+\alpha_k} e_{r-1,k}
\end{equation*}
For any $t\geq 0$, denote $\bold{e_r}$ the vector of $(e_{r,1},\ldots,e_{r,m})$. Denote $\bold M$ the matrix such that for all $(j,k)$, $M_{j,k}=
{r\over \Delta}\:\frac{\alpha_j}{\alpha_j+\alpha_k}$.
We finally obtain
\begin{equation*}
\bold{e_r}\leq M \bold{e_{r-1}}.
\end{equation*}
Therefore, if $M^r$ converges to zero exponentially fast in each coordinate, then also $\bold{e_r}$ converges exponentially fast to $0$. Following the same steps as the proof of theorem \ref{theorem:ISMainResult}, this will imply that for each node, the error of a decision made in $\I(r,0)$ is exponentially small in $r$ .
Note that  $\frac{r}{\Delta}\leq 1$. Recall that $\alpha_j=\rho^j$. Therefore, for each $j,k$, we have $M_{j,k}\leq \frac{\rho^j}{\rho^j+\rho^k}$. Define $M_\Delta$ to be a $\Delta\times\Delta$ matrix defined by  $M_{j,j}=1/2,M_{j,k}=1,j>k$ and $M_{j,k}=(1/\rho)^{k-j}, k>j$, for all $1\le j,k\le \Delta$.
Since $M\le M_\Delta$, it suffices to show that $M_\Delta^r$ converges to zero exponentially fast.
Proof of theorem \ref{theorem:ISMainResult2} will thus be completed with the proof of the following lemma:
\begin{Lemma}
Under the condition $\rho>25$, there exists $\delta=\delta(\rho)<1$ such that the absolute value of every entry of $M_\Delta^r$ is
at most $\delta^r(\rho)$.
\end{Lemma}

\begin{proof}
Let $\epsilon=1/\rho$. Since elements of $M$ are non-negative, it suffices to exhibit a strictly positive vector $x=x(\rho)$
and $0<\theta=\theta(\rho)<1$ such that $M'x\le \theta x$, where $M'$ is transpose of $M$.
Let $x$ be the vector defined by $x_k=\epsilon^{k/2}, 1\leq k \leq \Delta$. We show that for any $j$,
$$(M'x)_j\leq (1/2+2 \frac{\sqrt{\epsilon}}{1-\sqrt{\epsilon}}) x_j$$
It is easy to verify that when $\rho>25$, that is $\epsilon<1/25$, $(1/2+2\frac{\sqrt{\epsilon}}{1-\sqrt{\epsilon}})<1$,
and the proof would be complete.
Fix $1\leq j \leq \Delta$.
Then,
\begin{eqnarray*}
(M'x)_j&=&\sum_{1\leq k \leq j-1} M_{k,j}\: x_k +1/2 x_j +\sum_{j+1\leq k \leq \Delta} M_{k,j} \: x_k\\
&=& \sum_{1 \leq k \leq j-1}\epsilon^{j-k}\epsilon^{k/2} + 1/2 \epsilon^{j/2}+\sum_{j+1\leq k \leq \Delta} \epsilon^{k/2}
\end{eqnarray*}
Since $x_j=\epsilon^{j/2}$, we have
\begin{eqnarray*}
\frac{(M'x)_j}{x_j}&\leq& \sum_{1 \leq k \leq j-1}\epsilon^{(j-k)/2} + 1/2 +\sum_{j+1\leq k \leq \Delta} \epsilon^{(k-j)/2}\\
&=&1/2+\sum_{1\leq k \leq j-1} \epsilon^{k/2}+\sum_{1 \leq k\leq \Delta-j}\epsilon^{k/2}\leq 1/2+ \frac{2\epsilon^{1/2}}{1-\epsilon^{1/2}}
\end{eqnarray*}
This completes the proof of the lemma and of the theorem.
\end{proof}

\subsection{Hardness result. Proof of Theorem~\ref{theorem:ISMainResult3}}
The main idea of the proof is to show that the weight of a maximum weighted independent
set is close to the cardinality of a maximum independent set.  A similar proof idea was used
in~\cite{LubyVigoda} for proving the hardness of approximately counting independent sets in sparse graphs.  

Given a graph ${\mathbb G}$ with degree bounded by $\Delta$, let  $I^M$ denote (any) maximum cardinality independent set, and let $I^*$ denote the unique maximum weight independent set corresponding to i.i.d. weights with $\exp(1)$ distribution. We make use of the following result due to Trevisan~\cite{trevisan2001nar}.

\begin{Theorem}\label{theorem:Trevisan}
There exist $\Delta_0$ and $c^*$ such that for all $\Delta\ge \Delta_0$ the problem of approximating the largest independent set in graphs with degree at most $\Delta$ to within a factor $\rho=\Delta/2^{c^*\sqrt{\log\Delta}}$ is NP-complete.
\end{Theorem}
Our main technical result is the following proposition. It states that the ratio of the expected weight of a maximum weight independent set to the cardinality of a maximum independent set grows as the logarithm of the maximum degree of the graph.
\begin{Proposition}
Suppose $\Delta \ge 2$. For every graph ${\mathbb G}$ with maximum degree $\Delta$ and $n$ large enough, we have:
\begin{align*}
1 \leq \frac{E[ W(I^*) ]}{|I^M|} \leq 10 \log{\Delta}.
\end{align*}
\end{Proposition}
This in combination with Theorem~\ref{theorem:Trevisan} leads to the desired result.
\begin{proof}
Let $W(1)<W(2)<\cdots<W(n)$ be the ordered weights associated with our graph ${\mathbb G}$. 
Observe that
\begin{eqnarray*}
E[ W (I^*) ] &=& E[ \sum_{v \in I^*} W_v ] \\
&\leq& E[ \sum_{n-|I^*|+1}^n W(i) ]\\
&\leq& E[ \sum_{n-|I^M|+1}^n W(i) ].\\
\end{eqnarray*}

The exponential distribution implies
$E [W(j)]=H(n)-H(n-j)$, where $H(k)$ is the harmonic sum $\sum_{1\le i \le k} 1/i$.  Thus
\begin{align*}
\sum_{j=n-|I^M|+1}^{n}E [W(j)]&=\sum_{n-|I^M|+1\le j\le n}(H(n)-H(n-j)) \\
&=|I^M|H(n)-\sum_{j\le |I^M|-1}H(j).
\end{align*}
We use the bound $\log(k) \le H(k)-\gamma \le \log(k) +1$, where $\gamma \approx .57$ is Euler's constant.  Then
\begin{eqnarray*}
\sum_{j=n-|I^M|+1}^{n}E [W(j)] &\leq& |I^M|(H(n)-\gamma) + \log(|I^M|) + 2 -\sum_{1\le j\le |I^M|} \log(j)\\
&\leq& |I^M|(H(n)-\gamma) + \log(|I^M|) + 2 -\int_{1}^{|I^M|}\log(t) dt\\
&\leq& |I^M|\log(n) + |I^M| + \log(|I^M|) + 2 - |I^M|\log(|I^M|)+|I^M|\\
&\leq& (|I^M| + 1)(\log{n\over |I^M|}+2+\log(|I^M|)/|I^M|) \\
&\leq& |I^M|(\log(\Delta+1)+3) + (\log(\Delta+1)+3),
\end{eqnarray*}
where the bound $|I^M|\ge n/(\Delta+1)$ (obtained by using the greedy algorithm, see Section \ref{subsubsec:concentration}) is used.
Again using the bound $|I^M|\ge n/(\Delta+1)$, we find that $\frac{E[ W( I^*) ]}{|I^M|} \leq \log(\Delta+1)+3+o(1)$.
Since $E[ W(I^*) ] \geq E[ W(I^M) ] = |I^M|$, it follows that for all sufficiently large $n$,
$1 \leq \frac{E[ W( I^*) ]}{|I^M|} \leq \log(\Delta+1) + 4$.  The proposition follows since for all $\Delta \geq 2$ we have\ 
$\log(\Delta+1) + 4 \leq 10 \log{\Delta}$.
\end{proof}

\section{Conclusion}\label{section:ccl}
We considered an optimization model which encompasses many models from the literature
including graphical models, combinatorial optimization and economics.
In our model, cooperating agents within a networked structure choose decisions from a finite set of actions and seek to collectively optimize a
global welfare objective function, which can be additively decomposed on the nodes and edges of the network.
The main goal is to answer whether it's possible to find near optimal solutions efficiently, and if possible using distributed algorithms
relying only on local information. Despite the apparent NP-hardness of such a problem even in the approximation setting, we find that
in a framework where cost functions are random, this goal is often achievable. Specifically, we have constructed a general purpose
algorithm Cavity Expansion, which relies on the local information only, and thus is distributed. We have established that under the
so-called correlation decay property, our algorithm finds a near optimal solution with high probability. We have identified a variety
of models which exhibit the correlation decay property and we have proposed general purpose techniques, such as the coupling technique,
which we used to prove the correlation decay property.

Our results highlight interesting and intriguing connections between the fields of complexity of algorithms for combinatorial
optimization problems and statistical physics, specifically the cavity method and the issues of  long-range independence.
For example  in the special case of the MWIS problem we showed that the problem admits a PTAS, provided by the CE algorithm, for certain
node weight distribution, even though the maximum cardinality version of the same problem is known to be non-approximable unless P=NP.

It would be interesting to see  what weight distribution are amenable to the approach proposed in this paper. For example, one could consider
the case of Bernoulli weights and see whether the correlation decay property breaks down precisely when the approximation becomes NP-hard.
Furthermore, it would be interesting to see if the random weights assumption for general decision networks
can be substituted with deterministic weights which have some random like properties, in a fashion similar to the study of pseudo-random graphs. This would move our approach even closer to the worst-case combinatorial optimization setting.

The framework studied here can be further extended in several additional ways. First, we can consider a network of agents who, instead of cooperating, behave selfishly. Using ideas similar to those presented in this paper, we believe it is possible to identify settings where using a distributed procedures representing communication between the agents, one can find in polynomial time Nash equilibrium of the underlying system. Second, one can consider a dynamical setting where agents take repeated actions that affect both their reward and their future state. This class of models, known as factored Markov Decision Processes, has a very large number of applications (supply chain, communication networks, and many others), but optimality bounds have been identified only in very restricted settings. Again, concepts such as correlation decay  may be found useful to approach these problems and identify new settings where the solution can be found in polynomial time,  despite the curse of dimensionality typically exhibited by these models.

\bibliographystyle{amsalpha}
\bibliography{statphy}
\end{document}